\documentclass[10pt]{amsart}

\usepackage[T1]{fontenc}
\usepackage{lmodern}
\usepackage{microtype}
\usepackage{amsmath,amssymb,amsthm,mathtools}
\usepackage{booktabs}
\usepackage{tabularx}
\usepackage{graphicx}
\usepackage{enumitem}
\usepackage{tikz}
\usetikzlibrary{arrows.meta,calc,positioning}
\usepackage[
  colorlinks=true,
  citecolor=blue,
  linkcolor=blue,
  urlcolor=blue,
  pdftitle={Sharp Order Laws and Finite-Defect Stability for Domination versus Small Laplacian Eigenvalues in Trees},
  pdfauthor={Yufeng Wang},
  pdfsubject={Extremal and spectral graph theory for subcubic trees},
  pdfkeywords={domination number, Laplacian eigenvalues, subcubic tree, extremal graph theory, finite defect kernel}
]{hyperref}
\usepackage[nameinlink,capitalise]{cleveref}

\newtheorem{theorem}{Theorem}[section]
\newtheorem{lemma}[theorem]{Lemma}
\newtheorem{proposition}[theorem]{Proposition}
\newtheorem{corollary}[theorem]{Corollary}
\theoremstyle{definition}
\newtheorem{definition}[theorem]{Definition}

\theoremstyle{remark}
\newtheorem{remark}[theorem]{Remark}

\newcommand{\In}{\operatorname{In}}
\newcommand{\diag}{\operatorname{diag}}

\newcommand{\trans}{\mathsf{T}}

\tikzset{
  p2hub/.style={circle,draw=black,fill=black,minimum size=5pt,inner sep=0pt},
  p2support/.style={circle,draw=black,fill=white,minimum size=5pt,inner sep=0pt},
  p2deep/.style={circle,draw=blue!70!black,fill=blue!18,
    minimum size=5pt,inner sep=0pt},
  p2leaf/.style={rectangle,draw=black,fill=black,minimum size=3.5pt,inner sep=0pt},
  p2edge/.style={draw=black!60,line width=.8pt},
  p2new/.style={draw=red!75!black,line width=1.25pt},
  p2mark/.style={circle,draw=red!75!black,fill=white,line width=1.05pt,
    minimum size=6.2pt,inner sep=0pt},
  p2flow/.style={-{Latex[length=2mm]},line width=.7pt},
}

\newcommand{\PtwoMarkedSmithGraphic}{%
\begin{tikzpicture}[font=\small,baseline,
  smithpanel/.style={draw=black!18,rounded corners=3pt,fill=black!1}]
  \path[use as bounding box] (-.05,-5.42) rectangle (12.45,1.52);
  \draw[smithpanel] (.10,-1.66) rectangle (6.05,1.42);
  \draw[smithpanel] (6.35,-1.66) rectangle (12.30,1.42);
  \draw[smithpanel] (.10,-5.32) rectangle (6.05,-1.94);
  \draw[smithpanel] (6.35,-5.32) rectangle (12.30,-1.94);

  \begin{scope}[xshift=.20cm,yshift=.08cm]
    \node[font=\small] at (2.85,1.05) {(a) $A_3$};
    \node[p2mark] (ax) at (2.85,.38) {};
    \node[p2leaf] (al1) at (2.15,.75) {};
    \node[p2leaf] (al2) at (2.15,.01) {};
    \draw[p2new] (ax)--(al1) (ax)--(al2);
    \draw[p2flow] (2.85,.05)--(2.85,-.32);
    \node[p2deep,label={[font=\small]right:$2$}] at (2.85,-.62) {};
    \node at (2.85,-1.03) {$\delta=0,\ |C|=1$};
    \node[font=\footnotesize,text=blue!70!black] at (2.85,-1.52) {PASS};
  \end{scope}

  \begin{scope}[xshift=6.45cm,yshift=.08cm]
    \node[font=\small] at (2.85,1.05) {(b) $D_n$};
    \node[p2mark] (bx) at (2.05,.38) {};
    \node[p2leaf] (bl1) at (1.45,.75) {};
    \node[p2leaf] (bl2) at (1.45,.01) {};
    \node[p2deep] (bp1) at (2.82,.38) {};
    \node at (3.45,.38) {$\cdots$};
    \node[p2deep] (bp2) at (4.18,.38) {};
    \draw[p2new] (bx)--(bl1) (bx)--(bl2);
    \draw[p2edge] (bx)--(bp1) (bp1)--(3.24,.38)
      (3.68,.38)--(bp2);
    \draw[p2flow] (2.85,.05)--(2.85,-.32);
    \node[p2deep,label={[font=\small]below:$1$}] (bc1) at (1.78,-.52) {};
    \node[p2deep,label={[font=\small]below:$1$}] (bc2) at (2.53,-.52) {};
    \node at (3.16,-.52) {$\cdots$};
    \node[p2deep,label={[font=\small]below:$2$}] (bc3) at (3.93,-.52) {};
    \draw[p2edge] (bc1)--(bc2) (bc2)--(2.91,-.62)
      (3.41,-.62)--(bc3);
    \node[text=red!70!black] at (2.85,-1.22) {$\Longrightarrow\text{(a)}$};
  \end{scope}

  \begin{scope}[xshift=.20cm,yshift=-3.54cm]
    \node[font=\small] at (2.85,1.05) {(c) $\widetilde D_n$};
    \node[p2mark] (cx) at (1.92,.38) {};
    \node[p2leaf] (cl1) at (1.32,.75) {};
    \node[p2leaf] (cl2) at (1.32,.01) {};
    \node at (2.72,.38) {$\cdots$};
    \node[p2deep] (cc) at (3.57,.38) {};
    \node[p2deep] (cr1) at (4.39,.75) {};
    \node[p2deep] (cr2) at (4.39,.01) {};
    \draw[p2new] (cx)--(cl1) (cx)--(cl2);
    \draw[p2edge] (cx)--(2.45,.38) (2.99,.38)--(cc)--(cr1) (cc)--(cr2);
    \draw[p2flow] (2.85,.05)--(2.85,-.32);
    \node[p2deep,label={[font=\small]below:$1$}] (cd1) at (1.80,-.50) {};
    \node[p2deep,label={[font=\small]below:$0$}] (cd0) at (2.90,-.50) {};
    \node[p2deep] (cd2) at (3.90,-.18) {};
    \node[p2deep] (cd3) at (3.90,-.82) {};
    \node[font=\small,anchor=west] at (4.23,-.18) {$2$};
    \node[font=\small,anchor=west] at (4.23,-.82) {$2$};
    \draw[p2edge] (cd1)--(cd0)--(cd2) (cd0)--(cd3);
    \node at (3.10,-1.08) {$\delta=1,\ |C|=4$};
    \node[font=\footnotesize,text=blue!70!black] at (2.85,-1.59) {PASS};
  \end{scope}

  \begin{scope}[xshift=6.45cm,yshift=-3.54cm]
    \node[font=\small] at (2.85,1.05) {(d) $\widetilde D_4$};
    \node[p2mark] (dx) at (2.85,.38) {};
    \node[p2leaf] (dl1) at (2.15,.75) {};
    \node[p2leaf] (dl2) at (2.15,.01) {};
    \node[p2deep] (dr1) at (3.55,.75) {};
    \node[p2deep] (dr2) at (3.55,.01) {};
    \draw[p2new] (dx)--(dl1) (dx)--(dl2);
    \draw[p2edge] (dx)--(dr1) (dx)--(dr2);
    \draw[p2flow] (2.85,.05)--(2.85,-.32);
    \node[p2deep] (dd1) at (1.85,-.50) {};
    \node[p2deep,label={[font=\small]below:$0$}] (dd0) at (2.85,-.50) {};
    \node[p2deep] (dd2) at (3.85,-.50) {};
    \node[font=\small,anchor=east] at (1.52,-.50) {$2$};
    \node[font=\small,anchor=west] at (4.18,-.50) {$2$};
    \draw[p2edge] (dd1)--(dd0)--(dd2);
    \node at (3.10,-1.08) {$\delta=1,\ |C|=3$};
    \node[font=\footnotesize,text=red!70!black] at (2.85,-1.59) {REJECT};
  \end{scope}
\end{tikzpicture}%
}

\newcommand{\PtwoBoundaryGraphic}{%
\begin{tikzpicture}[font=\scriptsize,baseline]
  \path[use as bounding box] (0,-1.7) rectangle (12.5,1.9);
  \begin{scope}[xshift=0cm]
    \node[font=\small] at (1.75,1.65)
      {(a) one loaded arm $(\times20)$};
    \node[p2hub] (h0) at (.55,0) {};
    \node[p2support] (x0) at (1.8,0) {};
    \node[p2leaf] (l01) at (2.9,.55) {};
    \node[p2leaf] (l02) at (2.9,-.55) {};
    \draw[p2edge] (h0)--(x0)--(l01) (x0)--(l02);
    \node[below=2pt of x0] {$x\in P$};
    \node[align=center] at (1.75,-1.45) {$T[P]$ is independent};
  \end{scope}

  \begin{scope}[xshift=4.5cm]
    \node[font=\small] at (1.75,1.65) {(b) terminal tail};
    \node[p2hub] (h1) at (.1,0) {};
    \node[p2support] (x1) at (1.15,0) {};
    \node[p2leaf] (l1) at (1.15,-1.0) {};
    \node[p2support] (u1) at (2.35,0) {};
    \node[p2leaf] (v1) at (3.45,0) {};
    \draw[p2edge] (h1)--(x1)--(l1) (u1)--(v1);
    \draw[p2new] (x1)--(u1);
    \node[align=center,text=red!70!black] at (1.75,-1.45)
      {$\{\kappa(x),\kappa(u)\}=\{0,1\}$};
  \end{scope}

  \begin{scope}[xshift=9cm]
    \node[font=\small] at (1.75,1.65) {(c) bridge split};
    \node[p2hub] (a2) at (0,0) {};
    \node[p2support] (x2) at (1.15,0) {};
    \node[p2support] (y2) at (2.35,0) {};
    \node[p2hub] (b2) at (3.5,0) {};
    \node[p2leaf] (lx2) at (1.15,-1.0) {};
    \node[p2leaf] (ly2) at (2.35,-1.0) {};
    \draw[p2edge] (a2)--(x2) (y2)--(b2) (x2)--(lx2) (y2)--(ly2);
    \draw[p2new] (x2)--(y2);
    \node[align=center,text=red!70!black] at (1.75,-1.45)
      {$\{\kappa(x),\kappa(y)\}=\{1,1\}$};
  \end{scope}
\end{tikzpicture}%
}

\newcommand{\PtwoExceptionalAssembliesGraphic}{%
\begin{tikzpicture}[font=\footnotesize,baseline]
  \begin{scope}[xshift=0cm]
    \node[font=\small] at (3.15,1.85) {(a) E0 base, $q=2$};
    \node[p2deep,label=above:$2$] (e0a) at (0,0) {};
    \node[p2deep,label=above:$0$] (e0b) at (.8,0) {};
    \node[p2deep,label=above:$2$] (e0c) at (1.6,0) {};
    \draw[p2edge] (e0a)--(e0b)--(e0c);
    \node[p2deep,label=above:$2$] (e0s) at (3.15,0) {};
    \node[p2deep,label=above:$2$] (e0d) at (4.7,0) {};
    \node[p2deep,label=above:$0$] (e0e) at (5.5,0) {};
    \node[p2deep,label=above:$2$] (e0f) at (6.3,0) {};
    \draw[p2edge] (e0d)--(e0e)--(e0f);
    \node[p2support] (e0x) at (2.35,.55) {};
    \node[p2leaf] (e0xl) at (2.35,1.18) {};
    \draw[p2new] (e0c)--(e0x)--(e0s);
    \draw[p2edge] (e0x)--(e0xl);
    \node[p2support] (e0y) at (3.95,-.55) {};
    \node[p2leaf] (e0yl) at (3.95,-1.18) {};
    \draw[p2new] (e0s)--(e0y)--(e0d);
    \draw[p2edge] (e0y)--(e0yl);
    \node[p2support] (e0t) at (.05,-.7) {};
    \node[p2leaf] (e0tl) at (.05,-1.25) {};
    \draw[p2edge] (e0a)--(e0t)--(e0tl);
    \node at (1.05,-.8) {$\cdots$};
  \end{scope}

  \begin{scope}[xshift=8.1cm]
    \node[font=\small] at (3.0,1.85) {(b) E1 base, $q=2$};
    \node[p2deep,label=left:$0$] (c0) at (1.2,0) {};
    \node[p2deep,label=left:$1$] (c1) at (0,.8) {};
    \node[p2deep,label=left:$2$] (c2) at (0,-.8) {};
    \node[p2deep,label=right:$2$] (c3) at (2.15,0) {};
    \draw[p2edge] (c0)--(c1) (c0)--(c2) (c0)--(c3);
    \node[p2deep,label=above:$2$] (p0) at (4.45,0) {};
    \node[p2deep,label=above:$0$] (p1) at (5.25,0) {};
    \node[p2deep,label=above:$2$] (p2) at (6.05,0) {};
    \draw[p2edge] (p0)--(p1)--(p2);
    \node[p2support] (z) at (3.25,.55) {};
    \node[p2leaf] (zl) at (3.25,1.18) {};
    \draw[p2new] (c1)--(z)--(p0);
    \draw[p2edge] (z)--(zl);
    \node[p2support] (ct) at (.05,-1.50) {};
    \node[p2leaf] (ctl) at (-.5,-1.50) {};
    \draw[p2edge] (c2)--(ct)--(ctl);
    \node[p2support] (pt) at (6.05,-.7) {};
    \node[p2leaf] (ptl) at (6.05,-1.25) {};
    \draw[p2edge] (p2)--(pt)--(ptl);
    \node[text=black!60] at (2.25,-.72) {$\cdots$};
  \end{scope}
\end{tikzpicture}%
}

\newcommand{\PtwoLiftGrammarGraphic}{%
\begin{tikzpicture}[font=\footnotesize,baseline]
  \begin{scope}[xshift=0cm]
    \node[font=\small] at (1.45,1.75) {(i) $e\in M$};
    \node[p2hub] (a0) at (0,0) {};
    \node[p2support] (a1) at (1.45,0) {};
    \node[p2hub] (a2) at (2.9,0) {};
    \draw[blue!70!black,line width=1.3pt] (a0)--(a1)--(a2);
    \node[font=\small] at (1.45,-1.55) {no leaf};
  \end{scope}
  \begin{scope}[xshift=4.25cm]
    \node[font=\small] at (1.45,1.75) {(ii) $e\notin M$};
    \node[p2hub] (b0) at (0,0) {};
    \node[p2support] (b1) at (1.45,0) {};
    \node[p2hub] (b2) at (2.9,0) {};
    \node[p2leaf] (bl) at (1.45,-.9) {};
    \draw[p2edge] (b0)--(b1)--(b2) (b1)--(bl);
    \node[font=\small] at (1.45,-1.55) {one leaf};
  \end{scope}
  \begin{scope}[xshift=8.5cm]
    \node[font=\small] at (1.45,1.75) {(iii) hub completion};
    \node[p2hub] (c0) at (.7,0) {};
    \node[p2support] (c1) at (1.95,.55) {};
    \node[p2leaf] (c2) at (3.05,1.05) {};
    \node[p2support] (c3) at (1.95,-.55) {};
    \node[p2leaf] (c4) at (3.05,-1.05) {};
    \draw[p2edge] (-.2,0)--(c0) (c0)--(c1)--(c2) (c0)--(c3)--(c4);
    \node[font=\small,align=center] at (1.45,-1.55) {$3-\deg_H(v)$ arms};
  \end{scope}
\end{tikzpicture}%
}

\newcommand{\PtwoPfourRecoveryGraphic}{%
\begin{tikzpicture}[font=\small,baseline,
  every label/.append style={font=\small}]
  \path[use as bounding box] (-1.65,-1.88) rectangle (9.15,4.15);
  \begin{scope}[yshift=3.35cm]
    \node[font=\small] at (3.75,.75) {skeleton $H=P_4$};
    \node[p2hub,label=above:$h_1$] (s1) at (0,0) {};
    \node[p2hub,label=above:$h_2$] (s2) at (2.5,0) {};
    \node[p2hub,label=above:$h_3$] (s3) at (5,0) {};
    \node[p2hub,label=above:$h_4$] (s4) at (7.5,0) {};
    \draw[blue!70!black,line width=1.3pt] (s1)--(s2) (s3)--(s4);
    \draw[p2edge] (s2)--(s3);
    \node[font=\small,blue!70!black] at (1.25,-.42) {$M$};
    \node[font=\small,blue!70!black] at (6.25,-.42) {$M$};
  \end{scope}

  \begin{scope}
    \node[font=\small,anchor=west,fill=white,inner sep=1.5pt]
      at (4.45,1.48) {complete lift $W(P_4)$};
    \node[p2hub] (h1) at (0,0) {};
    \node[p2deep] (m12) at (1.25,0) {};
    \node[p2hub] (h2) at (2.5,0) {};
    \node[p2support] (c23) at (3.75,0) {};
    \node[p2hub] (h3) at (5,0) {};
    \node[p2deep] (m34) at (6.25,0) {};
    \node[p2hub] (h4) at (7.5,0) {};
    \draw[blue!70!black,line width=1.3pt] (h1)--(m12)--(h2)
      (h3)--(m34)--(h4);
    \draw[p2edge] (h2)--(c23)--(h3);
    \node[p2leaf] (e23) at (3.75,-1.0) {};
    \draw[p2edge] (c23)--(e23);

    \node[p2support] (a11) at (-.7,.72) {};
    \node[p2leaf] (a12) at (-1.45,1.12) {};
    \node[p2support] (a13) at (-.7,-.72) {};
    \node[p2leaf] (a14) at (-1.45,-1.12) {};
    \draw[p2edge] (h1)--(a11)--(a12) (h1)--(a13)--(a14);

    \node[p2support] (a21) at (2.5,.9) {};
    \node[p2leaf] (a22) at (2.5,1.72) {};
    \draw[p2edge] (h2)--(a21)--(a22);

    \node[p2support] (a31) at (5,-.9) {};
    \node[p2leaf] (a32) at (5,-1.72) {};
    \draw[p2edge] (h3)--(a31)--(a32);

    \node[p2support] (a41) at (8.2,.72) {};
    \node[p2leaf] (a42) at (8.95,1.12) {};
    \node[p2support] (a43) at (8.2,-.72) {};
    \node[p2leaf] (a44) at (8.95,-1.12) {};
    \draw[p2edge] (h4)--(a41)--(a42) (h4)--(a43)--(a44);
  \end{scope}

  \draw[p2flow] (3.35,2.98) -- (3.35,1.96)
    node[midway,left=3pt,font=\small] {forward};
  \draw[p2flow] (4.15,1.96) -- (4.15,2.98)
    node[midway,right=3pt,font=\small] {inverse};
\end{tikzpicture}%
}

\newcommand{\PtwoOrderStaircaseGraphic}{%
\begin{tikzpicture}[font=\scriptsize,baseline]
  \path[use as bounding box] (-.65,-.65) rectangle (11.35,4.2);
  \draw[->,black!70] (0,0)--(10.85,0)
    node[right] {$n$};
  \draw[->,black!70] (0,0)--(0,3.85)
    node[above,align=center] {maximum\\score};
  \foreach \y/\lab in {0/-2,.5/-1,1/0,1.5/1,2/2,2.5/3,3/4,3.5/5}
    \draw[black!35] (-.08,\y)--(.08,\y)
      node[left=3pt,text=black!70] {\lab};
  \foreach \x/\lab in {0/2,1.5/11,3/20,4.5/29,6/38,7.5/47,9/56,10.5/65}
    \draw[black!35] (\x,-.08)--(\x,.08)
      node[below=3pt,text=black!70] {\lab};
  \draw[blue!70!black,line width=1.35pt]
    (0,0)--(1.5,0)--(1.5,.5)--(3,.5)--(3,1)--(4.5,1)
    --(4.5,1.5)--(6,1.5)--(6,2)--(7.5,2)--(7.5,2.5)
    --(9,2.5)--(9,3)--(10.5,3)--(10.5,3.5)--(10.75,3.5);
  \foreach \x/\y/\q in {1.5/.5/1,3/1/2,4.5/1.5/3,6/2/4,7.5/2.5/5,9/3/6,10.5/3.5/7}
    \node[circle,fill=black,minimum size=3.4pt,inner sep=0pt,
      label={[text=black!75]above left:$q=\q$}] at (\x,\y) {};
  \draw[red!70!black,dashed,line width=.9pt] (4.5,-.05)--(4.5,3.75);
  \node[red!70!black,align=left,anchor=west] at (4.72,3.55)
    {first positive score\\at $n=29$};
  \node[blue!70!black,anchor=west] at (.45,3.55)
    {$\left\lfloor(n-20)/9\right\rfloor$};
  \node[align=left,anchor=east,text=black!70] at (10.75,.52)
    {black points: $|W(H)|=9q+2$};
\end{tikzpicture}%
}

\newcommand{\PtwoKernelDecompositionGraphic}{%
\begin{tikzpicture}[font=\scriptsize,baseline]
  \path[use as bounding box] (-.15,-1.75) rectangle (12.45,1.95);
  \node[font=\small] at (2.65,1.65) {a tree with fixed slack $k$};
  \node[p2deep] (a) at (.15,0) {};
  \node[p2support] (b) at (.8,0) {};
  \node[p2deep] (c) at (1.45,0) {};
  \node[p2deep] (d) at (2.2,0) {};
  \node[p2support] (e) at (2.95,.55) {};
  \node[p2deep] (f) at (3.7,.95) {};
  \node[p2support] (g) at (2.95,-.55) {};
  \node[p2deep] (h) at (3.7,-.95) {};
  \node[p2deep] (i) at (4.45,.95) {};
  \node[p2support] (j) at (5.1,.95) {};
  \node[p2leaf] (jl) at (5.55,1.35) {};
  \node[p2deep] (k) at (4.45,-.95) {};
  \node[p2support] (l) at (5.1,-.95) {};
  \node[p2leaf] (ll) at (5.55,-1.35) {};
  \draw[p2edge] (a)--(b)--(c)--(d)--(e)--(f)--(i)--(j)--(jl)
    (d)--(g)--(h)--(k)--(l)--(ll);
  \draw[red!70!black,rounded corners=7pt,line width=1.05pt,
    fill=red!5] (1.78,-1.25) rectangle (3.28,1.25);
  \node[text=red!70!black,align=center] at (2.53,0)
    {bounded\\defect core};
  \node[p2mark] at (1.78,0) {};
  \node[p2mark] at (3.28,.72) {};
  \node[p2mark] at (3.28,-.72) {};
  \node[align=center,text=blue!70!black,fill=white,inner sep=1pt] at (4.60,0)
    {ordinary lift\\forests};

  \draw[p2flow] (5.80,0)--(6.65,0)
    node[midway,above=2pt,align=center,text=black!70]
      {record core\\and ports};

  \node[font=\small] at (9.55,1.65)
    {finite kernel with retained ports};
  \node[draw=red!70!black,fill=red!5,rounded corners=5pt,font=\small,
    minimum width=2.25cm,minimum height=1.35cm,align=center] (K) at (9.45,.05)
    {$K\in\mathcal K_k$\\bounded labelled type};
  \node[p2mark,label={[font=\small]above:$p_1$}] (p1) at (8.32,.62) {};
  \node[p2mark,label={[font=\small]below:$p_2$}] (p2) at (8.32,-.46) {};
  \node[p2mark,label={[font=\small]above:$p_r$}] (p3) at (10.58,.12) {};
  \node[p2deep] (p1d) at (7.55,.82) {};
  \node[p2support] (p1s) at (7.02,1.06) {};
  \node[p2deep] (p2d) at (7.55,-.66) {};
  \node[p2support] (p2s) at (7.02,-.90) {};
  \node[p2deep] (p3d) at (11.22,.34) {};
  \node[p2support] (p3s) at (11.78,.58) {};
  \draw[p2edge] (p1)--(p1d)--(p1s)
    (p2)--(p2d)--(p2s) (p3)--(p3d)--(p3s);
  \node[text=black!55] at (6.70,1.18) {$\cdots$};
  \node[text=black!55] at (6.70,-1.02) {$\cdots$};
  \node[text=black!55] at (12.08,.70) {$\cdots$};
  \node[p2mark] at (8.32,.62) {};
  \node[p2mark] at (8.32,-.46) {};
  \node[p2mark] at (10.58,.12) {};
  \node[p2deep] at (.30,-1.55) {};
  \node[font=\small,anchor=west,text=black!70] at (.43,-1.55) {deep};
  \node[p2support] at (1.55,-1.55) {};
  \node[font=\small,anchor=west,text=black!70] at (1.72,-1.55) {support};
  \node[p2leaf] at (3.15,-1.55) {};
  \node[font=\small,anchor=west,text=black!70] at (3.32,-1.55) {leaf};
  \node[p2mark] at (4.15,-1.55) {};
  \node[font=\small,anchor=west,text=black!70] at (4.35,-1.55) {port};
\end{tikzpicture}%
}

\newcommand{\PtwoPhaseAtlasGraphic}{%
\begin{tikzpicture}[font=\footnotesize,baseline,
  band/.style={rounded corners=3pt,minimum height=.78cm,align=center,
    inner sep=3pt,line width=.75pt},
  card/.style={rounded corners=2pt,minimum height=.62cm,align=center,
    inner sep=2.5pt,line width=.65pt,font=\footnotesize},
  phasearrow/.style={-{Stealth[length=2mm]},black!60,line width=.65pt}]
  \path[use as bounding box] (-.10,-2.08) rectangle (12.45,1.78);
  \node[font=\small] at (6.2,1.52)
    {ordered integer slack layers $\Phi\in\mathbb Z_{\ge0}$};

  \node[band,draw=blue!65!black,fill=blue!5,minimum width=1.65cm] (a)
    at (.95,.62) {$0\le\Phi\le19$};
  \node[band,draw=red!65!black,fill=red!3,minimum width=.72cm] (b)
    at (2.35,.62) {$20$};
  \node[band,draw=red!65!black,fill=red!3,minimum width=.72cm] (c)
    at (3.30,.62) {$21$};
  \node[band,draw=teal!65!black,fill=teal!5,minimum width=1.70cm] (d)
    at (4.85,.62) {$22\le\Phi\le39$};
  \node[band,draw=red!65!black,fill=red!3,minimum width=.72cm] (e)
    at (6.35,.62) {$40$};
  \node[band,draw=red!65!black,fill=red!3,minimum width=.72cm] (f)
    at (7.30,.62) {$41$};
  \node[band,draw=red!65!black,fill=red!3,minimum width=.72cm] (g)
    at (8.25,.62) {$42$};
  \node[band,draw=black!45,dashed,fill=black!2,minimum width=2.15cm] (h)
    at (10.55,.62) {$\Phi>42$};

  \draw[phasearrow] (a)--(b);
  \draw[phasearrow] (b)--(c);
  \draw[phasearrow] (c)--(d);
  \draw[phasearrow] (d)--(e);
  \draw[phasearrow] (e)--(f);
  \draw[phasearrow] (f)--(g);
  \draw[phasearrow] (g)--(h);

  \node[card,draw=blue!65!black,fill=blue!2,minimum width=1.58cm] (ca)
    at (.95,-.50) {\textsc{loading}\\canonical lifts};
  \node[card,draw=red!65!black,fill=red!2,minimum width=.90cm] (cb)
    at (1.95,-1.48) {\textsc{residue}\\first};
  \node[card,draw=red!65!black,fill=red!2,minimum width=1.25cm] (cc)
    at (3.30,-.50) {\textsc{assembly}\\$E0/E1$; expansion};
  \node[card,draw=teal!65!black,fill=teal!2,minimum width=1.65cm] (cd)
    at (4.85,-1.48) {\textsc{loading}\\uniform corridor};
  \node[card,draw=red!65!black,fill=red!2,minimum width=.90cm] (ce)
    at (6.35,-.50) {\textsc{residue}\\second};
  \node[card,draw=red!65!black,fill=red!2,minimum width=1.40cm] (cf)
    at (7.30,-1.48) {\textsc{new atoms}\\$G_0,G_1,S^\star$};
  \node[card,draw=red!65!black,fill=red!2,minimum width=1.55cm] (cg)
    at (8.65,-.50) {\textsc{new atoms}\\$H_0,H_1,K_0$--$K_2$};
  \node[card,draw=black!45,dashed,fill=black!1,minimum width=2.20cm] (ch)
    at (10.55,-1.48) {\textsc{theorem only}\\no explicit atlas here};

  \foreach \u/\v in {a/ca,c/cc,d/cd,e/ce,f/cf,g/cg,h/ch}
    \draw[black!45,line width=.55pt,shorten >=2.2pt]
      (\u.south)--(\v.north);
  \draw[black!45,line width=.55pt,rounded corners=2pt]
    (b.south)--(2.35,.06)--(1.95,.06)--(cb.north);
\end{tikzpicture}%
}

\newcommand{\PtwoCertificateWorkflowGraphic}{%
\begin{tikzpicture}[font=\scriptsize,baseline,
  certstep/.style={rounded corners=3pt,minimum width=2.05cm,
    text width=1.65cm,minimum height=.92cm,align=center,
    inner sep=3pt,line width=.7pt},
  badge/.style={circle,draw=black!45,fill=white,minimum size=4.8mm,
    inner sep=0pt,font=\scriptsize\bfseries}]
  \path[use as bounding box] (-.10,-.82) rectangle (12.45,1.62);
  \node[certstep,draw=black!55,fill=black!3] (s1) at (1.15,.30)
    {symbolic\\finite bound};
  \node[certstep,draw=blue!65!black,fill=blue!4] (s2) at (3.65,.30)
    {exact bounded\\enumeration};
  \node[certstep,draw=teal!65!black,fill=teal!4] (s3) at (6.15,.30)
    {claim--output\\binding};
  \node[certstep,draw=blue!65!black,fill=blue!4] (s4) at (8.65,.30)
    {independent\\scope check};
  \node[certstep,draw=red!65!black,fill=red!3] (s5) at (11.15,.30)
    {fail-closed\\manifest};
  \foreach \a/\b in {s1/s2,s2/s3,s3/s4,s4/s5}
    \draw[p2flow] (\a.east)--(\b.west);
  \foreach \x/\n in {s1/1,s2/2,s3/3,s4/4,s5/5}
    \node[badge] at ([yshift=6mm]\x.north) {\n};
  \node[font=\footnotesize,text=black!65] at (6.15,-.60)
    {finite, exact, manuscript-bound, and reproducible};
\end{tikzpicture}%
}

\newcommand{\PtwoAtomGalleryGraphic}{%
\begin{tikzpicture}[font=\footnotesize,baseline,
  atom/.style={circle,draw=blue!70!black,fill=white,minimum size=5.5pt,inner sep=0pt},
  spec/.style={circle,draw=red!70!black,fill=white,line width=1pt,
    minimum size=5.5pt,inner sep=0pt},
  ae/.style={draw=black!65,line width=.8pt},
  panel/.style={draw=black!20,rounded corners=3pt,fill=black!1},
  wt/.style={font=\footnotesize,fill=white,inner sep=.5pt,text=black}]
  \path[use as bounding box] (-.05,-5.70) rectangle (12.45,3.15);

  \node[anchor=west,font=\small\bfseries,text=black!70] at (.08,2.96) {$\Phi=41$};
  \draw[panel] (.12,.62) rectangle (3.92,2.72);
  \draw[panel] (4.12,.62) rectangle (7.92,2.72);
  \draw[panel] (8.12,.62) rectangle (12.32,2.72);

  \begin{scope}[xshift=.25cm,yshift=1.62cm]
    \node[font=\small] at (1.77,.77) {$G_0$};
    \node[atom] at (1.77,0) {};
    \node[wt,anchor=west] at (1.98,0) {$3$};
    \node[text=blue!70!black] at (1.77,-.72) {$Q_C\succeq0$};
  \end{scope}

  \begin{scope}[xshift=4.25cm,yshift=1.62cm]
    \node[font=\small] at (1.77,.77) {$G_1$};
    \node[atom] (gc) at (1.77,0) {};
    \node[atom] (g1) at (.82,.52) {};
    \node[atom] (g2) at (.82,-.52) {};
    \node[atom] (g3) at (2.72,0) {};
    \draw[ae] (gc)--(g1) (gc)--(g2) (gc)--(g3);
    \node[wt] at (1.77,-.27) {$0$};
    \node[wt,anchor=east] at (.60,.52) {$2$};
    \node[wt,anchor=east] at (.60,-.52) {$2$};
    \node[wt,anchor=west] at (2.94,0) {$2$};
    \node[text=blue!70!black] at (1.77,-.78) {$Q_C\succeq0$};
  \end{scope}

  \begin{scope}[xshift=8.28cm,yshift=1.66cm]
    \node[font=\small] at (1.60,.77) {$S^\star$};
    \node[spec] (su) at (.32,.06) {};
    \node[spec] (sm) at (.96,.06) {};
    \node[spec] (sc) at (1.60,.06) {};
    \node[spec] (sv) at (1.60,.49) {};
    \node[spec] (svm) at (2.28,.49) {};
    \node[spec] (svl) at (2.96,.49) {};
    \node[spec] (sw) at (1.60,-.37) {};
    \node[spec] (swm) at (2.28,-.37) {};
    \node[spec] (swl) at (2.96,-.37) {};
    \draw[ae] (su)--(sm)--(sc)--(sv)--(svm)--(svl)
      (sc)--(sw)--(swm)--(swl);
    \node[wt,anchor=east] at (.10,.06) {$2$};
    \node[wt] at (.96,-.19) {$0$};
    \node[wt,anchor=west] at (1.84,.06) {$0$};
    \node[wt,anchor=east] at (1.38,.49) {$1$};
    \node[wt] at (2.28,.24) {$0$};
    \node[wt,anchor=west] at (3.18,.49) {$2$};
    \node[wt,anchor=east] at (1.38,-.37) {$1$};
    \node[wt] at (2.28,-.62) {$0$};
    \node[wt,anchor=west] at (3.18,-.37) {$2$};
    \node[text=red!70!black] at (1.60,-.84) {$n_-(Q_C)=1$};
  \end{scope}

  \node[anchor=west,font=\small\bfseries,text=black!70] at (.08,.35) {$\Phi=42$};
  \draw[panel] (.12,-2.05) rectangle (3.92,.05);
  \draw[panel] (4.12,-2.05) rectangle (7.92,.05);
  \draw[panel] (8.12,-2.05) rectangle (12.32,.05);

  \begin{scope}[xshift=.25cm,yshift=-1.02cm]
    \node[font=\small] at (1.77,.78) {$H_0$};
    \node[atom] (h00) at (1.10,0) {};
    \node[atom] (h01) at (2.44,0) {};
    \draw[ae] (h00)--(h01);
    \node[wt] at (1.10,-.28) {$2$};
    \node[wt] at (2.44,-.28) {$1$};
    \node[text=blue!70!black] at (1.77,-.78) {$Q_C\succeq0$};
  \end{scope}

  \begin{scope}[xshift=4.25cm,yshift=-1.02cm]
    \node[font=\small] at (1.77,.78) {$H_1$};
    \node[atom] (h10) at (1.42,0) {};
    \node[atom] (h11) at (2.10,0) {};
    \node[atom] (h12) at (2.78,0) {};
    \node[atom] (h13) at (.62,.55) {};
    \node[atom] (h14) at (.62,-.55) {};
    \draw[ae] (h10)--(h11)--(h12) (h10)--(h13) (h10)--(h14);
    \node[wt] at (1.42,-.28) {$0$};
    \node[wt] at (2.10,-.28) {$1$};
    \node[wt,anchor=west] at (3.00,0) {$1$};
    \node[wt,anchor=east] at (.40,.55) {$2$};
    \node[wt,anchor=east] at (.40,-.55) {$2$};
    \node[text=blue!70!black] at (1.77,-.82) {$Q_C\succeq0$};
  \end{scope}

  \begin{scope}[xshift=8.25cm,yshift=-1.02cm]
    \node[font=\small] at (1.92,.78) {$K_0$};
    \node[spec] (k03) at (.30,0) {};
    \node[spec] (k02) at (.90,0) {};
    \node[spec] (k01) at (1.50,0) {};
    \node[spec] (k00) at (2.10,0) {};
    \node[spec] (k04) at (2.70,0) {};
    \node[spec] (k05) at (3.35,.50) {};
    \node[spec] (k06) at (3.35,-.50) {};
    \draw[ae] (k03)--(k02)--(k01)--(k00)--(k04)--(k05) (k04)--(k06);
    \foreach \x/\v in {.30/2,.90/0,1.50/1,2.10/0,2.70/0}
      \node[wt] at (\x,-.28) {$\v$};
    \node[wt,anchor=west] at (3.56,.50) {$1$};
    \node[wt,anchor=west] at (3.56,-.50) {$1$};
    \node[text=red!70!black] at (1.92,-.82) {$n_-(Q_C)=1$};
  \end{scope}

  \node[anchor=west,font=\small\bfseries,text=black!70] at (.88,-2.35)
    {$\Phi=42$ (continued)};
  \draw[panel] (.88,-5.52) rectangle (6.03,-2.58);
  \draw[panel] (6.47,-5.52) rectangle (11.62,-2.58);

  \begin{scope}[xshift=1.08cm,yshift=-4.05cm]
    \node[font=\small] at (2.38,1.18) {$K_1$};
    \node[spec] (k13) at (.35,-.15) {};
    \node[spec] (k12) at (.95,-.15) {};
    \node[spec] (k11) at (1.55,-.15) {};
    \node[spec] (k10) at (2.15,-.15) {};
    \node[spec] (k14) at (2.75,-.15) {};
    \node[spec] (k15) at (3.35,-.15) {};
    \node[spec] (k16) at (3.95,-.15) {};
    \node[spec] (k17) at (2.15,.52) {};
    \node[spec] (k18) at (3.10,.92) {};
    \node[spec] (k19) at (1.20,.92) {};
    \draw[ae] (k13)--(k12)--(k11)--(k10)--(k14)--(k15)--(k16)
      (k10)--(k17)--(k18) (k17)--(k19);
    \foreach \x/\v in {.35/2,.95/0,1.55/1,2.15/0,2.75/1,3.35/0,3.95/2}
      \node[wt] at (\x,-.45) {$\v$};
    \node[wt,anchor=west] at (2.38,.42) {$0$};
    \node[wt,anchor=east] at (.98,.92) {$1$};
    \node[wt,anchor=west] at (3.32,.92) {$1$};
    \node[text=red!70!black] at (2.15,-1.10) {$n_-(Q_C)=1$};
  \end{scope}

  \begin{scope}[xshift=6.67cm,yshift=-4.05cm]
    \node[font=\small] at (2.38,1.18) {$K_2$};
    \node[spec] (k23) at (.35,-.15) {};
    \node[spec] (k22) at (.95,-.15) {};
    \node[spec] (k21) at (1.55,-.15) {};
    \node[spec] (k20) at (2.15,-.15) {};
    \node[spec] (k24) at (2.75,-.15) {};
    \node[spec] (k25) at (3.35,-.15) {};
    \node[spec] (k26) at (3.95,-.15) {};
    \node[spec] (k27) at (2.15,.35) {};
    \node[spec] (k28) at (2.15,.85) {};
    \node[spec] (k29) at (3.25,.85) {};
    \draw[ae] (k23)--(k22)--(k21)--(k20)--(k24)--(k25)--(k26)
      (k20)--(k27)--(k28)--(k29);
    \foreach \x/\v in {.35/1,.95/1,1.55/0,2.15/0,2.75/1,3.35/0,3.95/2}
      \node[wt] at (\x,-.45) {$\v$};
    \node[wt,anchor=east] at (1.92,.35) {$1$};
    \node[wt,anchor=east] at (1.92,.85) {$0$};
    \node[wt,anchor=west] at (3.47,.85) {$2$};
    \node[text=red!70!black] at (2.15,-1.10) {$n_-(Q_C)=1$};
  \end{scope}
\end{tikzpicture}%
}

\title[Order laws and finite-defect stability]{Sharp Order Laws and Finite-Defect Stability for Domination versus Small Laplacian Eigenvalues in Trees}
\author{Yufeng Wang}
\address{Independent Researcher, Foshan, Guangdong, China}
\email{yufeng.wang.research@gmail.com}
\urladdr{https://yiyuxi123.github.io/}
\thanks{ORCID: \href{https://orcid.org/0009-0001-5120-2046}{0009-0001-5120-2046}.}
\date{August 27, 2026}

\begin{document}

\begin{abstract}
For a tree $T$, let $\gamma(T)$ be its domination number and let $\mu(T)$
count the Laplacian eigenvalues in $[0,1)$.  The strict inequality
$\gamma(T)/\mu(T)<4/3$ leaves open three linked questions: the exact
finite-order defect, the structure near equality, and whether a fixed defect
permits only finitely many local obstructions.  We answer all three.  For
every $n\ge2$,
\[
 \max_{|V(T)|=n}\bigl(7\gamma(T)-9\mu(T)\bigr)
 =\left\lfloor\frac{n-20}{9}\right\rfloor.
\]
Thus order twenty-nine is the first at which the ratio $9/7$ can be
exceeded.  The proof is based on an exact decomposition of
$\Phi(T)=|V(T)|-20-9(7\gamma(T)-9\mu(T))$ into nonnegative integer
terms, together with a clean-contraction identity.

For subcubic trees, equality cases are canonical lifts of trees with perfect
matchings.  We classify every layer through $\Phi=42$, locating successive
phase transitions at $20$, $21$, $40$, $41$, and $42$.  The finiteness
mechanism is simple to state: after the repeatable matched-skeleton pieces
are removed, fixed slack leaves only a bounded exceptional core.  Formally,
a local compactness theorem bounds every nonordinary weighted deep component
at fixed packing surplus and residual negative index.  Consequently, for
each fixed $k$, every subcubic tree with $\Phi(T)=k$ consists of a bounded
ported defect kernel, an arbitrary compatible forest of ordinary tiles, and
at most $\lfloor k/21\rfloor$ clean expansions.
All unbounded statements are proved symbolically; exact computation is used
only for explicitly bounded atom lists and independent finite verification.

\end{abstract}

\maketitle

\section{Introduction}
\label{sec:introduction}

The distribution of Laplacian eigenvalues below a fixed threshold often
reflects combinatorial covering structure.  For a finite simple graph $G$,
write $\gamma(G)$ for the domination number and let $\mu(G)$ be the number
of Laplacian eigenvalues in $[0,1)$, counted with multiplicity.  Hedetniemi,
Jacobs, and Trevisan proved $\mu(G)\le\gamma(G)$
\cite{HedetniemiJacobsTrevisan2016}; later work developed equality criteria,
diameter-sensitive bounds, and extensions to other graph classes
\cite{CardosoJacobsTrevisan2017,GuoXueLiu2025,BarikJena2026}.  A recent
tree-specific analysis proved the strict ratio bound $\gamma(T)/\mu(T)<4/3$
and constructed an infinite family approaching $4/3$, while also isolating
structural inequalities involving leaves and support vertices
\cite{RajendraprasadSankaranarayanan2025}.  These results leave an extremal
question that is both quantitative and structural: at a fixed order, how
large can the gap between domination and the small-eigenvalue count be, and
what do the near-extremal trees look like?

The surrounding eigenvalue-distribution literature is broader than the
domination inequality alone.  For trees, small-eigenvalue counts have been
studied under structural transformations and relative to the average degree
\cite{BragaRodriguesTrevisan2013,JacobsOliveiraTrevisan2021}; for general
graphs, related interval counts have been bounded using independence,
diameter, and degree data
\cite{AhanjidehAkbariFakharanTrevisan2022,AkbariAlaeiyanDarougheh2025}.
Those results provide distributional context, while the present problem
couples the threshold-one count specifically to domination and asks for an
exact finite-order and fixed-defect structure.

That recent ratio theorem is the strongest direct precedent for the present
problem.  It already establishes strictness and asymptotic sharpness; indeed,
the tree later denoted $W(P_2)$ is its eleven-vertex example.  The present
work begins at the finer questions not settled by a limiting ratio: the
exact optimum at each order, a canonical inverse description of the equality
objects, finite local structure at every fixed defect, and an explicit
low-defect phase atlas.  This comparison delimits the relation to the cited
predecessor; we make no broader priority claim.

The coefficients $7$ and $9$ are dictated by the canonical equality lifts,
not chosen after the fact.  If the skeleton has $2q$ vertices, then its lift
satisfies
\[
 (|V|,\gamma,\mu)=(9q+2,4q+1,3q+1),
 \qquad 7\gamma-9\mu=q-2.
\]
Thus the zero-score lift $W(P_4)$ ($q=2$) has ratio $\gamma/\mu=9/7$,
whereas $q=3$ is the first positive-score lift and has order twenty-nine.
The linear score therefore records the first crossing of a barrier already
present in the equality family.  We ask for the sharp score at every order
and for every tree close to it.  A clean
three-edge contraction lowers the order by three while changing both
parameters by one.  Three effects must therefore
be retained at once: the half-open interval $[0,1)$ requires exact negative
inertia of $L(T)-I$; contraction changes order and score at different rates;
and global equality does not locate the domination choices or spectral
directions among the branches.  Enumeration can expose early examples, but
cannot supply this normalization or a structural exhaustion.

Our starting point is an exact slack decomposition.  Separate a tree into
its leaves, penultimate vertices, and deep vertices.  On a tree with no
adjacent degree-two vertices, known domination--spectrum inequalities and a
deep-vertex domination estimate give
\[
 \Phi(T):=|V(T)|-20-9\bigl(7\gamma(T)-9\mu(T)\bigr)
  =a+b+20s+21t+60r,
\]
where all five terms are nonnegative integers.  Clean contraction increases
$\Phi$ by at least twenty-one, so this local identity extends to every tree
and proves the global order law.  More importantly, the separated
coefficients make low slack rigid: below twenty, the residue, spectral, and
domination--spectral terms must all vanish.

The five summands record surplus leaves, unused deep-domination capacity,
the residue of $p-1$ modulo three, spectral excess over the support count,
and integral slack in the reduced domination--spectrum inequality.  The
coefficients $20,21,$ and $60$ are forced by this change of coordinates,
and their separation turns the global inequality into a stability theorem.

The vanishing terms expose a second representation.  Eliminating every
leaf bundle from $L(T)-I$ leaves a direct sum of weighted Laplacians on the
deep components.  Exact domination decomposes over the same components.  A
distance-three packing argument and positive semidefiniteness force a local
boundary pressure, while a global core-edge identity forces the sum of all
local pressures to have the opposite sign.  The resulting sign collision
leaves only weighted three-vertex paths.  These paths are precisely the
local pieces of a matched-skeleton lift $W(H)$, where $H$ is a subcubic tree
with a perfect matching.

This second representation is reversible.  From an uncoloured equality
tree one recovers the leaves and their neighbours, reads each deep
three-vertex path as a matching edge, suppresses every two-hub support into
a nonmatching edge, and deletes the remaining support--leaf arms.  The
recovered skeleton is therefore canonical up to isomorphism rather than an
existential witness attached to the proof.

This mechanism gives a sharp stability theorem rather than only an equality
case.  After the three small trees $P_2,P_3,K_{1,3}$ are listed, every
subcubic tree with $0\le\Phi(T)\le19$ has order at least five, and every unit
of slack is an additional leaf on a distinct arm support of $W(H)$;
conversely, every such loading has exactly that slack.  The construction has
$2q+2$ available arms when $H$ has $2q$ vertices, so the capacity is exact.

Beyond the small slack-twenty tree $P_4$, the pressure sum changes by exactly
one: it becomes
$1-e_{PP}$ instead of $-e_{PP}$.  This leaves only two possibilities.  If
$e_{PP}=0$, the old structure continues and all twenty units are distinct
arm loads.  If $e_{PP}=1$, local pressure again vanishes everywhere, and the
endpoints of the unique penultimate edge have deep-incidence multiset either
$\{0,1\}$ or $\{1,1\}$.  The first case is a terminal tail beyond an arm;
the second splits a leafed support on a nonmatching skeleton edge into two
adjacent supports.  A putative zero-edge exceptional component is excluded
by the local surplus theorem and the subcubic degree boundary.  Thus the
order-at-least-five part of the boundary layer has a complete three-type
classification; together with $P_4$ this is the all-order classification.
The local pictures are shown in \Cref{fig:stability-boundary}.

At $\Phi=21$, clean expansion and an additional negative residual direction
first become compatible with the slack identity.  Every clean quotient is an
equality tree, while the reduced residue branch gives one-leaf inflations of
both slack-twenty defects and two incidence families.  Packing equality suppresses
the apparent spectral branch
to a perfect-matching tree, where an induced $P_6$ forces one negative
direction too many.

The same construction also solves the fixed-order extremal problem.  An
equality lift has order $9q+2$, domination number $4q+1$, and small-eigenvalue
count $3q+1$.  Adding leaves at existing supports preserves both parameters,
so every intervening order is attained after integer rounding.  Consequently
\[
 \max_{|V(T)|=n}(7\gamma(T)-9\mu(T))
   =\left\lfloor\frac{n-20}{9}\right\rfloor
\]
for every $n\ge2$.  Distinct-arm loading preserves maximum degree three,
which yields every subcubic extremal residue layer from order twenty-nine
onward.  The unloaded lifts also show that the strict general-tree ratio
bound has supremum $4/3$ already inside the subcubic class.

The low-slack classifications suggest a stronger principle: after ordinary
matched-skeleton material is removed, a fixed value of $\Phi$ should leave
only bounded local information.  We prove this by suppressing a maximum
distance-three packing inside each weighted deep component.  The suppressed
edges form an almost-perfect matching, and the residual matrix becomes a
bounded-rank positive-semidefinite perturbation of $I-A(F)$.  Its negative
index bounds the diameter of $F$ by interlacing with induced paths; a Moore
bound then bounds the component order.  The only zero-cost component is the
ordinary weighted path $(2,0,2)$.  Globally, the five-term and pressure
identities bound every exceptional component, support event, and port.
This yields, for every fixed $k$, a finite set of ported defect kernels, with
arbitrary ordinary lift forests at their ports and at most
$\lfloor k/21\rfloor$ clean expansions.

The endpoints of the explicit atlas are arithmetically forced by the two
smallest positive coordinate costs in the same identity.  The newly feasible
budgets are
\[
\resizebox{\textwidth}{!}{$
\begin{array}{c|c|l}
\Phi & (s,t,a+b) & \text{new structural budget}\\ \hline
20&(1,0,0)&\text{first residue}\\
21&(0,1,0)&\text{first spectral excess or one clean contraction}\\
40&(2,0,0)&\text{second residue}\\
41&(1,1,0)\ \text{or}\ (2,0,1)&\text{mixed residue--spectral or pressure-two}\\
42&(0,2,0),\ (1,1,1),\ \text{or}\ (2,0,2)&
\text{second spectral, mixed one-surplus, or two-surplus boundary}.
\end{array}
$}
\]
Thus $40=2\cdot20$, $41=20+21$, and $42=2\cdot21$; forty-two is the
second spectral boundary, not an arbitrary catalogue cutoff.  We refine the
finite-kernel theorem into an explicit phase atlas through this boundary.  A
uniform spectral gap removes the putative one-negative row
below forty.  Layers $20$--$39$ form one loading corridor.  At forty a
second residue creates double-residue incidence trees; at forty-one two
pressure-two atoms and one equality-size one-negative atom appear; at
forty-two two new positive-semidefinite atoms, three one-negative atoms, and
two-exception assemblies appear.  A separate all-orders finite reduction
eliminates the two-negative equality row at forty-two.  These transitions
are governed by the coefficients of the same five-term identity, so the
atlas is an extension of the order law rather than a detached catalogue.

The main contributions fall under three linked statements:
\begin{enumerate}[label=(\roman*)]
  \item an exact order law and its canonical near-equality structure: the
        five-term identity gives the fixed-order maximum, the unique inverse
        matched-skeleton representation, and the complete classifications
        through slack twenty-one;
  \item a fixed-defect finiteness theorem: local compactness leaves only a
        finite effective set of ported exceptional kernels at each fixed
        slack, with arbitrary ordinary lift forests at their ports; and
  \item the explicit low-defect phase atlas forced by the same coordinates:
        the corridor $20\le\Phi\le39$ and the complete classifications at
        $\Phi=40,41,42$, after symbolic finite reductions and bounded exact
        certificates.
\end{enumerate}
The unbounded steps are symbolic.  Exact computation is used only after an
all-orders reduction has made a local atom list finite, and for independent
testing of constructions, invariants, and branch interfaces.

The paper follows the same three stages.  The preliminaries through
\Cref{sec:slack-twenty-one} establish the exact order law, canonical lifts,
and sharp near-equality layers.  \Cref{sec:local-compactness,sec:finite-kernels}
prove the fixed-defect mechanism.  The remaining structural sections derive
the spectral exclusions and classify the forced phases through forty-two;
the final sections state the exact-computation boundary, limitations, and
open problems.

\section{Preliminaries and two exact decompositions}
\label{sec:preliminaries}

All graphs are finite and simple.  We write $|G|:=|V(G)|$ for the order of a
graph, and in particular $|T|=|V(T)|$ for a tree.  For a graph $G$, let $L(G)$ be its
Laplacian matrix.  The inertia of a real symmetric matrix $A$ is written
\[
 \In(A)=(n_+(A),n_-(A),n_0(A)).
\]
Since $L(G)$ is positive semidefinite,
\begin{equation}
 \mu(G)=n_-(L(G)-I).                                  \label{eq:mu-inertia}
\end{equation}

Let $T$ be a tree of order at least three.  Write $L$ for its leaf set,
$P=N_T(L)$ for its penultimate set, and $D=V(T)\setminus(L\cup P)$ for its
deep set; their cardinalities are $\ell,p,d$.  For $v\in D$, let
\[
 b_v=|N_T(v)\cap P|.
\]
The components of $T[D]$ will always be denoted by $C$.
For $v\in V(C)$ we call
\[
 \deg_T(v)=\deg_C(v)+b_v
\]
its \emph{total degree}.  Thus total degree always means degree in the
original tree, while $\deg_C(v)$ means degree inside the weighted deep
component.  The terminology is used repeatedly in the local compactness
and phase-classification arguments.

\subsection{Leaf bundles and inertia}

The following congruence is the spectral reason that additional leaves can
be tracked exactly.  It permits arbitrary edges inside the support set.

\begin{lemma}[Leaf-bundle congruence]
\label{lem:leaf-bundle}
Let $G$ be a graph whose pendant vertices form nonempty bundles $L_u$
indexed by a set $U$ of non-pendant support vertices.  Put
$\lambda_u=|L_u|$, $R=\diag(\sqrt{\lambda_u}:u\in U)$, and
$D=V(G)\setminus(\bigcup_u L_u\cup U)$.  Then
\[
 L(G)-I\ \cong
 \begin{pmatrix}0&-R\\-R&0\end{pmatrix}
 \oplus (L(G)-I)[D,D]\oplus 0_{\sum_u(\lambda_u-1)},
\]
where $0_m$ denotes the $m\times m$ zero matrix and $\cong$ denotes real
congruence.  Consequently,
\begin{equation}
 n_-(L(G)-I)=|U|+n_-((L(G)-I)[D,D]).                 \label{eq:bundle-negative}
\end{equation}
\end{lemma}

\begin{proof}
In every bundle, replace the leaf coordinates by the normalized all-ones
coordinate and $\lambda_u-1$ orthogonal difference coordinates.  The
difference coordinates are decoupled zero modes.  Ordering the remaining
coordinates as bundle averages $x$, supports $y$, and deep vertices $z$
gives a quadratic form
\[
 -2x^{\trans}Ry+y^{\trans}Ay+2y^{\trans}Fz+z^{\trans}Qz,
\]
where $A$ includes all support--support terms and
$Q=(L(G)-I)[D,D]$.  The invertible shear
\[
 x=x'+\tfrac12R^{-1}Ay+R^{-1}Fz
\]
cancels both $y^{\trans}Ay$ and $2y^{\trans}Fz$.  The remaining form is
$-2{x'}^{\trans}Ry+z^{\trans}Qz$.  Each $2\times2$ block determined by a
diagonal entry of $R$ has one positive and one negative eigenvalue, proving
the congruence and \eqref{eq:bundle-negative}.
\end{proof}

The repeated-leaf contribution to the Laplacian eigenvalue one belongs to
the Faria residual-matrix line
\cite{Faria1985,AndradeCardosoPastenRojo2015}; a recent reduction formula is
given in \cite{TianWong2026}.  The congruence above is recorded because the
present argument needs the full positive/negative/zero inertia, not only the
multiplicity of one.  It also permits arbitrary edges inside the support set.

For the full leaf partition of a tree, the residual block is especially
simple:
\begin{equation}
 (L(T)-I)[D,D]
 =\bigoplus_C Q_C,
 \qquad
 Q_C=L(C)+\diag(b_v-1:v\in C).                       \label{eq:deep-residual}
\end{equation}
Therefore
\begin{equation}
 \mu(T)=p+\sum_C n_-(Q_C).                            \label{eq:mu-decomposition}
\end{equation}

\subsection{Exact domination on the deep forest}

For a component $C$, define
\[
 A_C=\{v\in C:b_v>0\},\qquad R_C=V(C)\setminus A_C,
\]
and
\[
 \delta_C=\min\{|S|:S\subseteq V(C),\ R_C\subseteq N_C[S]\}.
\]

\begin{lemma}[Deep-component domination decomposition]
\label{lem:domination-decomposition}
Every tree of order at least three satisfies
\begin{equation}
 \gamma(T)=p+\sum_C\delta_C.                          \label{eq:gamma-decomposition}
\end{equation}
\end{lemma}

\begin{proof}
Start with a minimum dominating set.  If a penultimate vertex $u$ is absent,
all its leaf neighbours must be selected; replace them by $u$.  Once $u$ is
selected, every selected leaf adjacent to it is redundant.  Repeating this
normalizes a minimum dominating set so that it contains all of $P$ and no
leaf.

The chosen vertices in $P$ already dominate $A_C$.  A vertex of $R_C$ has
no neighbour in $P$, and vertices in another deep component cannot dominate
it.  Thus the remaining choices are exactly independent closed-neighbourhood
covers of the sets $R_C$.  Taking a minimum cover in each component proves
both directions of \eqref{eq:gamma-decomposition}.
\end{proof}

Meir and Moon proved the full-target domination--packing identity for trees
\cite{MeirMoon1975}.  The partial-target form below follows by row restriction
in Farber's totally balanced closed-neighbourhood framework
\cite{Farber1984}; we include a direct tree proof for completeness.

\begin{lemma}[Partial domination--packing duality]
\label{lem:partial-packing}
For a tree $C$ and $R\subseteq V(C)$,
\[
 \min\{|S|:R\subseteq N_C[S]\}
 =\max\{|Z|:Z\subseteq R,\ d_C(z,z')\ge3
                  \text{ for }z\ne z'\}.
\]
\end{lemma}

\begin{proof}
A closed neighbourhood contains at most one member of a distance-three
packing, giving the easy inequality.  For the reverse inequality, root $C$.
Repeatedly choose a deepest uncovered target $u$.  Select its parent, or
$u$ itself if it is the root, and delete every target in the selected
vertex's closed neighbourhood.  Record $u$.

The selected vertices cover $R$.  A selected parent cannot be selected again:
its entire closed neighbourhood of targets was deleted in its first
iteration.  Two recorded targets cannot be at distance
at most two: an ancestor-side or same-depth target would have been deleted,
whereas a descendant-side target would have been deeper than the target
chosen first.  Hence the recorded targets form a packing, and the procedure
constructs a cover and a packing of the same cardinality.
\end{proof}

\begin{corollary}
\label{cor:three-delta}
If every vertex of $R$ has degree at least two in $C$, then
\[
 3\min\{|S|:R\subseteq N_C[S]\}\le |V(C)|.
\]
\end{corollary}

\begin{proof}
Take a maximum packing from \Cref{lem:partial-packing}.  Its closed
neighbourhoods are pairwise disjoint and each contains at least three
vertices.
\end{proof}

\section{The global order law and low-slack rigidity}
\label{sec:order-law}

We now derive the order normalization that drives the rest of the paper.  A
tree is \emph{reduced} if it has no edge whose two endpoints both have degree
two.

\begin{lemma}[Leaf-sensitive domination]
\label{lem:leaf-sensitive}
Every nontrivial tree $H$ satisfies
\[
 3\gamma(H)\le |V(H)|+\ell(H).
\]
\end{lemma}

\begin{proof}
Induct on $|V(H)|$; orders at most four are immediate.  Root $H$ at a leaf
and choose a support vertex $u$ farthest from the root.  All children of $u$
are leaves.  Let $X$ be that child set, put $k=|X|$, and let $v$ be the
parent of $u$.  If $v$ is the root leaf, then $H$ is a star.

If $k\ge2$, delete $X\cup\{u\}$.  The remaining nontrivial tree $H'$ has
order $|H|-k-1$ and at most $\ell(H)-k+1$ leaves.  A dominating set of $H'$,
together with $u$, dominates $H$, so induction gives
\[
 3\gamma(H)\le |H|+\ell(H)-2k+3\le |H|+\ell(H).
\]

Suppose $k=1$, and call the leaf child $x$.  If $\deg_H(v)\ge3$, delete
$x,u$; the order falls by two and the leaf count by one, and adjoining $u$
to a dominating set completes the estimate.  If $\deg_H(v)=2$, either
$H=P_4$ or deleting the path $x-u-v$ leaves a nontrivial tree with order
$|H|-3$ and at most $\ell(H)$ leaves.  Again adjoin $u$ and apply induction.
\end{proof}

The stronger independent-domination form is due to Favaron
\cite{Favaron1992}; the elementary form above is included because its leaf
term is exactly what is needed next.

\begin{lemma}[Deep-vertex domination]
\label{lem:deep-domination}
Every tree $T$ of order at least three satisfies
\[
 3(\gamma(T)-p)\le d.
\]
\end{lemma}

\begin{proof}
Select all penultimate vertices.  In a component $C$ of $T[D]$, every leaf
of $C$ is adjacent in $T$ to a penultimate vertex and is already dominated.
If $|C|\le2$, no additional vertex is needed.  Otherwise delete the leaves
of $C$ and call the remaining tree $K$.  If $K$ is a singleton, select it.
Otherwise every leaf of $K$ has a distinct neighbour among the deleted
leaves of $C$, so $\ell(K)\le\ell(C)$.  By
\Cref{lem:leaf-sensitive}, the non-leaves of $C$ can be dominated with at
most
\[
 \frac{|K|+\ell(K)}3\le\frac{|K|+\ell(C)}3=\frac{|C|}3
\]
vertices.  Summing over the deep components gives a dominating set of size
at most $p+d/3$.
\end{proof}

Equation~\eqref{eq:mu-decomposition} already gives $\mu(T)\ge p$
self-containedly.  We use the following reduced-tree inequality from
\cite[Lemma~12]{RajendraprasadSankaranarayanan2025}:
\begin{equation}
 \gamma(T)\le\mu(T)+\frac{p-1}{3};                   \label{eq:cited-reduced-bound}
\end{equation}

\paragraph{Clean-contraction interface.}
If $T'$ is obtained by contracting a clean three-edge path $a-b-c-d$ with
$\deg_T(b)=\deg_T(c)=2$, then
\begin{equation}
 |T|=|T'|+3,\qquad
 \mu(T)=\mu(T')+1,\qquad
 \gamma(T)\le\gamma(T')+1.                          \label{eq:clean-contraction-input}
\end{equation}

\begin{proof}[Proof of the clean-contraction interface]
Only the spectral identity requires explanation.  Use the exact
leaf-to-root congruence recurrence for $L(T)-I$,
\[
 f(v)=\deg_T(v)-1-\sum_{w\text{ a child of }v}\frac1{f(w)},
\]
interpreted in the extended real line with
\[
 \frac10=+\infty,
 \qquad
 \frac1{\pm\infty}=0.
\]
If a vertex has a zero child, record the parent value as $-\infty$ and
change the child's temporary zero sign to a positive sign.  This is exactly
the usual nonsingular $2\times2$ zero-pivot congruence: the reciprocal of
the resulting infinite parent is zero, so the edge above it contributes
nothing.  Thus this bookkeeping preserves the inertia delivered by the
standard finite algorithm, including when a zero pivot occurs below the
clean path.

Root on the $a$ side of the clean path and orient it as $a,b,c,d$.  Away
from this path the recurrences for $T$ and $T'$ agree.  Put $x=f_T(d)$;
the extended algorithm can produce $x\in\mathbb R\cup\{-\infty\}$.  On the
path,
\[
 f_T(c)=1-\frac1x,\qquad
 f_T(b)=1-\frac1{f_T(c)},\qquad
 1-\frac1{f_T(b)}=x.
\]
The last identity makes the pivot at $a$ equal to the pivot at the contracted
vertex of $T'$.  After the preceding zero-pivot bookkeeping, exactly one of
$d,c,b$ contributes an additional negative sign: it is $d$ when $x<0$
(including $x=-\infty$), $c$ when $0\le x<1$, and $b$ when $x\ge1$.
At $x=0$ or $x=1$ this statement is precisely the $2\times2$ zero-pivot
case just described.  Sylvester's law therefore gives
$\mu(T)=\mu(T')+1$.

For domination, let $S'$ be a minimum dominating set of $T'$ and let $e$ be
the contracted vertex.  If $e\in S'$, replace it by $a,d$.  If $e\notin S'$,
an external neighbour of $a$ or $d$ dominates $e$; add whichever of $b,c$
is at distance three from that neighbour.  In both cases the resulting set
dominates $T$ and has size at most $|S'|+1$.
\end{proof}

Whenever we call $T'$ a \emph{clean quotient} of $T$, we mean precisely a
tree obtained by one such contraction; no uniqueness of the chosen path is
part of the terminology.
Thus \eqref{eq:cited-reduced-bound} is the only imported mathematical
inequality used in the global-law and classification proofs.  Its cited
proof constructs a dominating set in the same bottom-up inertia scan: the
absence of adjacent degree-two vertices supplies the local domination check,
and the weight balance leaves the term $(p-1)/3$.  No equality or
classification statement is imported from that argument.

\begin{theorem}[Sharp global order law]
\label{thm:global-order-law}
Every tree $T$ of order $n\ge2$ satisfies
\begin{equation}
 9\bigl(7\gamma(T)-9\mu(T)\bigr)\le n-20.           \label{eq:global-order-law}
\end{equation}
\end{theorem}

\begin{proof}
First let $T$ be reduced and $n\ge3$.  Put
\[
 h=\gamma-p,\qquad t=\mu-p,\qquad
 p-1=3q+s\quad(s\in\{0,1,2\}),
\]
and
\[
 r=q-(\gamma-\mu)=q-h+t.
\]
These coordinates isolate five different sources of slack.  Their
nonnegativity follows independently:
\begin{center}
\small
\begin{tabularx}{.96\textwidth}{@{}lX@{}}
\toprule
Term & Reason for nonnegativity \\
\midrule
$\ell-p$ & Every penultimate vertex has a leaf neighbour, and distinct
penultimate vertices have distinct leaf neighbours.\\
$d-3h$ & This is the unused part of the deep-vertex domination bound
$3(\gamma-p)\le d$.\\
$s$ & This is the residue in $p-1=3q+s$, with $s\in\{0,1,2\}$.\\
$t$ & The general spectral support bound $\mu\ge p$ gives $t=\mu-p\ge0$.\\
$r$ & The reduced-tree inequality gives the integral bound
$\gamma-\mu\le\lfloor(p-1)/3\rfloor=q$.\\
\bottomrule
\end{tabularx}
\end{center}

The coefficients in the decomposition can now be recomputed without any
inequality.  Starting from the five proposed defects and substituting
$r=q-h+t$ gives
\begin{align*}
 &(\ell-p)+(d-3h)+20s+21t+60r\\
 &\qquad=\ell-p+d-3h+20s+21t+60(q-h+t)\\
 &\qquad=\ell+d-p+60q+20s-63h+81t.
\end{align*}
Since $p=3q+s+1$, the coefficient block satisfies
\[
 -p+60q+20s=19p-20.
\]
Using $n=\ell+p+d$, $\gamma=p+h$, and $\mu=p+t$ therefore yields the exact
five-term identity
\begin{equation}
 \Phi(T):=n-20-9(7\gamma-9\mu)
  = (\ell-p)+(d-3h)+20s+21t+60r.                    \label{eq:five-term-slack}
\end{equation}
Thus $\Phi(T)\ge0$ for every reduced tree.

For a general tree, induct on its order.  The complete base list is
\[
\begin{array}{c|rrrr}
T&\gamma&\mu&7\gamma-9\mu&\Phi\\ \hline
P_2&1&1&-2&0\\
P_3&1&1&-2&1\\
P_4&2&2&-4&20\\
K_{1,3}&1&1&-2&2.
\end{array}
\]
If the tree is not reduced, apply
\eqref{eq:clean-contraction-input}.  The definition of $\Phi$ gives
\begin{align*}
 7\gamma(T)-9\mu(T)
 &\le 7(\gamma(T')+1)-9(\mu(T')+1)\\
 &=7\gamma(T')-9\mu(T')-2,
\end{align*}
and hence
\[
 \Phi(T)\ge (|T'|+3)-20
 -9\bigl(7\gamma(T')-9\mu(T')-2\bigr)
 =\Phi(T')+21,
\]
and the induction hypothesis completes the proof.
\end{proof}

\subsection{Notation guide and why the five terms matter}

For reference, the main coordinates and the local quantities used below are
collected here:
\begin{center}
\small
\begin{tabularx}{.96\textwidth}{@{}lX@{}}
\toprule
Symbol & Meaning \\
\midrule
$n,\ell,p,d$ & Order and the sizes of the leaf, penultimate, and deep sets.\\
$\Phi,a,b$ & Global slack, leaf excess $a=\ell-p$, and deep-domination
surplus $b=d-3h$.\\
$b_v,B_C$ & Boundary incidences at $v\in D$ and their sum
$B_C=\sum_{v\in C}b_v$ on a deep component.\\
$A_C,R_C$ & Positive-boundary vertices of $C$ and the complementary
zero-boundary target set.\\
$h$ & Deep domination contribution $h=\gamma-p$.\\
$q,s$ & Quotient and residue in $p-1=3q+s$, $s\in\{0,1,2\}$.\\
$t$ & Spectral excess $t=\mu-p$.\\
$r$ & Reduced domination--spectrum slack $r=q-(\gamma-\mu)$.\\
$\delta_C,Q_C$ & Partial domination number of the zero-boundary targets in
$C$ and the residual symmetric block on $C$.\\
$z_C,\eta_C,\chi_C$ & Local packing surplus $|C|-3\delta_C$, residual
negative index $n_-(Q_C)$, and excess total degree above three.\\
$X_C,m_C$ & Total-degree-two vertices of $C$ and their number.\\
$\pi_C$ & Local pressure $B_C-3\delta_C-1$, introduced in
\Cref{sec:equality}.\\
$e_{PP}$ & Number of edges with both endpoints in the penultimate set.\\
$\mathcal O,\mathcal K_k$ & The ordinary weighted tile $(2,0,2)$ and the
finite set of ported kernels at slack $k$.\\
\bottomrule
\end{tabularx}
\end{center}

Identity \eqref{eq:five-term-slack} is more informative than the inequality
obtained by discarding its right-hand side.  The first two terms are
geometric: they record excess leaf bundles and unused deep-domination
capacity.  The remaining three are arithmetic or spectral: $s$ records the
congruence class of $p-1$, $t$ records negative spectral directions beyond
those forced by the leaf bundles, and $r$ records the integer gap in the
reduced domination--spectrum inequality.  None of these interpretations is
interchangeable with another; in particular, a finite value of $\Phi$ says
where the defect can occur, not merely how large the total defect is.

The first two positive coefficients also explain the later phase boundary.
A nonreduced tree pays at least $21$ units through contraction.  A reduced
tree with $t>0$ also pays at least $21$, while a nonzero residue $s$ first
costs $20$.  Thus the range $0\le\Phi\le19$ excludes all three mechanisms
before any structural argument is used.  At $\Phi=20$, exactly one of them,
the residue mechanism, becomes available; \Cref{sec:sharpness} shows that it
is realized by an actual subcubic family.

The coefficient separation in \eqref{eq:five-term-slack} is the source of
the exact stability window.

\begin{corollary}[Low-slack rigidity]
\label{cor:low-slack-rigidity}
Let $T$ be a tree of order at least five and suppose
$0\le\Phi(T)\le19$.  Then $T$ is reduced and
\[
 s=t=r=0.
\]
For a nonnegative integer $q$,
\begin{align}
 p=\mu&=3q+1,& \gamma&=4q+1,                         \label{eq:low-slack-parameters}\\
 n&=9q+2+\Phi(T),&7\gamma-9\mu&=q-2,                 \label{eq:low-slack-order}\\
 (\ell-p)+(d-3q)&=\Phi(T).                           \label{eq:leaf-deep-slack}
\end{align}
\end{corollary}

\begin{proof}
If $T$ were not reduced, its contraction $T'$ would have order at least two,
and \Cref{thm:global-order-law} together with the contraction estimate would
give $\Phi(T)\ge21$.  Thus $T$ is reduced.  In
\eqref{eq:five-term-slack}, none of $20s,21t,60r$ can be positive, so
$s=t=r=0$.  Hence $p=3q+1$, $\mu=p$, and
$\gamma-\mu=q$.  The remaining identities follow by substitution.
\end{proof}

\begin{remark}
The threshold nineteen is the largest one obtainable by coefficient
separation alone: at slack twenty the algebraic term $20s$ can be nonzero.
\Cref{sec:sharpness} will show that this new algebraic possibility is
realized by actual subcubic trees.
\end{remark}

\section{Local pressure and matched-skeleton equality trees}
\label{sec:equality}

The equality structure is governed by a local boundary-surplus inequality.
For a tree $C$ and weights $b:V(C)\to\mathbb Z_{\ge0}$, put
\begin{align*}
 A&=\{v:b_v>0\},
 &B&=\sum_{v\in V(C)}b_v,\\
 Q&=L(C)+\diag(b_v-1:v\in V(C)).
\end{align*}

\begin{theorem}[Local boundary surplus]
\label{thm:local-surplus}
Let $C$ have order $s$ divisible by three.  Suppose every leaf of $C$ lies
in $A$,
\[
 \min\{|S|:V(C)\setminus A\subseteq N_C[S]\}=s/3,
\]
and $Q\succeq0$.  Then
\begin{equation}
 B-s\ge s/3.                                          \label{eq:local-surplus}
\end{equation}
Equality holds if and only if $C$ is obtained by subdividing every edge of a
specified perfect matching in a tree $F$, with weight two on the original
vertices of $F$ and weight zero on the subdivision vertices.
\end{theorem}

\begin{proof}
By \Cref{lem:partial-packing}, the zero-weight set contains a
distance-three packing $Z$ of size $s/3$.  No member of $Z$ is a leaf, and
the closed neighbourhoods $N_C[z]$ are pairwise disjoint.  Hence
\[
 s\ge\sum_{z\in Z}|N_C[z]|\ge3|Z|=s.
\]
Equality holds throughout: every $z$ has degree two, and the closed
neighbourhoods partition $V(C)$.  Put $O=V(C)\setminus Z$.  Every vertex of
$O$ has exactly one neighbour in $Z$, and $|O|=2s/3$.

Order $Q$ by $Z,O$.  The $Z$ principal block is the identity.  Suppressing
the vertices of $Z$ produces a tree $F$ on $O$; the suppressed two-edge
paths form a perfect matching of $F$.  If $R$ is the $Z$-by-$O$ incidence
matrix, then every row of $R$ has two ones and every column has one, and
\[
 Q=\begin{pmatrix}
 I&-R\\
 -R^{\mathsf T}&L(C)[O,O]+\diag(b_v-1:v\in O)
 \end{pmatrix}.
\]
Thus the Schur complement of the identity block is
\[
 L(C)[O,O]+\diag(b_v-1:v\in O)-R^{\mathsf T}R
 =L(F)+\diag(b_v-2:v\in O).
\]
It is positive semidefinite, so its all-ones quadratic form gives
\[
 0\le B-2|O|=B-4s/3,
\]
which is \eqref{eq:local-surplus}.

If equality holds, the all-ones vector belongs to the kernel of the Schur
complement.  Its row sums are $b_v-2$, so every vertex of $O$ has weight two
and every vertex of $Z$ has weight zero.  The converse follows by reversing
the suppression: elimination leaves $L(F)$, and the subdivision vertices
form the required packing.
\end{proof}

We now introduce the canonical equality objects.

\begin{definition}[Matched-skeleton lift]
\label{def:matched-lift}
Let $H$ be a subcubic tree on $2q$ vertices with perfect matching $M$.
Construct $W(H)$ by the following rules:
\begin{enumerate}[label=(\roman*)]
  \item subdivide every edge $e$ once, with subdivider $c_e$;
  \item attach one leaf to $c_e$ when $e\notin M$;
  \item at every original vertex $v$, attach $3-\deg_H(v)$ new pendant
        paths of length two.
\end{enumerate}
The support on one of the paths in (iii) is called an \emph{arm support}.
\end{definition}

Subcubic trees with a perfect matching are the classical molecular graphs of
acyclic polyenes \cite{Wang2016Estrada}.  What is specific here is the lift
$W(H)$, its domination--Laplacian equality, and the canonical inverse.

\begin{figure}[t]
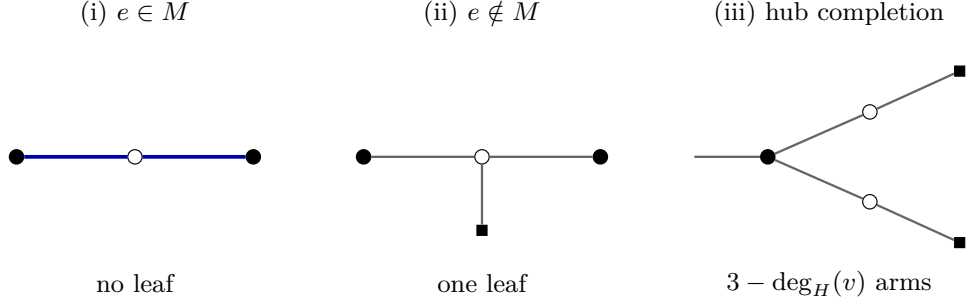

\centering
\resizebox{\textwidth}{!}{\PtwoLiftGrammarGraphic}
\caption{The local grammar of the matched-skeleton lift.  Black circles are
the original vertices of $H$, black squares are leaves, and white circles are
subdividers or arm supports.  A matching subdivider receives no leaf, a nonmatching subdivider
receives one leaf, and every unused degree slot becomes a length-two arm.}
\label{fig:lift-grammar}
\end{figure}

A tree has at most one perfect matching, because the symmetric difference of
two perfect matchings would contain an alternating cycle.  Thus the matching
in \Cref{def:matched-lift} is determined by $H$.

\begin{proposition}[Parameters of the lift]
\label{prop:lift-parameters}
For every $H$ in \Cref{def:matched-lift}, the tree $W(H)$ is subcubic and
\begin{align}
 |V(W(H))|&=9q+2,& \gamma(W(H))&=4q+1,               \label{eq:lift-order-gamma}\\
 \In(L(W(H))-I)&=(5q+1,3q+1,q),& \mu(W(H))&=3q+1.    \label{eq:lift-inertia}
\end{align}
In particular, $\Phi(W(H))=0$.
\end{proposition}

\begin{proof}
The lift has $2q$ original vertices and $2q-1$ edge subdividers.  Exactly
$q-1$ nonmatching subdividers receive leaves.  The number of arms is
\begin{equation}
 \sum_{v\in V(H)}(3-\deg_H(v))
 =3(2q)-2(2q-1)=2q+2.                                \label{eq:arm-count}
\end{equation}
Each arm contributes a support and a leaf, so the order is
$2q+(2q-1)+(q-1)+2(2q+2)=9q+2$.

The penultimate vertices are the $q-1$ nonmatching subdividers and the
$2q+2$ arm supports, so $p=3q+1$.  The deep forest consists of $q$ disjoint
three-vertex paths, one for each matching edge: the endpoints are original
vertices and the middle vertex is the matching subdivider.  Every endpoint
has exactly two penultimate neighbours, whereas each middle vertex has none.
Thus every residual block in \eqref{eq:deep-residual} is
\[
 \begin{pmatrix}2&-1&0\\-1&1&-1\\0&-1&2\end{pmatrix},
\]
with inertia $(2,0,1)$.  The leaf-bundle congruence gives
\eqref{eq:lift-inertia}.  In the domination decomposition, each deep path
has exactly its middle vertex as a target, so $\delta_C=1$ and
$\gamma=p+q=4q+1$.
\end{proof}

The case $q=1$ is $W(P_2)$, precisely the eleven-vertex tree with
$(\gamma,\mu)=(5,4)$ in
\cite[Figure~2]{RajendraprasadSankaranarayanan2025}; thus that example is the
smallest member of this canonical equality family.

\subsection{A complete example and the inverse map}

Take $H=P_4$ with vertices $h_1h_2h_3h_4$ and its unique perfect matching
$M=\{h_1h_2,h_3h_4\}$.  The lift in
\Cref{fig:p4-lift-recovery} contains four hubs and three edge subdividers.
The middle edge $h_2h_3$ is the only nonmatching edge, so its subdivider
receives one leaf.  The arm multiplicities at the four hubs are
\[
 3-\deg_H(h_i)=2,1,1,2,
\]
giving six arm supports and six arm leaves.  Thus the full lift has
\[
 4+3+1+2\cdot6=20
\]
vertices.  Its penultimate set consists of the nonmatching subdivider and
the six arm supports, hence $p=7$.  Its deep forest consists of the two
paths
\[
 h_1-c_{12}-h_2,\qquad h_3-c_{34}-h_4,
\]
each with boundary weights $(2,0,2)$.  The exact decompositions therefore
give
\[
 \gamma(W(P_4))=7+2=9,\qquad
 \mu(W(P_4))=7,\qquad
 \In(L(W(P_4))-I)=(11,7,2).
\]
The score is $7\cdot9-9\cdot7=0$, so $9(7\gamma-9\mu)=0=20-20$.

\begin{figure}[t]
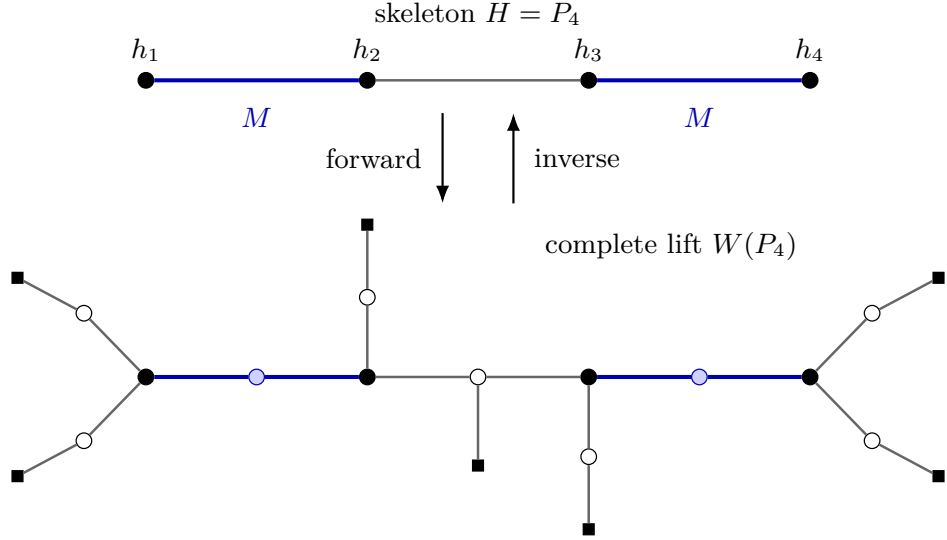

\centering
\resizebox{\textwidth}{!}{\PtwoPfourRecoveryGraphic}
\caption{The $q=2$ example and its inverse.  Colours display the forward
construction only and are ignored by the canonical inverse.  Blue subdividers belong to the
two deep three-vertex components and encode the matching edges.  The white
subdivider with a leaf encodes the unique nonmatching edge; the other white
vertices are arm supports.  The paired arrows summarize the forward lift and
canonical recovery from the uncoloured tree.}
\label{fig:p4-lift-recovery}
\end{figure}

The reverse direction uses no labels from the construction.  Given the
uncoloured tree in the lower panel, first form the graph-invariant sets
$L$, $P=N(L)$, and $D=V\setminus(L\cup P)$.  The two components of $T[D]$
are three-vertex paths; their endpoints are the four hubs and their middle
vertices determine the two matching pairs.  A vertex of $P$ with two hub
neighbours is then suppressed to a nonmatching edge, while a vertex of $P$
with one hub neighbour is deleted together with its leaf.  These operations
recover the displayed $P_4$ and its perfect matching.  The same recipe is
used in the general classification proof below.

\begin{theorem}[Subcubic equality classification]
\label{thm:subcubic-equality}
Let $T$ be a subcubic tree of order at least two.  Then
\[
 9(7\gamma(T)-9\mu(T))=|V(T)|-20
\]
if and only if either $T=P_2$ or $T\cong W(H)$ for a subcubic tree $H$ with
a perfect matching.  In the latter case $H$ is uniquely recoverable up to
isomorphism.
\end{theorem}

\begin{proof}
The tree $P_2$ has $\gamma=\mu=1$.  Now let $T$ be a nontrivial equality
tree.  The base table in the proof of \Cref{thm:global-order-law} gives
$\Phi(P_3)=1$, $\Phi(P_4)=20$, and $\Phi(K_{1,3})=2$, so no tree of order
three or four is an equality tree.  Hence $|V(T)|\ge5$, and
\Cref{cor:low-slack-rigidity} gives, for some $q\ge1$,
\begin{equation}
 \ell=p=\mu=3q+1,\qquad d=3q,\qquad \gamma=4q+1.      \label{eq:equality-tuple}
\end{equation}
Equations \eqref{eq:gamma-decomposition} and
\eqref{eq:mu-decomposition} imply
\begin{equation}
 \sum_C\delta_C=q,\qquad Q_C\succeq0\quad\text{for every }C.             \label{eq:equality-local-data}
\end{equation}
Because a deep vertex is neither a leaf nor adjacent to a leaf,
subcubicity gives
\begin{equation}
 2\le\deg_C(v)+b_v\le3.                              \label{eq:deep-boundary}
\end{equation}

Write $B_C=\sum_{v\in C}b_v$.  If $\delta_C>0$,
\Cref{cor:three-delta} gives $|C|\ge3\delta_C$, while
$Q_C\succeq0$ gives $B_C\ge|C|$ from the all-ones quadratic form.
Equality $B_C=3\delta_C$ would force $Q_C\boldsymbol1=0$ and hence
$b_v=1$ for every $v$, contradicting the existence of a zero-weight target.
If $\delta_C=0$, then $R_C=\varnothing$, so every $b_v\ge1$ and
$B_C\ge |C|\ge1$.  Equality $B_C=1$ would make $C$ a singleton with
$b_v=1$, contradicting $2\le\deg_C(v)+b_v$.  Thus in all
cases
\begin{equation}
 \pi_C:=B_C-3\delta_C-1\ge0.                          \label{eq:local-pressure}
\end{equation}

Let $c$ be the number of deep components and let $e_{PP}$ count edges with
both endpoints in $P$.  The leaf-deleted core is a tree, so
\[
 e_{PP}+\sum_C B_C+(d-c)=p+d-1.
\]
Using \eqref{eq:equality-tuple} and \eqref{eq:equality-local-data},
\begin{equation}
 \sum_C\pi_C=-e_{PP}.                                \label{eq:pressure-collision}
\end{equation}
The left side is nonnegative and the right side nonpositive.  Therefore
$e_{PP}=0$ and $\pi_C=0$ for every component.
In particular, every $\delta_C$ is positive: otherwise pressure equality
would give the already excluded value $B_C=1$.

Fix a component.  Pressure equality gives
$B_C=3\delta_C+1$.  The all-ones bound and
\Cref{cor:three-delta} show that
$3\delta_C\le|C|\le3\delta_C+1$.  The upper alternative would give
$B_C=|C|$ and hence $b_v=1$ for every vertex, contradicting
$\delta_C>0$.  Thus $|C|=3\delta_C$.  Every leaf of $C$ lies in $A_C$ by
\eqref{eq:deep-boundary}, so \Cref{thm:local-surplus} applies and gives
\[
 1=B_C-|C|\ge |C|/3=\delta_C.
\]
Therefore $\delta_C=1$, and equality in the local theorem makes $C$ the
weighted path $(2,0,2)$.  There are exactly $q$ such components.

It remains to recover $H$.  The endpoints of the deep paths become the
vertices of $H$, and the two endpoints of each path are joined by a matching
edge.  Every penultimate vertex has one leaf and, since $e_{PP}=0$, one or
two deep neighbours.  A penultimate vertex with two deep neighbours is
suppressed to a nonmatching edge of $H$; one with one deep neighbour is an
arm support and is deleted together with its leaf.  These graph-invariant
operations produce a connected acyclic subcubic graph with the displayed
perfect matching, and reversing them gives $T=W(H)$.  They also prove
uniqueness.  The converse is \Cref{prop:lift-parameters}.
\end{proof}

\begin{corollary}
\label{cor:subcubic-ratio}
The supremum of $\gamma(T)/\mu(T)$ over subcubic trees equals $4/3$.
\end{corollary}

\begin{proof}
The strict upper bound $\gamma/\mu<4/3$ for trees is proved in
\cite{RajendraprasadSankaranarayanan2025}.  Along any sequence of lifts with
$q\to\infty$,
\[
 \frac{\gamma(W(H))}{\mu(W(H))}
 =\frac{4q+1}{3q+1}\longrightarrow\frac43.
\]
\end{proof}

\section{The exact maximum at every order}
\label{sec:exact-extremal}

The global inequality becomes an exact extremal formula because additional
leaves at an existing support do not change either parameter.

\begin{lemma}[Leaf inflation]
\label{lem:leaf-inflation}
Let $u$ be a vertex of a tree $T$ with a leaf neighbour $v$, and let $T^+$
be obtained by adding a new leaf $w$ adjacent to $u$.  Then
\[
 \gamma(T^+)=\gamma(T),\qquad \mu(T^+)=\mu(T).
\]
The multiplicity of the Laplacian eigenvalue one increases by one.
\end{lemma}

\begin{proof}
A minimum dominating set of $T$ may be chosen to contain $u$: replace $v$
by $u$ if necessary.  It also dominates $w$, so
$\gamma(T^+)\le\gamma(T)$.  Conversely, if a minimum dominating set of
$T^+$ contains $w$, replace $w$ by $v$ when neither $u$ nor $v$ is selected,
and otherwise delete $w$.  The resulting set dominates $T$ and has no larger
cardinality.  Hence the domination numbers agree.

For the spectral statement, order the coordinates of $L(T^+)-I$ as
$(v,w,u,R)$.  Its quadratic form is
\[
 -2x_u(x_v+x_w)+\deg_T(u)x_u^2
   +2x_u b^{\trans}x_R+x_R^{\trans}Cx_R.
\]
The invertible coordinates
\[
 y=x_v+x_w-\frac{x_u}{2},\qquad z=x_v-x_w
\]
turn this form into the quadratic form of $L(T)-I$ on $(y,u,R)$, with no
$z$ term.  Thus
\[
 L(T^+)-I\cong (L(T)-I)\oplus[0].
\]
The negative index is unchanged and the nullity increases by one.
\end{proof}

A related pendant-addition statement appears in
\cite[Lemma~12]{BarikJena2026}.  We include the congruence because the proof
here needs both parameter invariance and the exact increase of the
eigenvalue-one multiplicity.

For a set $X$ of distinct arm supports of $W(H)$, let $W(H;X)$ be the tree
obtained by adding one second leaf at every support in $X$.  Distinctness is
exactly what preserves maximum degree three.

\begin{corollary}[Loaded-lift parameters]
\label{cor:loaded-parameters}
If $H$ has $2q$ vertices and $|X|=k$, then
\begin{align*}
 |V(W(H;X))|&=9q+2+k,\\
 \gamma(W(H;X))&=4q+1,\\
 \mu(W(H;X))&=3q+1,\\
 \Phi(W(H;X))&=k.
\end{align*}
Moreover, the number of available arm supports is exactly $2q+2$.
\end{corollary}

\begin{proof}
Iterate \Cref{lem:leaf-inflation} and use
\Cref{prop:lift-parameters,eq:arm-count}.
\end{proof}

\begin{theorem}[Exact fixed-order maximum]
\label{thm:exact-order-maximum}
For every integer $n\ge2$,
\begin{equation}
 \max_{\substack{T\text{ a tree}\\|V(T)|=n}}
 \bigl(7\gamma(T)-9\mu(T)\bigr)
 =\left\lfloor\frac{n-20}{9}\right\rfloor.          \label{eq:exact-order-maximum}
\end{equation}
For every $n\ge29$, the same maximum is attained by a subcubic tree.
Consequently, order twenty-nine is the first order for which
$7\gamma(T)>9\mu(T)$ is possible.
\end{theorem}

\begin{proof}
The upper bound is \Cref{thm:global-order-law}.  For $2\le n\le10$, the
star $K_{1,n-1}$ has $\gamma=\mu=1$ and score
$-2=\lfloor(n-20)/9\rfloor$.

Let $n\ge11$ and write
\[
 q=\left\lfloor\frac{n-2}{9}\right\rfloor,
 \qquad k=n-(9q+2),\qquad 0\le k\le8.
\]
Take $H=P_{2q}$ and start from $W(H)$.  Adding $k$ leaves at any one existing
support gives an $n$-vertex tree with unchanged parameters and score
\[
 7(4q+1)-9(3q+1)=q-2
  =\left\lfloor\frac{n-20}{9}\right\rfloor.
\]
This proves equality at every order.

If $n\ge29$, then $q\ge3$, and $W(P_{2q})$ has $2q+2\ge8$ arm supports.
Loading $k$ distinct arms gives the same score while preserving maximum
degree three.  Finally, the right side of \eqref{eq:exact-order-maximum} is
nonpositive through order twenty-eight and equals one at order twenty-nine.
\end{proof}

\begin{figure}[t]
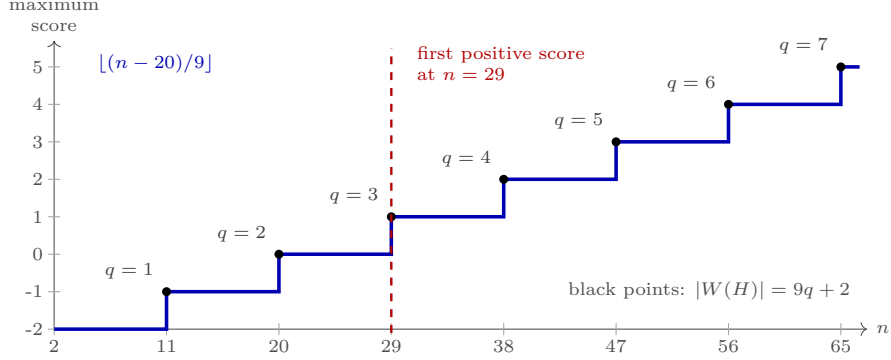

  \centering
  \resizebox{.94\textwidth}{!}{\PtwoOrderStaircaseGraphic}
  \caption{The exact fixed-order maximum as a staircase.  Each black point
  is the unloaded canonical lift order $9q+2$; leaf inflation fills the
  following plateau without changing $\gamma$ or $\mu$.  The dashed line
  marks the first positive score at order twenty-nine.}
  \label{fig:order-staircase}
\end{figure}

\section{Subcubic stability through slack nineteen}
\label{sec:stability}

We now prove that the equality representation persists through slack
nineteen.  The result is stronger than a classification
of maximizers: it controls trees one and two integer score levels below the
maximum as well.

The maximum-degree-three hypothesis is essential throughout
\Cref{sec:stability,sec:sharpness,sec:slack-twenty-one}.  It is what turns
the residual boundary condition into
$2\le\deg_C(v)+b_v\le3$ and limits each hub to $3-\deg_H(v)$ arms.  No
classification in these sections is asserted for trees of maximum degree at
least four.

The key point is that the global slack in this range has already lost its
spectral and residue terms, but it has not yet been shown where the remaining
leaf/deep slack sits.  The local pressure argument below rules out hidden
deep slack.  It does so by assigning a nonnegative pressure to every deep
component and then observing that the edge count of the leaf-deleted core
forces the sum of those pressures to be nonpositive.

\begin{theorem}[Nineteen-layer stability]
\label{thm:nineteen-stability}
Let $T$ be a subcubic tree of order at least two.  If
\[
 0\le\Phi(T)\le19,
\]
then either
\[
 (T,\Phi(T))\in\{(P_2,0),(P_3,1),(K_{1,3},2)\},
\]
or $|V(T)|\ge5$ and there is a unique triple $(q,H,X)$, up to isomorphism
of $(H,X)$, where $q\ge1$, $H$ is a subcubic tree on $2q$ vertices with a
perfect matching, and $X$ is a set of distinct arm supports of $W(H)$, such
that
\begin{equation}
 T\cong W(H;X),\qquad |X|=\Phi(T).                    \label{eq:stability-representation}
\end{equation}
The alternatives are disjoint.  Conversely, the three displayed small trees
have their stated slacks, and every loaded lift with $|X|\le19$ satisfies
\eqref{eq:stability-representation}.
\end{theorem}

\begin{proof}
If $|V(T)|\le4$, the base table in the proof of
\Cref{thm:global-order-law} gives exactly the three displayed cases (the
remaining tree $P_4$ has slack twenty).  Hence assume $|V(T)|\ge5$.
Put $j=\Phi(T)$.  By \Cref{cor:low-slack-rigidity},
\begin{equation}
 p=\mu=3q+1,\quad \gamma=4q+1,\quad
 n=9q+2+j,\quad (\ell-p)+(d-3q)=j.                  \label{eq:stability-parameters}
\end{equation}
If $q=0$, then $\gamma=1$, so $T$ is a star.  A star of order at least five
is not subcubic; hence $q\ge1$.

The exact decompositions from \Cref{sec:preliminaries} give
\begin{equation}
 \sum_C\delta_C=q,\qquad Q_C\succeq0\quad\text{for every }C.             \label{eq:stability-local-data}
\end{equation}
The subcubic boundary \eqref{eq:deep-boundary} remains valid.  Define
$B_C=\sum_{v\in C}b_v$ and
\[
 \pi_C=B_C-3\delta_C-1.
\]
The proof of \eqref{eq:local-pressure} used only
\Cref{cor:three-delta}, positive semidefiniteness, and the boundary
condition, so $\pi_C\ge0$ in the present setting as well.

Here $3\delta_C$ is the boundary budget forced by a maximum
distance-three packing, while the extra unit is forced by integrality.  If
$B_C=3\delta_C$, the all-ones vector would lie in the kernel of $Q_C$ and
would make every boundary weight equal to one, contradicting the existence
of a zero-weight packed target.  Thus $\pi_C$ measures genuine local excess,
not an artefact of summing several components.

Let $c$ be the number of deep components and let $e_{PP}$ count edges inside
the penultimate set.  The leaf-deleted core identity now gives
\[
 \sum_C B_C=p+c-1-e_{PP}=3q+c-e_{PP}.
\]
Together with \eqref{eq:stability-local-data}, this yields the same sign
collision as in the equality case:
\begin{equation}
 \sum_C\pi_C=-e_{PP}.                                \label{eq:stability-collision}
\end{equation}
Therefore $e_{PP}=0$ and $\pi_C=0$ for every component.
In particular, no component has $\delta_C=0$, since pressure equality would
then give the excluded value $B_C=1$.

The collision in \eqref{eq:stability-collision} is exact.  Its left-hand
side is a sum of independently nonnegative local quantities; its right-hand
side is the negative of an edge count.  Hence no cancellation between a
``good'' and a ``bad'' component is possible.  Every component must achieve
local pressure equality and the penultimate set must be independent.  This
is the step that converts a numerical low-slack statement into a
componentwise classification.

For completeness, we repeat the consequence of local pressure equality.
The identity $B_C=3\delta_C+1$, the packing bound, and the all-ones
positive-semidefinite bound give
\[
 3\delta_C\le |C|\le3\delta_C+1.
\]
The upper alternative forces every $b_v=1$ and hence $\delta_C=0$, a
contradiction.  Thus $|C|=3\delta_C$.  Every leaf of $C$ has positive weight,
so \Cref{thm:local-surplus} applies and gives
\[
 1=B_C-|C|\ge|C|/3=\delta_C.
\]
Consequently $\delta_C=1$, and $C$ is the weighted path $(2,0,2)$.  Since
the sum of all $\delta_C$ is $q$, there are $q$ components and
\begin{equation}
 d=3q.                                                \label{eq:no-deep-slack}
\end{equation}
Substitution into \eqref{eq:stability-parameters} gives
\begin{equation}
 \ell-p=j.                                           \label{eq:all-slack-is-leaf}
\end{equation}

Let $\lambda_u$ be the number of leaf neighbours of $u\in P$.  Every
$\lambda_u\ge1$.  The leaf-deleted core is connected and nontrivial, so
every support has a core neighbour.  Subcubicity gives $\lambda_u\le2$,
and a support with two leaves has exactly one core neighbour.  Therefore
\[
 j=\ell-p=\sum_{u\in P}(\lambda_u-1)
\]
counts exactly the supports with two leaves.

Delete one leaf at each of those supports and call the resulting tree
$T_0$.  Every affected support retains one leaf, so the penultimate and deep
sets, all partial-domination targets, and every residual matrix remain
unchanged.  Thus $\gamma(T_0)=4q+1$ and $\mu(T_0)=3q+1$, while
$|T_0|=9q+2$.  By \Cref{thm:subcubic-equality},
$T_0\cong W(H)$ for a unique subcubic perfect-matching skeleton $H$.
The supports from which leaves were deleted have one core neighbour in
$T_0$ and are therefore exactly arm supports.  Their set is $X$, and
$|X|=j$, proving \eqref{eq:stability-representation} and canonical recovery.

Conversely, \Cref{cor:loaded-parameters} shows directly that every defined
$W(H;X)$ has $\Phi=|X|$ and the stated parameters.
\end{proof}

The stability theorem has an exact capacity statement.

\begin{corollary}[Exact layer capacity]
\label{cor:stability-capacity}
For integers $q\ge1$ and $0\le j\le19$, a subcubic tree in the
$(q,j)$ stability layer exists if and only if
\[
 j\le2q+2.
\]
\end{corollary}

\begin{proof}
The necessity is the arm count \eqref{eq:arm-count}.  For sufficiency, take
$H=P_{2q}$ and load any $j$ of its $2q+2$ arms.
\end{proof}

\begin{corollary}[All subcubic maximizers]
\label{cor:all-subcubic-maximizers}
Let $q\ge1$ and $0\le k\le8$.  A subcubic tree of order $9q+2+k$ attains
the all-tree maximum in \eqref{eq:exact-order-maximum} if and only if it is
$W(H;X)$ for a subcubic perfect-matching tree $H$ on $2q$ vertices and a
$k$-element set of distinct arm supports.  Such a tree exists exactly when
$k\le2q+2$.  In particular, all nine residue layers occur for every
$q\ge3$.
\end{corollary}

\begin{proof}
At order $9q+2+k$, the maximum score is $q-2$, so a maximizer has
$\Phi=k$.  Apply \Cref{thm:nineteen-stability,cor:stability-capacity}.  The
converse follows from \Cref{cor:loaded-parameters}.
\end{proof}

More generally, a tree in the $j$-slack layer has score $q-2$ and order
$9q+2+j$, while the all-tree maximum at that order is
\[
 q-2+\left\lfloor\frac{j}{9}\right\rfloor.
\]
Thus the layers $0\le j\le8$ are extremal, $9\le j\le17$ are one integer
score level below extremal, and $18\le j\le19$ are two levels below.

\section{Complete classification at the structural phase boundary}
\label{sec:sharpness}

At slack twenty, the residue term $20s$ in
\eqref{eq:five-term-slack} can become active.  The pressure identity then
has one unit on its right-hand side, and that single unit permits exactly two
non-loaded operations.  This section remains entirely within the subcubic
class; the degree-three boundary is used both in the local-surplus exclusion
and in recovering the two endpoint-incidence signatures.

\begin{definition}[Two boundary defects]
\label{def:twenty-defects}
Let $H$ be a subcubic tree on $2q$ vertices with a perfect matching.
\begin{enumerate}[label=(\roman*)]
  \item If $x$ is an arm support of $W(H)$, the \emph{terminal-tail defect}
        is obtained by adjoining new vertices $u,v$ and the path $x-u-v$.
  \item If $z$ is the leafed support subdividing a nonmatching edge, let
        $a,b$ be its deep neighbours and $w$ its leaf.  The
        \emph{bridge-split defect} deletes $z,w$ and replaces them by
        \[
          a-x-y-b,\qquad x-w_x,\qquad y-w_y,
        \]
        where $x,y,w_x,w_y$ are new.
\end{enumerate}
\end{definition}

\begin{figure}[t]
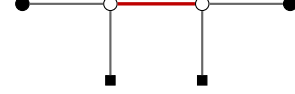

\centering
\resizebox{\textwidth}{!}{\PtwoBoundaryGraphic}
\caption{The non-small structural families at slack twenty; $P_4$ is the
separate small boundary case.  Twenty distinct arm loads
continue the independent-penultimate structure.  The two non-loaded operations
create one edge in $T[P]$; the numbers of deep neighbours at its endpoints
are $\{0,1\}$ for a terminal tail and $\{1,1\}$ for a bridge split.  Red
denotes this unique support--support edge; black squares are leaves.  Panel
\textup{(a)} depicts one of the twenty distinct loaded arms.}
\label{fig:stability-boundary}
\end{figure}

\begin{theorem}[Complete slack-twenty classification]
\label{thm:slack-twenty-sharpness}
Let $T$ be a subcubic tree of order at least two.  Then $\Phi(T)=20$ if and only if either
$T\cong P_4$, or $|V(T)|\ge5$ and exactly one of the following holds.
\begin{enumerate}[label=(\Alph*)]
  \item $T\cong W(H;X)$, where $H$ is a subcubic perfect-matching tree on
        $2q$ vertices, $q\ge9$, and $X$ consists of twenty distinct arm
        supports.
  \item $T$ is a terminal-tail defect of $W(H)$ for some $q\ge1$.
  \item $T$ is a bridge-split defect of $W(H)$ for some $q\ge2$.
\end{enumerate}
In cases \textup{(B)}--\textup{(C)},
\begin{equation}
 |T|=9q+4,\qquad \gamma(T)=4q+2,\qquad
 \mu(T)=3q+2.                                        \label{eq:twenty-defect-parameters}
\end{equation}
The first orders of \textup{(A)}, \textup{(B)}, and \textup{(C)} are
$103$, $13$, and $22$, respectively.  Every slack-twenty order is congruent
to four modulo nine.  The small case $P_4$ and the three displayed families
are mutually disjoint.
\end{theorem}

\begin{proof}
For $|V(T)|\le4$, the base table in the proof of
\Cref{thm:global-order-law} shows that slack twenty occurs exactly for
$P_4$.  Hence assume $|V(T)|\ge5$.
Suppose first that $\Phi(T)=20$.  A non-reduced tree in the present domain
has slack at least twenty-one by the clean-contraction step, so $T$ is
reduced.  In the notation of \eqref{eq:five-term-slack}, write
\[
 a=\ell-p,\qquad b=d-3h.
\]
The five-term identity becomes
\begin{equation}
 20=a+b+20s+21t+60r.                                 \label{eq:twenty-split}
\end{equation}
All terms are nonnegative integers.  Hence $t=r=0$ and exactly one of
\begin{equation}
 s=0,\quad a+b=20,
 \qquad\text{or}\qquad
 s=1,\quad a=b=0                                    \label{eq:twenty-branches}
\end{equation}
holds.  In both branches
\begin{equation}
 \mu=p,\qquad h=q,\qquad \gamma=p+q.                 \label{eq:twenty-common}
\end{equation}
The case $q=0$ is excluded: in the first branch it would give a subcubic
star of order at least five, and in the second it gives order four.

We use the local data from \Cref{sec:preliminaries}.  Equations
\eqref{eq:twenty-common} and the exact domination and inertia decompositions
give
\begin{equation}
 \sum_C\delta_C=q,\qquad Q_C\succeq0.                \label{eq:twenty-local-data}
\end{equation}
Define $B_C=\sum_{v\in C}b_v$ and
\[
 \pi_C=B_C-3\delta_C-1.
\]
The packing, all-ones, and boundary argument used in
\eqref{eq:local-pressure} depends only on
\eqref{eq:twenty-local-data} and \eqref{eq:deep-boundary}; therefore
$\pi_C\ge0$.  If $c$ is the number of deep components and $e_{PP}$ counts
edges in $T[P]$, the leaf-deleted core gives
\[
 \sum_C B_C=p+c-1-e_{PP}.
\]
Since $p-1=3q+s$, we obtain the boundary pressure identity
\begin{equation}
 \sum_C\pi_C=s-e_{PP}.                               \label{eq:twenty-pressure}
\end{equation}

Suppose first that $s=0$.  The two sides of
\eqref{eq:twenty-pressure} have opposite signs, so $e_{PP}=0$ and every
$\pi_C=0$.  The local equality argument in the proof of
\Cref{thm:nineteen-stability} then makes every component the weighted path
$(2,0,2)$, whence $d=3q$.  Thus $\ell-p=20$.  Every penultimate vertex has
one or two leaves and a core neighbour, so exactly twenty distinct supports
have a second leaf.  Deleting those leaves preserves all $\delta_C$ and
$Q_C$ and produces an equality tree $W(H)$.  The affected vertices are arm
supports, proving (A).  The capacity $2q+2\ge20$ is equivalent to $q\ge9$.

It remains to analyze $s=1$.  Now
\begin{equation}
 p=\ell=3q+2,\qquad d=3q.                             \label{eq:twenty-residue-one}
\end{equation}
Equation \eqref{eq:twenty-pressure} permits $e_{PP}=0$ or $1$.  We show
that the zero-edge alternative is impossible.  If $e_{PP}=0$, exactly one
component $C_*$ has $\pi_{C_*}=1$.  From
\[
 \sum_C|C|=3q=3\sum_C\delta_C
\]
and the packing bound $|C|\ge3\delta_C$ for positive cover number, there are
no zero-cover components and equality holds componentwise.  Hence for the
exceptional component
\begin{equation}
 |C_*|=3\delta_{C_*},\qquad B_{C_*}-|C_*|=2.          \label{eq:exceptional-two}
\end{equation}
The local surplus theorem gives $\delta_{C_*}\le2$.

If $\delta_{C_*}=1$, then $C_*=P_3$.  The subcubic boundary and
$B_{C_*}=5$ force weights $(2,1,2)$, but then the target set is empty and
the cover number is zero.  If $\delta_{C_*}=2$, equality holds in
\Cref{thm:local-surplus}.  The only four-vertex tree with a perfect matching
is $P_4$, so its matching subdivision is the weighted path
$(2,0,2,2,0,2)$.  Its two adjacent internal weight-two vertices have total
degree four, again impossible.  Thus
\begin{equation}
 e_{PP}=1,\qquad \pi_C=0\quad\text{for every }C.       \label{eq:one-pp-edge}
\end{equation}
Every deep component is consequently the regular weighted path $(2,0,2)$.

Let $xy$ be the unique edge of $T[P]$, and put
\[
 \kappa(z)=|N_T(z)\cap D|.
\]
Each penultimate vertex has exactly one leaf.  Since $T$ is subcubic,
$\kappa(x),\kappa(y)\in\{0,1\}$; connectivity excludes $\{0,0\}$.

If the incidence multiset is $\{0,1\}$, delete the zero-incidence endpoint
and its leaf.  The other endpoint becomes an arm support; the deep data are
unchanged, and the resulting equality tree is $W(H)$.  Reversing this
operation is the terminal tail in (B).

If the multiset is $\{1,1\}$, let $a,b$ be the two deep neighbours.  They
lie in different deep components, since otherwise the deep path between
them together with $a-x-y-b$ would form a cycle.  Replace $x,y$ and their
leaves by one vertex $z$ adjacent to $a,b$ and to one leaf.  This preserves
every deep boundary weight and residual matrix and produces an equality
tree $W(H)$.  The new $z$ is a nonmatching-edge support; reversing the
operation is the bridge split in (C).  Such a support exists precisely when
$q\ge2$.

Conversely, twenty distinct arm loads preserve $\gamma$ and $\mu$ and add
twenty vertices, so (A) has slack twenty.  A terminal tail or bridge split
increases the order by two and the penultimate count by one while preserving
the deep targets and residual matrices.  Thus both defects satisfy
\eqref{eq:twenty-defect-parameters}; direct substitution gives
\[
 7\gamma-9\mu=q-4,\qquad \Phi=20.
\]
The base table already verifies $\Phi(P_4)=20$.  All three operations
preserve maximum degree three.  Finally, $T[P]$ is
edgeless in (A) and has one edge in (B)--(C), while its endpoint incidence
multiset distinguishes the latter two.  This proves the exhaustive and
disjoint classification.
\end{proof}

The theorem strengthens sharpness into a boundary census.  The continued
loading mechanism cannot appear before $q=9$, whereas the residue-one
mechanism begins immediately at $q=1$.  The bridge split first appears at
$q=2$, when the perfect-matching skeleton has its first nonmatching edge.
Thus the coefficient twenty marks both the first nonzero algebraic residue and
the exact point at which two graph-invariant local defects enter the theory.

\section{The first post-boundary layer}
\label{sec:slack-twenty-one}

At slack twenty-one, clean contraction and one additional negative residual
direction first become compatible with the slack identity.  We first formalize the two
exceptional reduced incidence types.  As in the preceding stability results,
the subcubic bound is indispensable: it restricts every weighted deep degree
to two or three and is the hypothesis that makes the marked Smith reduction
finite.

\begin{definition}[Exceptional incidence assemblies]
\label{def:exceptional-assemblies}
Regard each unit of a deep boundary weight $b_v$ as a labelled stub at $v$.
Fix $q\ge2$.
\begin{enumerate}[label=(\roman*)]
  \item The E0 deep forest consists of one weight-two singleton and $q$
        weighted paths $(2,0,2)$.
  \item The E1 deep forest consists of one weighted claw with centre weight
        zero and leaf weights $1,2,2$, together with $q-1$ weighted paths
        $(2,0,2)$.
\end{enumerate}
In either case introduce a set $P$ of $3q+2$ support vertices and attach one
leaf to every member of $P$.  Join deep stubs to $P$ subject to the following
rules: every stub is used exactly once; every support has incidence degree one
or two; after each deep component is contracted, the resulting simple
bipartite incidence graph on the component vertices and $P$ is a tree; and
every stub at the unique stub-bearing deep vertex of total degree two meets an
incidence-degree-two support.  No other edges are present.  The resulting
leafed tree is an \emph{E0 assembly} or \emph{E1 assembly}, respectively.
\end{definition}

The last rule is exactly reducedness.  The other total-degree-two deep
vertices are the stubless middles of ordinary $(2,0,2)$ paths and meet only
total-degree-three deep vertices.  A support of incidence degree one has total
degree two, so an adjacent deep vertex must have total degree three; at the
unique stub-bearing deficit vertex this is ensured precisely by the last
rule.

Both grammars exist for every $q\ge2$.  At $q=2$, place the E0 singleton
between two ordinary components on a component path and realize the two path
edges by incidence-degree-two supports.  For E1, place the claw at one end of
a two-component path and use its weight-one leaf stub on the connector.
Terminate every remaining stub at a distinct incidence-degree-one support.
To pass from $q$ to $q+1$, choose a terminal support at the far ordinary
component, attach it to one stub of a new ordinary component so that it
becomes a connector, and terminate the other three new stubs separately.
This adds one component and three supports and preserves every rule.

\begin{figure}[t]
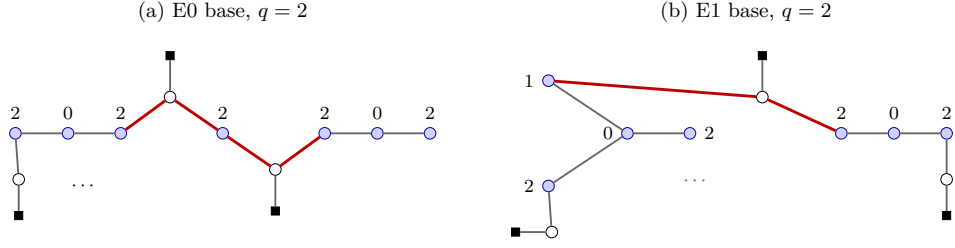

\centering
\resizebox{\textwidth}{!}{\PtwoExceptionalAssembliesGraphic}
\caption{Contracted incidence schematics for the $q=2$ bases of E0 and E1.
Blue vertices display the weighted deep components; white vertices are
supports and each has one black-square leaf.  Each red two-edge chain passes through
a connector support of incidence degree two.  Representative terminal
supports are shown; every
omitted stub terminates at its own incidence-degree-one support.}
\label{fig:exceptional-assemblies}
\end{figure}

\begin{lemma}[Marked Smith reduction]
\label{lem:marked-smith}
Let $C$ be a tree with nonnegative integral boundary weights $b_v$.  Suppose
that
\[
 |C|=3\delta+1,\qquad \sum_{v\in C}b_v=|C|+1,
 \qquad 2\le \deg_C(v)+b_v\le3,
\]
that the zero-weight target set has cover number $\delta$, and that
\[
 Q_C=\diag\bigl(\deg_C(v)+b_v-1:v\in C\bigr)-A(C)\succeq0.
\]
Then either $C$ is a weight-two singleton and $\delta=0$, or $C$ is a
weighted claw whose centre has weight zero and whose leaves have weights
$1,2,2$, and $\delta=1$.
\end{lemma}

\begin{proof}
The degree sum of $C$ gives
\[
 \sum_{v\in C}\bigl(3-\deg_C(v)-b_v\bigr)
 =3|C|-(2|C|-2)-(|C|+1)=1.
\]
Thus there is a unique vertex $x$ with
$\deg_C(x)+b_x=2$; every other vertex has total degree three, and hence
\begin{equation}
 Q_C=2I-A(C)-e_xe_x^{\mathsf T}.                   \label{eq:marked-smith-q}
\end{equation}
Here $e_x$ denotes the standard basis vector indexed by $x$.
Adjoin two new leaves at $x$ and call the resulting tree $G$.  Eliminating
the two leading diagonal entries of $2I-A(G)$ gives the Schur complement in
\eqref{eq:marked-smith-q}.  Therefore
\[
 Q_C\succeq0 \quad\Longleftrightarrow\quad \rho(A(G))\le2.
\]

We now use Smith's published classification of connected graphs of
adjacency index at most two \cite{Smith1970}, in the modern displayed form
of \cite[Section~3.4 and Appendix Table~A1.1]{CvetkovicRowlinsonSimic2004}.
Because $G$ is a tree, Smith's cycle family is absent; only the finite
$A,D,E$ and affine $D,E$ trees displayed below remain.
The marking supplies a concrete filter that we record explicitly: the
vertex $x$ must be incident with at least two leaves in $G$.  In the path
family this happens only at the middle of $P_3$; in a finite $D$ tree it is
the unique fork; in an affine $D$ tree it is one of the two terminal forks;
in $\widetilde D_4=K_{1,4}$ it is the centre; and none of the six exceptional
$E$ trees has such a fork.  Deleting the marked pair therefore gives exactly
the cases in the table and in \Cref{fig:marked-smith-filter}.  The figure
records the root left by the deletion and the weights forced by
$b_v=3-\deg_C(v)$ for $v\ne x$ and $b_x=2-\deg_C(x)$.
In the standard arm descriptions of $E_6,E_7,E_8$ and their affine
extensions, every trivalent vertex has at most one pendant arm of length
one; hence no vertex can carry the two deleted leaves required at $x$.
\begin{center}
\small
\begin{tabularx}{.96\textwidth}{@{}l>{\raggedright\arraybackslash}X>{\raggedright\arraybackslash}X@{}}
\toprule
Smith family for $G$ & Admissible position of $x$ and deletion & Consequence \\
\midrule
$A_n=P_n$ & Only the middle of $P_3$; deletion leaves a rooted singleton & Only $G=P_3$ survives,
giving the weight-two singleton.\\
$D_n$ (finite snake) & The unique fork; deletion leaves an endpoint-rooted path & All boundary weights are
positive; $\delta=0$ forces $|C|=1$.\\
$\widetilde D_n$ (double snake) & Either terminal fork; deletion leaves a broom with one fork & One zero-weight
fork gives $\delta=1$, $|C|=4$, and weights $0;1,2,2$.\\
$\widetilde D_4=K_{1,4}$ & The centre; deletion leaves a three-vertex path rooted at its middle & Its order is
incompatible with $|C|=3\delta+1$.\\
\shortstack{$E_6,E_7,E_8,$\\$\widetilde E_6,\widetilde E_7,\widetilde E_8$} & No vertex is
incident with two leaves & The marked-pair condition excludes them.\\
\bottomrule
\end{tabularx}
\end{center}
\begin{figure}[t]
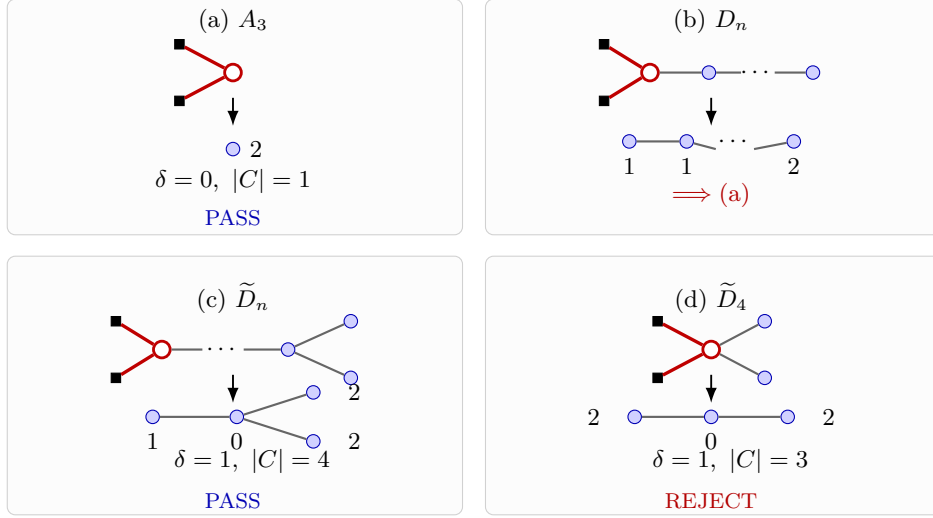

  \centering
  \resizebox{\textwidth}{!}{\PtwoMarkedSmithGraphic}
  \caption{The marked Smith filter.  Each panel deletes the two black-square
  leaves at the red-ringed root.  The lower labels are the forced boundary
  weights of $C$; the terminal annotation records the test outcome.
  The finite $D$ family reduces to (a); the exceptional $E$ types fail the
  marking condition.}
  \label{fig:marked-smith-filter}
\end{figure}
\begin{enumerate}[label=(\roman*)]
  \item The path case is $G=P_3$, so deleting the marked pair leaves a
        singleton $C$ of weight two.
  \item In the finite $D$-type (snake) case, deleting the marked pair leaves
        a path rooted at an endpoint.  The total-degree equations make every
        boundary weight positive.  Thus $\delta=0$, and
        $|C|=3\delta+1$ again leaves only the singleton.
  \item In the affine $D$-type (double-snake) case, deletion leaves a broom
        with one remaining fork.  That fork is the only zero-weight vertex,
        so $\delta=1$ and $|C|=4$.  Hence the marked root must be adjacent to
        the remaining fork; any longer connecting path would make $|C|>4$.
        The weights forced by total degree are $0$ at the centre and
        $1,2,2$ at the three leaves.
  \item The degree-four boundary $G=K_{1,4}$ leaves a three-vertex path
        rooted at its middle vertex.  Its order is incompatible with
        $|C|=3\delta+1$.
\end{enumerate}
The six exceptional Smith trees have no branch vertex incident with two
leaves in this marked configuration.  Direct substitution verifies the two
surviving weighted trees, which proves the lemma.
\end{proof}

\begin{lemma}[Small-diameter matching trees]
\label{lem:small-diameter-matching}
If a subcubic tree $F$ has a perfect matching and $|F|\ge8$, then
$\operatorname{diam}(F)\ge5$; consequently $F$ contains an induced $P_6$.
\end{lemma}

\begin{proof}
For diameter at most three, $F$ is a star or a double star.  Every leaf must
be matched to its support, so each central vertex supports at most one leaf;
the largest possibility is $P_4$.  For diameter four, root $F$ at its unique
central vertex $c$.  The matching pairs $c$ with one neighbour, and that
neighbour can have no child.  Each of the at most two remaining branches has
at most one depth-two child, since all its leaves must be matched to the same
branch vertex.  Hence $|F|\le1+1+2\cdot2=6$.  Thus order at least eight
forces diameter at least five, and a geodesic supplies an induced $P_6$.
\end{proof}

\begin{theorem}[Complete slack-twenty-one classification]
\label{thm:slack-twenty-one}
Let $T$ be a subcubic tree of order at least two.  Then $\Phi(T)=21$ if and only if exactly one of
the following holds.
\begin{enumerate}[label=(\Alph*)]
  \item $T$ is nonreduced and some (equivalently, every) clean quotient $T'$
        satisfies
        \[
          \Phi(T')=0,\qquad \gamma(T)=\gamma(T')+1,
          \qquad \mu(T)=\mu(T')+1.
        \]
  \item $T=W(H;X)$, where $H$ has $2q$ vertices, $q\ge10$, and
        $|X|=21$.  This family first occurs at order $113$.
  \item Either $T$ is the loaded-$P_4$ order-five boundary, or $T$ is
        obtained by adding a second leaf at any degree-two penultimate vertex
        of a terminal-tail defect with $q\ge1$ or a bridge-split defect with
        $q\ge2$.
  \item $T$ is an E0 or E1 assembly with $q\ge2$.  These two types first
        occur at order $23$.
\end{enumerate}
In \textup{(C)}--\textup{(D)}, with $q=0$ for the loaded-$P_4$ boundary,
\[
 |T|=9q+5,\qquad \gamma(T)=4q+2,
 \qquad \mu(T)=3q+2.
\]
The four branches are disjoint, and no reduced spectral-excess type occurs.
\end{theorem}

\begin{proof}
The base table in the proof of \Cref{thm:global-order-law} has no
slack-twenty-one tree of order at most four.  Hence any tree under
consideration has order at least five.
Assume first that $\Phi(T)=21$.  If $T$ is nonreduced, then for any clean
quotient write
$k=\gamma(T')+1-\gamma(T)\ge0$.  The exact shifts under contraction give
\[
 \Phi(T)=\Phi(T')+21+63k.
\]
Since $\Phi(T')\ge0$, slack twenty-one forces $\Phi(T')=k=0$, which is
\textup{(A)}.  Henceforth assume that $T$ is reduced.

Use the coordinates of \eqref{eq:five-term-slack}.  The integer equation
\[
 21=a+b+20s+21t+60r
\]
has exactly three solutions:
\[
 \begin{array}{c|cccc}
 &s&t&a+b&r\\ \hline
 R0&0&0&21&0\\
 R1&1&0&1&0\\
 R2&0&1&0&0.
 \end{array}
\]
In $R0$ and $R1$, the exact decompositions give
$\sum_C\delta_C=q$ and $Q_C\succeq0$.  Thus the local pressure
$\pi_C=B_C-3\delta_C-1$ is nonnegative, while the leaf-deleted core gives
\begin{equation}
 \sum_C\pi_C=s-e_{PP}.                              \label{eq:twenty-one-pressure}
\end{equation}

In $R0$, \eqref{eq:twenty-one-pressure} forces $e_{PP}=0$ and zero
pressure in every component.  The equality argument from
\Cref{thm:nineteen-stability} gives $q$ weighted paths $(2,0,2)$ and
$a=21$.  If $q=0$, then $p=1$, $d=0$, and $\ell=22$, so $T$ would be a
non-subcubic star.  Thus $q\ge1$.  Deleting the twenty-one second leaves
recovers $W(H)$, and the arm
capacity $2q+2\ge21$ is equivalent to $q\ge10$.

In $R1$, first suppose $e_{PP}=1$.  Every component again has zero pressure,
so $(a,b)=(1,0)$.  Deleting the unique excess leaf produces a slack-twenty
tree; equivalently, $T$ has a unique double-leaf support and we delete either
one of its two indistinguishable leaves.  The two leaves have the same open
neighbourhood, so exchanging them is an automorphism and either deletion
produces the same unlabelled tree.  For $q\ge1$, the resulting tree is
reduced: if the altered support meets another penultimate vertex, connectivity
forces that other endpoint to have a deep neighbour and hence degree three;
if it meets a deep vertex, zero-pressure rigidity makes that vertex an
endpoint of a $(2,0,2)$ path with total degree three.  No other degree
changes.  Thus \Cref{thm:slack-twenty-sharpness} identifies the base as either a
terminal-tail or a bridge-split defect.  The double-leaf support had degree
two before the deleted leaf was restored.  Conversely, a second leaf may be
added at any degree-two penultimate vertex of either defect: leaf inflation
preserves $\gamma$ and $\mu$, raises $\Phi$ by one, and preserves
subcubicity and reducedness.  In a terminal tail this includes both the
degree-two endpoint of its $P$--$P$ edge and every untouched ordinary arm;
in a bridge split the two $P$--$P$ endpoints have degree three, but every
untouched ordinary arm remains available.  Their respective endpoint
deep-incidence signatures $\{0,1\}$ and $\{1,1\}$ also prove that the two
subfamilies are disjoint.  At $q=0$, deletion gives $P_4$ and restoration
gives the unique reduced order-five boundary tree.  This proves (C).

Suppose next that $e_{PP}=0$.  One component $E$ has pressure one and all
others are ordinary weighted paths.  The alternative $(a,b)=(1,0)$ would
force $|E|=3\delta_E$ and $B_E=3\delta_E+2$.  The local surplus theorem
leaves $\delta_E\in\{1,2\}$; these give respectively the target-free path
$(2,1,2)$ and a subdivided $P_4$ with adjacent total-degree-four vertices,
both impossible.  Hence $(a,b)=(0,1)$ and
\begin{equation}
 |E|=3\delta_E+1,\qquad B_E=|E|+1.                  \label{eq:e37-exceptional}
\end{equation}
The hypotheses of \Cref{lem:marked-smith} now hold for $E$, so $E$ is either
the weight-two singleton ($\delta_E=0$) or the weighted claw
$(0;1,2,2)$ ($\delta_E=1$).
Reattaching their boundary stubs gives exactly the assemblies in
\Cref{def:exceptional-assemblies}.  In E0 the identity
$B-p=(4q+2)-(3q+2)=q$ gives $q$ incidence-degree-two supports.  The two
singleton stubs must use two distinct such supports, so $q\ge2$.  In E1,
$B-p=(4q+1)-(3q+2)=q-1$, and its one deficit stub requires a connector, again
forcing $q\ge2$.  The explicit path constructions following the definition
prove sufficiency for every such $q$.

Finally consider $R2$.  Here
\[
 p=\ell=3q+1,\quad d=3q+3,\quad
 \sum_C\delta_C=q+1,\quad \sum_C n_-(Q_C)=1.
\]
For $\delta_C>0$, \Cref{lem:partial-packing} gives
$|C|\ge3\delta_C$, whereas $\delta_C=0$ contributes the strict surplus
$|C|>0$.  Since $d=\sum_C|C|=3q+3=3\sum_C\delta_C$, no zero-cover component
exists and equality holds componentwise: $|C|=3\delta_C$.
Suppressing the packed degree-two vertices produces a
subcubic tree $F_C$ on $2\delta_C$ vertices with a perfect matching.  Indeed,
packing equality makes the closed neighbourhoods of the packed total-degree-two
vertices partition $C$, so every other vertex meets exactly one packed vertex;
reducedness then forces every such outside vertex to have total degree three.
This also supplies the boundary count used below.  The
residual block satisfies
\begin{equation}
 Q_C\cong I_{\delta_C}\oplus(I-A(F_C)).             \label{eq:e37-suppression}
\end{equation}
Packing equality also gives $B_C=2\delta_C+2$ and hence
$\pi_C=1-\delta_C$.  The core edge count is
\[
 \sum_C\pi_C=-3-e_{PP}.
\]
Exactly one residual block, say $Q_E$, has one negative direction; all other
blocks are positive semidefinite and have nonnegative pressure.  Therefore
$1-\delta_E\le-3-e_{PP}$, so $\delta_E\ge4$.  Its suppressed tree has at
least eight vertices and therefore contains an induced $P_6$ by
\Cref{lem:small-diameter-matching}.  Since $P_6$ has
two adjacency eigenvalues greater than one
($\lambda_2(P_6)=2\cos(2\pi/7)>1$), interlacing and
\eqref{eq:e37-suppression} force at least two negative directions, a
contradiction.  Thus $R2$ is empty.

Conversely, suppose that one of \textup{(A)}--\textup{(D)} holds.  In
\textup{(A)}, a clean contraction lowers the order by three and the assumed
parameter shifts give
\[
 \Phi(T)=\Phi(T')+21=21.
\]
Thus the existence of one such quotient implies slack twenty-one, and the
forward argument then shows that every clean quotient has the same property.
In \textup{(B)}, \Cref{cor:loaded-parameters} gives $\Phi(T)=|X|=21$.
Every tree in \textup{(C)} is a one-leaf inflation of a slack-twenty tree
(including $P_4$ at the boundary), so \Cref{lem:leaf-inflation} raises
$\Phi$ from twenty to twenty-one.

For \textup{(D)}, the incidence rules give $p=\ell=3q+2$ and
$d=3q+1$.  The E0 forest has cover sum $q$, and the E1 forest has cover sum
$1+(q-1)=q$.  Every residual block is positive semidefinite: ordinary paths
have the zero-negative block from the equality case, the E0 singleton has
$Q=[1]$, and direct substitution verifies the E1 claw.  Hence
\[
 \gamma=p+q=4q+2,\qquad \mu=p=3q+2,
\]
so $|T|=2p+d=9q+5$ and direct substitution gives $\Phi(T)=21$.
The incidence definition also gives subcubicity and reducedness.

Finally, \textup{(A)} is nonreduced and the other branches are reduced.
Among the reduced branches, \textup{(C)} has $e_{PP}=1$, whereas
\textup{(B)} and \textup{(D)} have $e_{PP}=0$; the latter two are separated
by $\ell-p=21$ versus $0$.  The terminal-tail and bridge-split subfamilies
inside \textup{(C)} retain signatures $\{0,1\}$ and $\{1,1\}$, and E0/E1
inside \textup{(D)} have different exceptional deep-component profiles.
Thus all branches are disjoint, completing the equivalence.
\end{proof}

\begin{remark}[Capacity thresholds]
The lower bounds on $q$ come from different integer capacities.  A
twenty-one-load lift needs $2q+2\ge21$, hence $q\ge10$.  An E0 assembly needs
two connectors for the two stubs of its exceptional singleton, while E1
needs at least one connector beyond its $q-1$ connector count; both therefore
start at $q=2$.  Terminal-tail inflation starts at $q=1$, bridge-split
inflation at $q=2$, and the loaded $P_4$ is the separate $q=0$ boundary.
\end{remark}

The theorem identifies a second phase change.  Slack twenty introduces a
single edge inside the penultimate set; slack twenty-one permits clean
expansion and two exceptional weighted components, but packing equality
still excludes the first genuinely spectral defect.

\section{Local compactness of weighted deep components}
\label{sec:local-compactness}

The classifications above repeatedly isolate a finite exceptional weighted
component and an unbounded collection of ordinary components.  We now prove
that this is not an accident of the first few layers.  The next theorem is
the compactness mechanism behind every fixed-slack statement in the rest of
the paper.

Let $C$ be a tree with integral boundary weights $b_v\ge0$ such that
$\deg_C(v)+b_v\ge2$ for every vertex.  Put
\[
 Q_C=L(C)+\diag(b_v-1:v\in C),\qquad
 R_C=\{v:b_v=0\}.
\]
Let $\delta_C$ be the minimum number of closed neighbourhoods in $C$ that
cover $R_C$, and define
\[
 z_C=|C|-3\delta_C,\qquad
 \eta_C=n_-(Q_C),\qquad
 \chi_C=\sum_{v\in C}(\deg_C(v)+b_v-3)_+.
\]
The reducedness condition used below is precisely that the vertices of total
degree two form an independent set.

\begin{theorem}[Local compactness]
\label{thm:local-compactness}
Under the preceding hypotheses, assume in addition that the vertices of
total degree two form an independent set.  Then $z_C\ge0$, and $|C|$ is
bounded by a function of $(z_C,\eta_C,\chi_C)$ alone.  More explicitly, set
\[
 \Delta_C=3+\chi_C,
 \qquad D_C=3(\eta_C+4z_C+\chi_C)+1.
\]
Then one may take
\begin{equation}
 |C|\le \frac32\left(
  1+\Delta_C\frac{(\Delta_C-1)^{D_C}-1}{\Delta_C-2}
 \right).                                           \label{eq:local-moore}
\end{equation}
In the subcubic case, with $w=\eta_C+4z_C$, the simpler bound
\begin{equation}
 |C|\le\frac32\bigl(3\cdot2^{3w+1}-2\bigr)          \label{eq:local-subcubic-bound}
\end{equation}
holds.  Moreover,
\begin{equation}
 (z_C,\eta_C,\chi_C)=(0,0,0)
 \quad\Longleftrightarrow\quad
 C\text{ is the weighted path }(2,0,2).             \label{eq:ordinary-characterization}
\end{equation}
\end{theorem}

\begin{proof}
By the partial domination--packing duality from
\Cref{sec:preliminaries}, choose a distance-three packing
$Z\subseteq R_C$ of size $\delta_C$.  Every target has internal degree at
least two.  If $Z_2$ is the set of degree-two packed targets and
$Z_+=Z\setminus Z_2$, disjointness of the closed neighbourhoods gives
\begin{equation}
 3\delta_C+
 \sum_{x\in Z}(\deg_C(x)-2)
 \le |C|=3\delta_C+z_C.                             \label{eq:packing-budget}
\end{equation}
Thus $z_C\ge0$, $|Z_+|\le z_C$, and at most
\[
 z_C-\sum_{x\in Z}(\deg_C(x)-2)
\]
vertices lie outside the packed closed neighbourhoods.

Each member of $Z_2$ has diagonal entry one in $Q_C$.  Reducedness says its
two neighbours have total degree at least three, while packing says the
neighbour pairs for distinct members of $Z_2$ are disjoint.  Eliminate the
$Z_2$ coordinates.  Graphically this suppresses every vertex of $Z_2$ and
produces a tree $F$ in which the new edges form a matching $M$.  We have
\[
 |M|=|Z_2|=\delta_C-|Z_+|,
 \qquad |F|=2\delta_C+z_C+|Z_+|.
\]
Consequently the number of vertices of $F$ not covered by $M$ is
\begin{equation}
 |F|-2|M|=z_C+3|Z_+|\le4z_C.                       \label{eq:unmatched-budget}
\end{equation}

Schur complementation gives a congruence
\begin{equation}
 Q_C\cong I_{|Z_2|}\oplus\bigl(I-A(F)+D\bigr),     \label{eq:compact-schur}
\end{equation}
where $D\succeq0$ is diagonal.  Relative to the baseline identity, $D$ can
be supported only on the unmatched vertices counted in
\eqref{eq:unmatched-budget} and on matched endpoints whose total degree is
greater than three.  Therefore
\begin{equation}
 \operatorname{rank}D\le4z_C+\chi_C,
 \qquad \Delta(F)\le3+\chi_C.                      \label{eq:rank-degree-bounds}
\end{equation}

Write
\[
 \nu_1(F)=n_-(I-A(F))
 =|\{\lambda\in\operatorname{Spec}A(F):\lambda>1\}|.
\]
A positive-semidefinite perturbation of rank $r$ removes at most $r$
negative directions.  From \eqref{eq:compact-schur} and
\eqref{eq:rank-degree-bounds},
\begin{equation}
 \nu_1(F)\le \eta_C+4z_C+\chi_C.                  \label{eq:index-budget}
\end{equation}
If the diameter of $F$ is $D_F$, a geodesic is an induced
$P_{D_F+1}$.  Interlacing and
$\nu_1(P_m)=\lfloor m/3\rfloor$ yield $D_F\le D_C$.  Applying the ordinary
Moore bound with maximum degree $\Delta_C$ bounds $|F|$ by the parenthesis
in \eqref{eq:local-moore}.  Because the suppressed edges form a matching,
$|C|=|F|+|M|\le3|F|/2$, proving \eqref{eq:local-moore}.  When
$\chi_C=0$, breadth-first expansion gives
$|F|\le1+3\sum_{j=0}^{D_F-1}2^j=3\cdot2^{D_F}-2$, which proves
\eqref{eq:local-subcubic-bound}.

It remains to identify the zero triple.  If
$(z_C,\eta_C,\chi_C)=(0,0,0)$, every packed target has degree two, $M$ is a
perfect matching of $F$, and
\begin{equation}
 Q_C\cong I_{\delta_C}\oplus(I-A(F)),
 \qquad n_-(I-A(F))=0.                              \label{eq:zero-tile-schur}
\end{equation}
If $|F|\ge4$, then a tree with a perfect matching is not a star and contains
an induced $P_4$.  Since $P_4$ has an adjacency eigenvalue greater than one,
interlacing contradicts \eqref{eq:zero-tile-schur}.  Hence $F=P_2$; its
matching edge has one suppressed target, so $C=P_3$.  Total degree three at
the endpoints and total degree two at the centre force weights $(2,0,2)$.
The converse follows by direct calculation: this weighted path has
$\delta_C=1$, $z_C=\chi_C=0$, and inertia $(2,0,1)$ for $Q_C$.
\end{proof}

\begin{remark}
The bounds are deliberately coarse.  Their purpose is to convert fixed
slack into finiteness, not to provide a practical enumeration bound.  The
$\chi_C$ extension will be used only for the bounded excess degree created
by clean contractions; it is not a classification for unrestricted maximum
degree.
\end{remark}

\section{The fixed-slack finite-defect theorem}
\label{sec:finite-kernels}

This is the structural culmination of the paper.  The explicit phase lists
that follow are applications demonstrating that the theorem can be made
concrete; the general conclusion here is not limited to any atlas cutoff.

For a reduced tree, retain the canonical partition
$V(T)=L\sqcup P\sqcup D$.  For each weighted deep component $C$, write
\[
 \delta_C=\gamma(C;R_C),\quad z_C=|C|-3\delta_C,
 \quad\eta_C=n_-(Q_C),\quad
 \pi_C=B_C-3\delta_C-1.
\]
Call the weighted path
\[
 \mathcal O=(2,0,2)
\]
the \emph{ordinary deep tile}.

\begin{definition}[Ported defect kernel]
A support event is exceptional if the support has more than one leaf, is
incident with a support--support edge, has degree greater than three in a
clean quotient, or is adjacent to a nonordinary deep tile.  A ported defect
kernel records all nonordinary deep tiles, all exceptional support events,
their incidences, and labelled half-edges leading into the remaining
ordinary incidence forest.  Kernel isomorphism preserves weights, leaf
multiplicities, and port labels but ignores the size and shape of the
ordinary lift forest attached at the ports.
\end{definition}

Outside the kernel there are only copies of $\mathcal O$ and one-leaf
supports of degree at most three.  Under the inverse construction of
\Cref{sec:equality}, this is a partial matched-skeleton lift forest.

\begin{figure}[t]
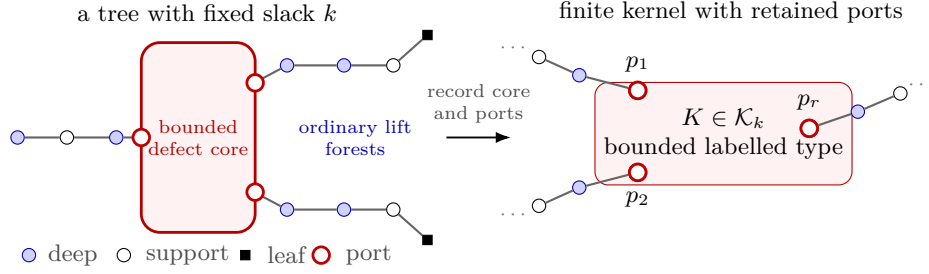

  \centering
  \resizebox{\textwidth}{!}{\PtwoKernelDecompositionGraphic}
  \caption{The ported-kernel decomposition.  At fixed slack, every
  nonordinary deep component and exceptional support event is absorbed into
  a bounded labelled core.  Only compatible partial matched-skeleton lift
  forests may continue through its ports (the short continuation fragments drawn
  outside the red kernel box).  Three representative ports are drawn; in
  general they are $p_1,\ldots,p_r$.  Nonreduced trees
  require at most $\lfloor k/21\rfloor$ reverse clean expansions.}
  \label{fig:ported-kernel}
\end{figure}

\begin{theorem}[Fixed-slack finite defects]
\label{thm:fixed-slack-kernels}
For every integer $k\ge0$, there is an effectively computable finite set
$\mathcal K_k$ of ported defect-kernel types such that every subcubic tree
$T$ with $\Phi(T)=k$ is obtained by
\begin{enumerate}[label=\textup{(\roman*)}]
 \item choosing a kernel in $\mathcal K_k$;
 \item attaching arbitrary compatible partial matched-skeleton lift forests
       at its ports; and
 \item applying at most $\lfloor k/21\rfloor$ clean three-edge expansions.
\end{enumerate}
For reduced trees the last step is omitted.  Thus every fixed-slack
subcubic layer has only finitely many local defect types.
\end{theorem}

\begin{lemma}[Bounded-excess quotient kernels]
\label{lem:bounded-excess-quotient-kernel}
For fixed nonnegative integers $k$ and $m$, let $G$ be a reduced tree of
order at least five such that
\[
 0\le \Phi(G)\le k,
 \qquad
 \chi(G):=\sum_{v\in V(G)}(\deg_G(v)-3)_+\le m.
\]
Then the number and order of its nonordinary deep components, the number of
exceptional support events, and the number of ports are bounded by
effectively computable functions of $(k,m)$.  Consequently only finitely
many ported quotient-kernel types occur.
\end{lemma}

\begin{proof}
The five-term identity remains an exact identity for a reduced tree without
an upper degree assumption.  Hence each of $a,b,t,r$ is at most $k$, and
the local decompositions give
\[
 \sum_C z_C=b,
 \qquad
 \sum_C\eta_C=t,
 \qquad
 \sum_C\chi_C\le\chi(G)\le m.
\]
Thus at most $b+t+\chi(G)$ components have a nonzero triple
$(z_C,\eta_C,\chi_C)$.  By \Cref{thm:local-compactness}, each belongs to an
effectively computable finite weighted list depending only on $(k,m)$; the
zero triple is precisely the ordinary tile $\mathcal O$.  Notice also that
the total degree of every vertex is at most $3+m$, so the possible boundary
weights on every bounded component form a finite set.

We next bound the support part of the kernel.  For a deep component $C$,
the degree sum in the leaf-deleted core is
\[
 \sum_{v\in C}\bigl(\deg_C(v)+b_v\bigr)
   =2(|C|-1)+B_C.
\]
Since
$\deg_C(v)+b_v\le3+(\deg_C(v)+b_v-3)_+$, it follows that
\begin{equation}
 B_C\le |C|+\chi_C+2.                              \label{eq:quotient-boundary}
\end{equation}
Consequently only boundedly many support incidences meet nonordinary
components.  Supports carrying more than one leaf are bounded in number by
the leaf excess $a$, while supports of degree greater than three are bounded
in number by $\chi(G)$.

Finally the core Euler identity gives, without a degree restriction,
\begin{equation}
 \sum_C\pi_C=s+3r-3t-e_{PP}.                       \label{eq:quotient-pressure}
\end{equation}
Every positive-semidefinite component has nonnegative pressure.  A component
of negative pressure must therefore have $\eta_C>0$ and belongs to the
bounded nonordinary list just obtained.  Its order and boundary weights are
bounded, so its negative pressure is bounded as well.  Equation
\eqref{eq:quotient-pressure} now bounds $e_{PP}$, hence also the number of
endpoints of support--support edges.  Together with
\eqref{eq:quotient-boundary}, the bounds on multiple-leaf and high-degree
supports bound every exceptional support event.  The kernel has bounded
order and maximum degree at most $3+m$, so it has boundedly many ports and
only finitely many labelled ported isomorphism types.  Every bound used is
explicitly enumerable from \Cref{eq:local-moore}, proving effectiveness.
\end{proof}

\begin{proof}[Proof of \Cref{thm:fixed-slack-kernels}]
The four trees $P_2,P_3,P_4,K_{1,3}$ are finite base cases: whenever one of
them occurs in the target layer, include its whole isomorphism type as a
closed kernel with no ports.  We therefore apply the canonical-partition
argument below only to trees of order at least five.

First suppose that $T$ is reduced.  Put
\[
 a=\ell-p,\qquad b=d-3h.
\]
The five-term identity \eqref{eq:five-term-slack} reads
\begin{equation}
 k=a+b+20s+21t+60r,                                \label{eq:kernel-five-term}
\end{equation}
where $a,b,t,r\ge0$ and $s\in\{0,1,2\}$.  Hence only finitely many
coordinate rows occur at fixed $k$.  The exact local decompositions give
\begin{equation}
 b=\sum_C z_C,\qquad t=\sum_C\eta_C.              \label{eq:local-budgets}
\end{equation}
At most $b+t$ deep components therefore have
$(z_C,\eta_C)\ne(0,0)$.  By \Cref{thm:local-compactness}, every such
component belongs to a finite weighted list, while
\eqref{eq:ordinary-characterization} identifies every zero-cost component
with $\mathcal O$.

For reference, a fully explicit, intentionally large bound is obtained by
putting
\begin{equation}
 W_k=4k+\left\lfloor\frac{k}{21}\right\rfloor,
 \qquad
 M_k=\left\lceil\frac32
   \left(3\cdot2^{3W_k+1}-2\right)\right\rceil.
                                                               \label{eq:Mk}
\end{equation}
Every nonordinary component has at most $M_k$ vertices, and there are at
most $k+\lfloor k/21\rfloor$ of them.

It remains to bound exceptional support events.  The leaf excess $a$ bounds
the number of supports carrying more than one leaf.  The core-edge identity
becomes
\begin{equation}
 \sum_C\pi_C=s+3r-3t-e_{PP}.                       \label{eq:kernel-pressure}
\end{equation}
Ordinary components have pressure zero.  Positive-semidefinite components
have nonnegative pressure, while every negative component is drawn from the
finite exceptional list.  Since $|\pi_C|\le6|C|+1$ in the subcubic case,
one coarse consequence is
\begin{equation}
 e_{PP}\le2+3\left\lfloor\frac{k}{60}\right\rfloor
 +\left\lfloor\frac{k}{21}\right\rfloor(6M_k-2).   \label{eq:epp-bound}
\end{equation}
Thus loaded supports, endpoints of support--support edges, and supports
adjacent to exceptional components are bounded in number.

Every remaining support has one leaf, no support neighbour, and only
ordinary deep neighbours.  The inverse equality construction identifies
the complement of the bounded kernel with a partial matched-skeleton lift
forest: the endpoints of an ordinary tile are skeleton vertices, its
zero-weight centre subdivides a matching edge, incidence-two supports
subdivide nonmatching skeleton edges, and incidence-one supports are arms.
The kernel has bounded order and degree, so only finitely many weighted
ported isomorphism types occur.  This proves the reduced case.

Now let $T$ be nonreduced.  While the current tree is nonreduced and has
order at least five, contract a clean three-edge path.  The last contraction
is allowed to land below order five; stop as soon as the result is either a
reduced quotient or one of the four finite base trees above.  For
each performed contraction from a current tree $U$ to $U'$, writing
$g=\gamma(U')+1-\gamma(U)\ge0$, the exact shift formula gives
\begin{equation}
 \Phi(U)=\Phi(U')+21+63g.                           \label{eq:kernel-contraction}
\end{equation}
Because the terminal object has nonnegative slack, after at most
$m\le\lfloor k/21\rfloor$ such contractions we reach an object $T_0$ with
slack $0\le k'\le k$.  If $T_0$ is one of the four finite bases, absorb it
as a closed kernel and proceed directly to the finite reverse-expansion
step below.  Otherwise $T_0$ is reduced.  Contraction need not preserve
subcubicity, so
define
\[
 \chi(G)=\sum_{v\in V(G)}(\deg_G(v)-3)_+.
\]
If a contracted path has endpoint degrees $\alpha,\beta$, then
\[
 (\alpha+\beta-5)_+
 \le(\alpha-3)_++(\beta-3)_++1.
\]
Therefore
\begin{equation}
 \chi(T_0)\le m,\qquad \Delta(T_0)\le m+3.        \label{eq:quotient-excess}
\end{equation}
Apply \Cref{lem:bounded-excess-quotient-kernel} to $T_0$ with parameters
$(k,m)$.  This is the point at which the possible high degrees created by
contraction are controlled: \eqref{eq:quotient-excess} supplies the global
excess-degree budget required by that lemma.  Hence only finitely many
quotient slacks, contraction counts, exceptional components, supports, and
ports remain.

Finally reverse at most $m$ contractions.  For a reduced terminal quotient,
the degree bound in \eqref{eq:quotient-excess} leaves only finitely many
incident-edge partitions at each expansion; for a finite base quotient this
finiteness is immediate.  Keep only the reversals recovering a subcubic
tree.  Taking the finite union over all allowed quotient rows and finite
bases produces $\mathcal K_k$.
\end{proof}

\begin{remark}[Effective does not mean practical]
The proof gives a finite enumeration: list the contraction rows and
five-term coordinates, enumerate weighted components up to
\eqref{eq:local-moore}, assemble the bounded support kernels, and test the
bounded reverse expansions.  The Moore bounds are far too large for a
useful recognition algorithm.  A canonical minimum kernel and a short list
of named rewrites remain separate problems.
\end{remark}

\section{Uniform spectral exclusions}
\label{sec:spectral-gaps}

The five-term identity lists algebraically possible branches.  Two spectral
facts remove whole collections of those branches uniformly, rather than one
slack at a time.

\subsection{Two local spectral lemmas}

Let $C$ be a weighted subcubic deep component, put
\[
 X_C=\{v:\deg_C(v)+b_v=2\},\qquad m_C=|X_C|,
\]
and assume $X_C$ is independent, as it is in a reduced tree.  Then
\[
 Q_C=2I-A(C)-\diag(1_{X_C}).
\]

\begin{lemma}[Near-packing spectral gap]
\label{lem:near-packing-gap}
If $|C|=3\delta_C+1$ and $n_-(Q_C)=1$, then $m_C\le4$.
\end{lemma}

\begin{proof}
Choose a distance-three packing $Z\subseteq R_C$ of size $\delta_C$.
Disjointness of its closed neighbourhoods shows that either every packed
target has degree two or exactly one has degree three.

We use three elementary matching facts.  First, a subcubic tree with a
matching missing one vertex and order at least nine contains an induced
$P_6$.  Indeed, if its diameter were at most four, rooting at a central
vertex or central edge shows that at most one branch may contain the
unmatched vertex, while every other branch is internally matched; the
subcubic bound then permits at most seven vertices.  Second, a subcubic tree
with a perfect matching and order at least eight contains an induced $P_6$
by \Cref{lem:small-diameter-matching}.  Third, the complete perfect-matching
list through order six is
\begin{equation}
\begin{array}{c|c|c}
|H|&H&\nu_1(H)\\ \hline
2&P_2&0\\
4&P_4&1\\
6&P_6&2\\
6&S_{1,2,2}&1.
\end{array}                                        \label{eq:small-pm-list}
\end{equation}
The last column follows from the characteristic polynomials.  Since $P_6$
has two adjacency eigenvalues greater than one, an induced $P_6$ gives two
negative directions for $I-A(F)$ by interlacing.

In the first case, suppressing all packed targets produces a subcubic tree
$F$ on $2\delta_C+1$ vertices.  The suppressed edges form a matching
covering all but one vertex $w$, and
\begin{equation}
 Q_C\cong I_{\delta_C}\oplus
 \bigl(I-A(F)+\varepsilon e_we_w^{\trans}\bigr),
 \qquad m_C=\delta_C+1-\varepsilon,                 \label{eq:near-packing-A}
\end{equation}
where $\varepsilon\in\{0,1\}$.  If $\varepsilon=0$ and $m_C\ge5$, then
$|F|\ge9$; a subcubic tree with a matching missing one vertex and order at
least nine contains an induced $P_6$.  Since $P_6$ has two adjacency
eigenvalues greater than one, \eqref{eq:near-packing-A} has at least two
negative directions.

Suppose $\varepsilon=1$.  Every component of $F-w$ has a perfect matching,
and its principal block is $I-A(F-w)$.  If this block has fewer than two
negative directions, \eqref{eq:small-pm-list} and the induced-$P_6$ fact
leave at most one component with $\nu_1=1$, all others being copies of
$P_2$.  There are at most three components.  Since $m_C=\delta_C\ge5$,
their total order is at least ten, and the subcubic degree bound forces the
sole boundary
\[
 \delta_C=5,\qquad F-w\cong S_{1,2,2}\sqcup P_2\sqcup P_2.
\]
There are three attachment orbits in $S_{1,2,2}$: its short leaf, a long-arm
middle, and a long leaf.  The attachments to the two $P_2$ components are
unique up to symmetry.  For the three displayed orbits, exact
characteristic polynomials of the residual block give
\begin{equation}
\resizebox{\textwidth}{!}{$
\begin{array}{c|l|c}
\text{attachment}&\chi(x)/[x^2(x-2)^2]&\In\\ \hline
\text{short leaf}&(x^2-x-1)(x^5-7x^4+13x^3+x^2-14x+5)&(7,2,2)\\
\text{long-arm middle}&x^7-8x^6+19x^5-5x^4-29x^3+21x^2+7x-5&(7,2,2)\\
\text{long leaf}&x^7-8x^6+19x^5-5x^4-28x^3+18x^2+8x-4&(7,2,2).
\end{array}$}                                      \label{eq:near-packing-A-table}
\end{equation}
Sturm sign variation gives two negative roots in every row.  Thus the first
packing case is impossible when $m_C\ge5$.

In the second packing case, let $z$ be the unique degree-three packed target.
Suppress the other $\delta_C-1$ packed vertices and delete $z$.  The three
rooted branches $K_i$ have matching-covered interiors.  If $\varepsilon_i$
records whether the branch root has total degree three and $\delta_i$ is the
number of matching edges in that branch, then
\begin{equation}
 m_C=\sum_{i=1}^3(\delta_i+1-\varepsilon_i).        \label{eq:near-packing-B}
\end{equation}
To justify the remaining finite boundary, write the residual block on a
rooted branch $K_i$ as
\[
 S_i=I-A(K_i)+\varepsilon_i e_{y_i}e_{y_i}^{\trans}.
\]
If $S_i\succeq0$, then $K_i$ is a singleton.  For $\varepsilon_i=0$, any
connected tree of order at least three contains an induced $P_3$, whose
adjacency index exceeds one.  For $\varepsilon_i=1$, a matching edge in
$K_i-y_i$ gives the principal path block with diagonal $(2,1,1)$ and
determinant $-1$.

If $n_-(S_i)=1$, then
\begin{equation}
 \delta_i+1-\varepsilon_i\le
 \begin{cases}4,&\varepsilon_i=0,\\3,&\varepsilon_i=1.
 \end{cases}                                       \label{eq:rooted-branch-bound}
\end{equation}
The first line follows from the near-perfect-matching induced-$P_6$ fact.
For the second, deletion of $y_i$ and \eqref{eq:small-pm-list} give the
following conservative boundary-forest superset; direct leaf elimination
gives
\begin{equation}
\begin{array}{c|c|c|c}
K_i-y_i&\text{raw attachments}&\text{orbits}&\In(S_i)\\ \hline
S_{1,2,2}\sqcup P_2&10&3&(6,2,1)\\
P_4\sqcup P_2\sqcup P_2&16&2&(6,2,1)\\
S_{1,2,2}\sqcup P_2\sqcup P_2&20&3&(7,2,2).
\end{array}                                        \label{eq:rooted-branch-table}
\end{equation}
Some raw attachments in this table give the branch root degree three before
the edge to $z$ is restored, so they are not legal full subcubic branches.
Keeping them is a safe over-enumeration: every legal branch is included, and
the asserted one-negative bound holds for the larger set.  Hence the second
line of \eqref{eq:rooted-branch-bound} follows.

The branch principal submatrix $\bigoplus_iS_i$ has at most one negative
direction.  If all three branches are positive semidefinite, they are
singletons and \eqref{eq:near-packing-B} gives $m_C\le3$.  Otherwise exactly
one branch is one-negative and the other two are singletons.  Under
$m_C\ge5$, every legal state is covered by the four mark-count rows
\begin{equation}
\begin{array}{c|c|c|c|c}
\varepsilon&m_{\rm branch}&\text{marked singleton branches}&
\text{rooted types}&\In(S)\\ \hline
0&3&2&3&(6,2,0)\\
0&4&1&1&(7,2,1)\\
0&4&2&1&(7,2,1)\\
1&3&2&2&(7,2,1).
\end{array}                                        \label{eq:near-packing-B-table}
\end{equation}
The six rooted representatives used in these rows are listed next.  Vertices
are numbered from zero and every root is vertex zero; `central $r$' means
that $r$ marked singleton branches accompany the displayed one-negative
branch.
\begin{equation}
\resizebox{\textwidth}{!}{$
\begin{array}{c|c|l|c|c|l|c}
&\varepsilon&E(K)&\text{root}&\In(S_i)&\text{central cases }r:\In(S)&
\text{full-degree status}\\ \hline
A1&0&01,03,12,34&0&(3,1,1)&2:(6,2,0)&\text{conservative only}\\
A2&0&01,12,13,34&0&(4,1,0)&2:(6,2,0)&\text{legal}\\
A3&0&01,12,23,34&0&(3,1,1)&2:(6,2,0)&\text{legal}\\
B1&0&01,12,23,25,34,56&0&(4,1,2)&1:(7,2,1),\ 2:(7,2,1)&\text{legal}\\
C1&1&01,03,05,12,34,56&0&(4,1,2)&2:(7,2,1)&\text{conservative only}\\
C2&1&01,12,23,25,34,56&0&(5,1,1)&2:(7,2,1)&\text{legal}.
\end{array}$}                                      \label{eq:rooted-six-catalog}
\end{equation}
For A1 the root has degree two in $K$, so restoring the edge to $z$ conflicts
with $\varepsilon=0$; for C1 the root already has degree three in $K$, so
restoring that edge gives degree four.  Their inclusion can only strengthen
the exclusion.  For A2, A3, B1, and C2, the root-deleted forest is perfectly matched and direct
inspection verifies the stated mark-count row.  Exact rational leaf
congruence gives the branch and central inertias displayed in
\eqref{eq:rooted-six-catalog}; the machine-readable companion certificate is
\path{near-packing-rooted-type-catalog.json}.  Thus every legal state is covered
and has at least two negative directions after central assembly.  This
contradicts $n_-(Q_C)=1$.  The degree-three packing case is therefore
impossible, and $m_C\le4$.
\end{proof}

\begin{lemma}[Two-surplus marked-centroid gap]
\label{lem:two-surplus-gap}
If $m_C\ge6$ and $n_-(Q_C)=1$, then $m_C=6$ and
\begin{equation}
 |C|-3\delta_C\ge7.                                \label{eq:two-surplus-rigidity}
\end{equation}
In particular, $|C|=3\delta_C+2$ and $m_C\ge6$ imply
$n_-(Q_C)\ne1$.
\end{lemma}

\begin{proof}
For an induced connected subtree $H$, the all-ones vector has energy
\begin{equation}
 \boldsymbol1^{\trans}Q_C[H]\boldsymbol1=2-|X_C\cap H|.  \label{eq:marked-energy}
\end{equation}
No edge can split $X_C$ into two parts of size at least three.  Indeed, if
$C-uv=A\sqcup B$ with $a=|A\cap X_C|\ge3$ and
$b=|B\cap X_C|\ge3$, the constant vector on $A$ cannot minimize energy
subject to value one at $u$.  Otherwise differentiation in each nonroot
coordinate would give $(Q_C[A]\boldsymbol1)_x=0$ for $x\ne u$.  Thus every
marked nonroot vertex is a leaf of $A$, and every unmarked nonroot vertex has
degree two in $A$.  Every branch from $u$ is consequently a path ending at
one marked leaf, so
\[
 a=\deg_A(u)+1_{\{u\in X_C\}}.
\]
If $u\in X_C$, then $\deg_C(u)\le2$ and $\deg_A(u)\le1$; if
$u\notin X_C$, then $\deg_C(u)\le3$ and $\deg_A(u)\le2$.  Both cases give
$a\le2$, a contradiction.  Hence there is a rooted vector $f$ on $A$ with
$f(u)=1$ and $q_A(f)<2-a$.  Together with the constant vector on $B$, its
Gram matrix is
\[
 \begin{pmatrix}q_A(f)&-1\\-1&2-b\end{pmatrix},
\]
which is negative definite because both diagonal entries are negative and
$q_A(f)(2-b)-1>(a-2)(b-2)-1\ge0$.  This would give two negative
directions.

Choose a centroid $c$ for the unit weights on $X_C$.  If $m_C\ge7$, every
component of $C-c$ has at most two marks.  Indeed, a component with at least
three marks has at most $m_C/2$ marks by the centroid property, so the edge
joining it to $c$ has at least three marks on both sides, contrary to the
preceding exclusion.
Subcubicity gives at most six marks when $c\notin X_C$ and at most five when
$c\in X_C$, a contradiction.  Thus $m_C=6$; moreover $c$ is unmarked,
trivalent, and each of the three branches $B_i$ of $C-c$ contains exactly
two marks.

By \eqref{eq:marked-energy}, the constant vector has zero energy on each
$B_i$.  It minimizes energy subject to value one at the root $y_i$: if a
rooted vector $f$ had energy $\alpha<0$, put the constant value one on the
other two branches and value $t$ at $c$.  For the resulting vector $g_t$,
\[
 q_C(g_t)=2t^2-4t,\qquad f^{\trans}Q_Cg_t=-t.
\]
When $0<t<-4\alpha/(1-2\alpha)$, both diagonal entries of the Gram matrix
of $f,g_t$ are negative and
\[
 \alpha(2t^2-4t)-t^2>0.
\]
Their span would be negative definite, contradicting $n_-(Q_C)=1$.

The constrained minimum therefore occurs at the constant vector.
Differentiating in nonroot coordinates gives
$(Q_C[B_i]\boldsymbol1)_x=0$ for $x\ne y_i$.  Hence every marked nonroot
vertex is a leaf and every unmarked nonroot vertex has degree two.  Since
$B_i$ has exactly two marks,
\begin{equation}
 2=\deg_{B_i}(y_i)+1_{\{y_i\in X_C\}}.              \label{eq:centroid-root-degree}
\end{equation}
Thus $B_i$ is a path.  If $y_i$ is marked, it is one endpoint and the other
endpoint is the second mark; if it is unmarked, it lies internally between
the two marked endpoints.  Equation \eqref{eq:centroid-root-degree} also
gives $b_{y_i}=0$, and $b_c=0$ because $c$ is unmarked and trivalent.
Every other branch vertex has boundary weight one.  The zero-boundary target
set is therefore exactly $\{c,y_1,y_2,y_3\}$, so $c$ alone covers it and
$\delta_C=1$.
Each branch has at least three vertices, giving
\[
 |C|-3\delta_C=1+\sum_{i=1}^3|B_i|-3\ge7.
\]
This proves \eqref{eq:two-surplus-rigidity} and its stated consequence.
\end{proof}

\subsection{Uniform exclusions}

\begin{proposition}[Uniform one-negative gap]
\label{prop:uniform-one-negative-gap}
For every reduced subcubic tree,
\[
 s=0,\qquad t=1\quad\Longrightarrow\quad r\ge1.
\]
No condition on $a=\ell-p$ or $b=d-3h$ is needed.
\end{proposition}

\begin{proof}
Assume $r=0$.  Then
\[
 p=3q+1,\qquad h=\sum_C\delta_C=q+1,\qquad
 \sum_C n_-(Q_C)=1.
\]
Let $N$ be the unique one-negative component, put
$z_N=|N|-3\delta_N$, and let $m_N$ count its total-degree-two vertices.
The core identity and the local degree sum give
\begin{equation}
 \sum_C\pi_C=-3-e_{PP},
 \qquad m_N=z_N+1-\pi_N.                            \label{eq:uniform-gap-identities}
\end{equation}
All components other than $N$ are positive semidefinite and have
nonnegative pressure, so $\pi_N\le-3-e_{PP}$.

If $z_N=0$, packing equality suppresses $N$ to a subcubic
perfect-matching tree $F_N$ and
\[
 Q_N\cong I_{\delta_N}\oplus(I-A(F_N)),
 \qquad \pi_N=1-\delta_N.
\]
Thus $\delta_N\ge4$.  Such a perfect-matching tree contains an induced
$P_6$, and interlacing gives at least two negative directions, a
contradiction.

If $z_N=1$, \eqref{eq:uniform-gap-identities} gives $m_N\ge5$.  Suppress a
maximum distance-three packing as in \Cref{lem:near-packing-gap}.
Reducedness supplies the required independence, and that lemma gives
$m_N\le4$, a contradiction.

Finally, if $z_N\ge2$, the same rooted suppression applied at a marked
centroid is exactly \Cref{lem:two-surplus-gap}.  It forces $m_N=6$ and
$z_N\ge7$, whence $\pi_N=z_N+1-m_N=z_N-5\ge2$.  This contradicts
$\pi_N\le-3-e_{PP}$.  Hence $r\ge1$.
\end{proof}

\begin{corollary}
\label{cor:no-spectral-below-forty}
A reduced tree with $\Phi<40$ has $t=0$.
\end{corollary}

\begin{proof}
The coefficient $21t$ excludes $t\ge2$.  If $t=1$ and $s\ge1$, the cost is
at least $41$; if $s=0$, \Cref{prop:uniform-one-negative-gap} forces
$r\ge1$ and hence cost at least $81$.
\end{proof}

The second exclusion is needed only at the top layer of our explicit atlas.

\begin{lemma}[Low-$\nu_1$ perfect-matching bound]
\label{lem:low-nu-perfect-matching}
Let $F$ be a subcubic tree with a perfect matching, and put
\[
 \nu_1(F)=|\{\lambda\in\operatorname{Spec}A(F):\lambda>1\}|.
\]
If $\nu_1(F)\le2$, then
\begin{equation}
\begin{array}{c|ccc}
\nu_1(F)&0&1&2\\ \hline
|F|/2&1&\le3&\le6.
\end{array}                                        \label{eq:low-nu-table}
\end{equation}
\end{lemma}

\begin{proof}
The path $P_{10}$ has three adjacency eigenvalues greater than one.  Hence an
induced-geodesic argument bounds the diameter of $F$ by eight.  A subcubic
tree of diameter at most eight is obtained either from a central vertex with
at most three rooted binary branches of height at most three, or from a
central edge with two such sides.  If $\mathcal R_h$ denotes the rooted
binary types of height at most $h$, the recursion
\[
 \mathcal R_0=\{()\},\qquad
 \mathcal R_{h+1}=\{\text{multisets of at most two members of }
 \mathcal R_h\}
\]
gives $|\mathcal R_0|,\ldots,|\mathcal R_3|=1,3,10,66$.
Canonical generation from the two centre descriptions produces $50{,}248$
trees; $279$ have perfect matchings.  Exact rational congruence applied to
$I-A(F)$ leaves one type with $\nu_1=0$, two with $\nu_1=1$, and eleven
with $\nu_1=2$; their maximum half-orders are respectively $1,3,6$.
Because the induced-$P_{10}$ step supplies the all-orders diameter bound,
this finite census proves \eqref{eq:low-nu-table} rather than merely testing
small orders.
\end{proof}

\begin{corollary}[The two-negative equality row at slack forty-two]
\label{cor:no-two-negative-s42}
No reduced subcubic tree has
\[
 \Phi=42,\qquad (s,t,r,a,b)=(0,2,0,0,0).
\]
\end{corollary}

\begin{proof}
Packing equality suppresses each nonordinary component $C$ to a subcubic
perfect-matching tree $F_C$ with
\[
 Q_C\cong I_{\delta_C}\oplus(I-A(F_C)),
 \qquad \pi_C=1-\delta_C.
\]
If one component carries both negative directions,
\Cref{lem:low-nu-perfect-matching} gives $\delta_C-1\le5$; if two
components carry one each, their total defect is at most four.  Ordinary
components contribute zero.  The global pressure identity for this row,
however, requires
\[
 \sum_C(\delta_C-1)=6+e_{PP}\ge6,
\]
a contradiction.
\end{proof}

\section{The reduced phase corridor from twenty to thirty-nine}
\label{sec:phase-corridor}

Write $\Phi(T)=20+k$ with $0\le k\le19$.  We use \emph{loading} for the
operation of giving an incidence-degree-one support a second leaf.  A load
is legal when reducedness is preserved.  A $j$-loaded E0 or E1 assembly is
defined as in \Cref{def:exceptional-assemblies}, except that exactly $j$
distinct incidence-degree-one supports are loaded; a deficient exceptional
stub may meet either an incidence-degree-two connector or a loaded support.
This convention includes the small boundary values not present in the
unloaded definition.

\begin{figure}[t]
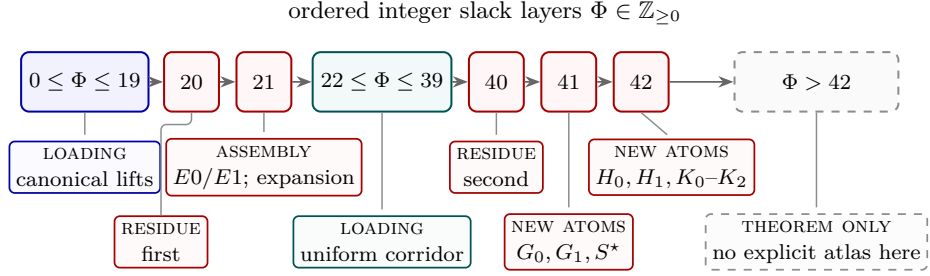

  \centering
  \resizebox{\textwidth}{!}{\PtwoPhaseAtlasGraphic}
  \caption{Navigation map for the explicit fixed-slack atlas.  The cards give
  selected mechanisms first appearing in the indicated integer layers; their
  category tags make the distinction independent of colour.  The general
  finite-kernel theorem continues beyond the explicit cutoff at forty-two.}
  \label{fig:phase-atlas}
\end{figure}

\begin{theorem}[Uniform corridor classification]
\label{thm:phase-corridor}
Let $T$ be a reduced subcubic tree and $0\le k\le19$.  Then
$\Phi(T)=20+k$ if and only if exactly one of the following holds.
\begin{enumerate}[label=\textup{(\roman*)}]
 \item $T=W(H;X)$ is a matched-skeleton lift with $20+k$ distinct loaded
       arms, where $H$ has order $2q$ and
       $20+k\le2q+2$.
 \item $T$ is obtained from a reduced $s=1,e_{PP}=1$ slack-twenty residue
       tree by $k$ legal loads on distinct incidence-one supports, where
       $k\le2q+2$.
 \item If $k\ge1$, $T$ is a $j$-loaded E0 or E1 assembly with $j=k-1$ and
       \begin{equation}
       \begin{array}{c|c}
       \mathrm{E0}&q\ge0,\quad q+j\ge2,\quad j\le2q+2,\\
       \mathrm{E1}&q\ge1,\quad q+j\ge2,\quad j\le2q+3.
       \end{array}                                  \label{eq:E-capacities}
       \end{equation}
\end{enumerate}
The parameters in the first family are
\[
 |T|=9q+k+22,\qquad \gamma(T)=4q+1,\qquad \mu(T)=3q+1,
\]
and those in the last two families are
\[
 |T|=9q+k+4,\qquad \gamma(T)=4q+2,\qquad \mu(T)=3q+2.
\]
\end{theorem}

\begin{proof}
By \Cref{cor:no-spectral-below-forty}, $t=0$; also $r=0$, while $s=2$
already costs forty.  Hence the five-term identity leaves only
\[
 (s,a+b)=(0,20+k)\quad\text{or}\quad(1,k).
\]
If $s=0$, the pressure collision gives $e_{PP}=0$ and zero pressure in every
component.  Thus every deep component is ordinary, $b=0$, and $a=20+k$.
The equality reconstruction is family \textup{(i)}, with $2q+2$ available
arms.

Suppose $s=1$.  Then
\begin{equation}
 \sum_C\pi_C=1-e_{PP},                              \label{eq:corridor-pressure}
\end{equation}
so $e_{PP}\in\{0,1\}$.  If $e_{PP}=1$, all pressures vanish, $b=0$,
and $a=k$.  Exactly $q$ supports have incidence degree two and $2q+2$ have
incidence degree one.  The latter and only the latter can be loaded, giving
family \textup{(ii)}.

If $e_{PP}=0$, exactly one component $E$ has pressure one and all others are
ordinary.  It satisfies
\[
 B_E=3\delta_E+2,\qquad |E|=3\delta_E+b.
\]
The equality-size case $b=0$ is impossible by the pressure-one local lemma
used in \Cref{thm:slack-twenty-one}.  When $b=1$, the marked Smith reduction
gives precisely the E0 weight-two singleton or the E1 weighted claw
$(0;1,2,2)$; the remaining leaf excess is $a=k-1=j$.  If $b=2$,
positive semidefiniteness forces the adjacent total-degree-two edge
$(1,1)$, contradicting reducedness.  If $b\ge3$, then
$\boldsymbol1^{\trans}Q_E\boldsymbol1=2-b<0$.  Thus only family
\textup{(iii)} survives.

For E0 the connector and terminal-support counts are $q$ and $2q+2$;
its two deficient stubs require $q+j\ge2$.  For E1 those counts are $q-1$
and $2q+3$, and its one deficient stub gives $q\ge1$ and $q+j\ge2$.
These are exactly \eqref{eq:E-capacities}.  Minimum constructions followed
by terminal extension prove sufficiency.  Direct domination and inertia
decompositions give the displayed parameters and all converses.
\end{proof}

\begin{proposition}[Continuation of the one-residue branch]
\label{prop:one-residue-continuation}
Let $T$ be a reduced subcubic tree with
\[
 s=1,\qquad t=r=0,\qquad a+b=k
\]
for an arbitrary integer $k\ge0$.  Then $T$ is either a $k$-loaded
slack-twenty residue tree with $k\le2q+2$, or, when $k\ge1$, a
$(k-1)$-loaded E0/E1 assembly satisfying \eqref{eq:E-capacities} with
$j=k-1$.
\end{proposition}

\begin{proof}
Nothing in the $s=1$ part of the preceding proof uses $k\le19$.
Equation \eqref{eq:corridor-pressure} first gives $e_{PP}\in\{0,1\}$.
For $e_{PP}=1$, pressure vanishes componentwise and gives $b=0,a=k$.
For $e_{PP}=0$, the unique pressure-one component has $b=1$ and is E0 or
E1, while $a=k-1$.  The same incidence counts give the stated capacities.
\end{proof}

\begin{corollary}[Nonreduced corridor]
\label{cor:nonreduced-corridor}
For $1\le k\le19$, every nonreduced tree of slack $20+k$ is a clean
expansion of a slack-$(k-1)$ quotient with exact unit shifts in $\gamma$ and
$\mu$.  No nonreduced tree has slack twenty.
\end{corollary}

\begin{proof}
The contraction identity
$\Phi(T)=\Phi(T')+21+63g$ and the bound $20+k\le39$ force $g=0$ and
$\Phi(T')=k-1$.
\end{proof}

At forty, the new row $s=2,t=r=0,a+b=0$ appears.  This is the first
second-residue phase and the subject of the next section.

\section{The second-residue phase at slack forty}
\label{sec:slack-forty}

For $q\ge0$, take $q$ ordinary weighted deep paths $(2,0,2)$ and
$p=3q+3$ supports, each carrying one leaf.  Attach all deep stubs to
supports, add exactly two support--support edges, require every support to
have core degree one or two, and require the graph obtained by contracting
each deep component to be a tree.  The expanded tree is called a
\emph{double-residue incidence tree}.  Reducedness is automatic: a terminal
support meeting a deep component meets a total-degree-three endpoint, while
the two endpoints of a support--support edge cannot both be isolated from
the remaining core.

The family begins at $q=0$ with a three-support path, one leaf at each
support.  A terminal extension attaches a new ordinary component through a
terminal support and ends its other three stubs at three new supports.  It
increases $q$ by one and preserves the grammar.

\begin{theorem}[Slack forty]
\label{thm:slack-forty}
Let $T$ be a subcubic tree of order at least two.  Then $\Phi(T)=40$ if and
only if exactly one of the following holds.
\begin{enumerate}[label=\textup{(\roman*)}]
 \item $T$ is nonreduced and every clean quotient has slack nineteen with
       exact unit shifts in $\gamma$ and $\mu$.
 \item $T$ is a forty-loaded equality lift, with $q\ge19$.
 \item $T$ is obtained from a slack-twenty residue tree by twenty legal
       loads, with $q\ge9$.
 \item $T$ is a nineteen-loaded E0 assembly with $q\ge9$, or a
       nineteen-loaded E1 assembly with $q\ge8$.
 \item $T$ is a double-residue incidence tree.
\end{enumerate}
The parameters of families \textup{(ii)}, \textup{(iii)}--\textup{(iv)},
and \textup{(v)} are respectively
\[
\begin{array}{c|c|c|c}
&|T|&\gamma(T)&\mu(T)\\ \hline
\textup{(ii)}&9q+42&4q+1&3q+1\\
\textup{(iii)}\text{--}\textup{(iv)}&9q+24&4q+2&3q+2\\
\textup{(v)}&9q+6&4q+3&3q+3.
\end{array}
\]
\end{theorem}

\begin{proof}
For a nonreduced tree,
$\Phi(T)=\Phi(T')+21+63g$ forces $g=0$ and $\Phi(T')=19$, giving
\textup{(i)}.  Suppose that $T$ is reduced.  The algebraic row
$(s,t,r,a+b)=(0,1,0,19)$ is empty by
\Cref{prop:uniform-one-negative-gap}.  The $s=0$ row is the pressure-zero
matched-skeleton family with $a=40$, whose arm capacity is $q\ge19$.  The
$s=1$ row is the $k=20$ case of
\Cref{prop:one-residue-continuation}; its capacity conditions give \textup{(iii)} and
\textup{(iv)} with exactly the stated minima.

It remains to consider $s=2,t=r=0,a=b=0$.  Here
\[
 p=3q+3,\qquad h=q,\qquad d=3q,
\]
and packing equality holds in every positive-semidefinite residual block.
Suppression gives
\[
 Q_C\cong I_{\delta_C}\oplus(I-A(F_C)),
\]
where $F_C$ is a connected subcubic tree with a perfect matching.
Positive semidefiniteness implies $\rho(A(F_C))\le1$.  A connected tree
with an edge has adjacency index at least one, with equality only for $P_2$.
Thus every component is $\mathcal O$ and has pressure zero.  The global
identity $\sum_C\pi_C=2-e_{PP}$ gives $e_{PP}=2$, which is precisely the
double-residue grammar.

Conversely that grammar has $q$ ordinary components, $3q+3$ supports, and
two support--support edges.  The domination and inertia decompositions give
$\gamma=4q+3$, $\mu=3q+3$, and $|T|=9q+6$, hence $\Phi=40$.
Reducedness, the residue $s$, and $e_{PP}$ separate the five branches.
\end{proof}

\section{Pressure-two and one-negative phases at slack forty-one}
\label{sec:slack-forty-one}

Two new local atoms appear.  We isolate them before stating the global
classification.

\begin{lemma}[The slack-forty-one local atoms]
\label{lem:s41-atoms}
Suppose a weighted subcubic deep component $C$ satisfies
\[
 |C|=3\delta_C+1,\qquad \pi_C=2,\qquad Q_C\succeq0.
\]
Then $C$ is exactly one of
\begin{equation}
 G_0=\text{a weight-three singleton},\qquad
 G_1=\text{a claw with weights }(0;2,2,2).          \label{eq:G-atoms}
\end{equation}
Their $(|C|,\delta_C,B_C,\pi_C)$ tuples are $(1,0,3,2)$ and
$(4,1,6,2)$.

If instead $C$ has packing equality and is the unique one-negative component
in the row $s=1,t=1,a=b=r=0$, then it is the unique atom $S^\star$ obtained
by subdividing the three edges of the perfect matching of $S_{1,2,2}$,
assigning weight zero to the new vertices and weight
$3-\deg_{S_{1,2,2}}(v)$ to the old vertices.  It satisfies
\begin{equation}
 |S^\star|=9,\quad \delta=3,\quad B=8,\quad
 \pi=-2,\quad n_-(Q)=1.                            \label{eq:Sstar-data}
\end{equation}
\end{lemma}

\begin{proof}
For the positive-semidefinite branch, the degree identity
$m_C=(|C|-3\delta_C)+1-\pi_C$ gives $m_C=0$.  Hence every vertex has total
degree three and $Q_C=2I-A(C)$.  Smith's classification of connected
graphs with adjacency index at most two, restricted by subcubicity and the
order equation $|C|=3\delta_C+1$ \cite{Smith1970}, leaves only $A_1$ and $D_4=K_{1,3}$,
with the weights in \eqref{eq:G-atoms}.

For the one-negative branch, suppression gives
\[
 Q_C\cong I_{\delta_C}\oplus(I-A(F_C)),
 \qquad \pi_C=1-\delta_C,
\]
where $F_C$ is a subcubic perfect-matching tree of order $2\delta_C$.
The global pressure is $-2-e_{PP}$, so $\delta_C\ge3+e_{PP}$.  If
$\delta_C\ge4$, an induced $P_6$ creates at least two negative directions.
Thus $e_{PP}=0$ and $\delta_C=3$.  The order-six perfect-matching trees are
$P_6$ and $S_{1,2,2}$; the first has two adjacency eigenvalues greater than
one and the second exactly one.  Reversing suppression gives the stated
$S^\star$ and its data.
\end{proof}

We next name the global grammars.  All supports carry one leaf unless a load
is specified.
\begin{enumerate}[label=\textup{(G\arabic*)}]
 \item A one-loaded double-residue tree has one legal load on a family from
       \Cref{sec:slack-forty}.
 \item A pressure-two assembly uses one $G_0$ and $q$ ordinary components,
       or one $G_1$ and $q-1$ ordinary components, together with $3q+3$
       supports, no support--support edge, and a contracted incidence tree.
 \item A mixed E assembly uses E0 with $q$ ordinary components, or E1 with
       $q-1$ ordinary components, together with $3q+3$ supports, exactly one
       support--support edge, and a contracted incidence tree; every deficient
       exceptional stub meets an incidence-two support.
 \item A spectral assembly uses one $S^\star$, $q-2$ ordinary components,
       $3q+2$ supports, no support--support edge, and a contracted incidence
       tree.
\end{enumerate}

\begin{theorem}[Slack forty-one]
\label{thm:slack-forty-one}
A subcubic tree $T$ has $\Phi(T)=41$ if and only if exactly one of the
following holds.
\begin{enumerate}[label=\textup{(\roman*)}]
 \item $T$ is a clean expansion of a slack-twenty quotient with exact unit
       shifts.
 \item $T$ is a forty-one-loaded equality lift, $q\ge20$.
 \item $T$ is a twenty-one-loaded slack-twenty residue tree, $q\ge10$.
 \item $T$ is a twenty-loaded E0 or E1 assembly, in either case $q\ge9$.
 \item $T$ is one of \textup{(G1)}--\textup{(G3)}: one-loaded
       double-residue and $G_0$ start at $q=0$; $G_1$, mixed E0, and mixed
       E1 start at $q=1$.
 \item $T$ is a spectral $S^\star$ assembly, starting at $q=2$.
\end{enumerate}
Families \textup{(ii)}, \textup{(iii)}--\textup{(iv)}, and
\textup{(v)}--\textup{(vi)} have parameters
\[
\begin{array}{c|c|c|c}
&|T|&\gamma(T)&\mu(T)\\ \hline
\textup{(ii)}&9q+43&4q+1&3q+1\\
\textup{(iii)}\text{--}\textup{(iv)}&9q+25&4q+2&3q+2\\
\textup{(v)}\text{--}\textup{(vi)}&9q+7&4q+3&3q+3.
\end{array}
\]
\end{theorem}

\begin{proof}
The clean-expansion branch follows from
$\Phi(T)=\Phi(T')+21+63g$.  For a reduced tree the five-term identity has
exactly the following rows:
\[
\begin{array}{c|c|c|c}
s&t&r&a+b\\ \hline
0&0&0&41\\
1&0&0&21\\
2&0&0&1\\
0&1&0&20\\
1&1&0&0.
\end{array}
\]
The first row follows from pressure rigidity, and the second from
\Cref{prop:one-residue-continuation}; the fourth is empty by
\Cref{prop:uniform-one-negative-gap}.  Thus no algebraic row is omitted.

For $s=2,t=0,a+b=1$, the pressure identity is
$\sum_C\pi_C=2-e_{PP}$.  If $(a,b)=(1,0)$, every component is ordinary,
$e_{PP}=2$, and the leaf excess is exactly one legal load on a double-residue
tree.  If $(a,b)=(0,1)$, one surplus-one component $E$ carries pressure
$2-e_{PP}$.  Pressure zero violates the total-degree lower bound; pressure
one gives E0/E1 and $e_{PP}=1$; pressure two gives $G_0/G_1$ and
$e_{PP}=0$ by \Cref{lem:s41-atoms}.  These are precisely grammars
\textup{(G1)}--\textup{(G3)}.

For $s=1,t=1,a=b=0$, \Cref{lem:s41-atoms} gives the unique $S^\star$
component and $e_{PP}=0$; stub counting gives \textup{(G4)}.  The contracted
incidence tree supplies the converses, while terminal extension proves
 existence above the listed bases.  Indeed, the E0 and E1 mixed forests have
 cover sum $q$ and respectively $q+1$ and $q$ contracted component vertices.
 Their boundary-stub counts are $4q+2$ and $4q+1$.  After the single
 support--support edge is added, these are exactly the edge counts required
 for a tree on respectively $4q+4$ and $4q+3$ contracted vertices.  The
 $q=1$ E1 base uses its weight-one
 stub at one endpoint of the support--support edge; the other endpoint is
deep-free, correcting the tempting but unnecessary bound $q\ge2$.
 The canonical partition, residue coordinates, pressure, and local atom type
separate the branches.

For completeness, the minimum parameters and converses can be checked
without enumeration.  The loaded equality row needs $2q+2\ge41$, hence
$q\ge20$.  The residue row needs $2q+2\ge21$, hence $q\ge10$; in the E row
\eqref{eq:E-capacities} with $j=20$ gives $q\ge9$ for both E0 and E1.  In
the new $s=2$ row, the once-loaded double-residue and $G_0$ stars exist at
$q=0$, while $G_1$ and both mixed E types first exist at $q=1$; the explicit
mixed-E1 base is described above.  The $S^\star$ grammar has $q-2$ ordinary
components and therefore starts at $q=2$.  Terminal extension adds one
ordinary component, three supports, and three leaves, hence changes
$(|T|,\gamma,\mu)$ by $(9,4,3)$ in every nonisolated chain.

For the $G$ and mixed-E rows, $p=3q+3$, $d=3q+1$, $h=q$, and $t=0$.
For the $S^\star$ row, $p=3q+2$, $d=3q+3$, $h=q+1$, and $t=1$.
Both calculations give $(|T|,\gamma,\mu)=(9q+7,4q+3,3q+3)$.
Together with the inherited parameter formulae, this directly verifies the
converse and the displayed value $\Phi=41$ branch by branch.
\end{proof}

\section{The phase atlas at slack forty-two}
\label{sec:slack-forty-two}

Slack forty-two is the first layer at which two units of local surplus, one
negative direction with surplus, and two negative directions are all
algebraically visible.  The last alternative will be eliminated by
\Cref{cor:no-two-negative-s42}; the first two produce the following new
atoms.

\begin{lemma}[The slack-forty-two local atoms]
\label{lem:s42-atoms}
Suppose $C$ is a reduced weighted subcubic component.

If
\[
 |C|=3\delta_C+2,\qquad \pi_C=2,\qquad Q_C\succeq0,
\]
then $C$ is one of
\[
\begin{aligned}
 H_0&=P_2\text{ with weights }(2,1),\\
 H_1&=\text{the five-vertex broom with weights }(0,1,1,2,2).
\end{aligned}
\]
Their $(|C|,\delta_C,B_C,\In(Q_C))$ data are
\begin{equation}
 (2,0,3,(2,0,0)),\qquad (5,1,6,(4,0,1)).          \label{eq:H-data}
\end{equation}
For an unambiguous labelled representative of $H_1$, take edges
$01,03,04,12$ and weights $(b_0,\ldots,b_4)=(0,1,1,2,2)$; vertex $2$ is
the weight-one terminal and the unique stub-bearing total-degree-two vertex.

If
\[
 |C|=3\delta_C+1,\qquad \pi_C=-2,\qquad n_-(Q_C)=1,
\]
then $C$ is one of three weighted types $K_0,K_1,K_2$.  With vertices
numbered from zero, canonical representatives are
\begin{equation}
\resizebox{\textwidth}{!}{$
\begin{array}{c|l|l|c}
&E(C)&(b_0,\ldots,b_{|C|-1})&
(|C|,\delta_C,B_C,\In(Q_C))\\ \hline
K_0&01,04,12,23,45,46&(0,1,0,2,0,1,1)&(7,2,5,(5,1,1))\\
K_1&01,04,07,12,23,45,56,78,79&(0,1,0,2,1,0,2,0,1,1)&
(10,3,8,(7,1,2))\\
K_2&01,04,07,12,23,45,56,78,89&(0,0,1,1,1,0,2,1,0,2)&
(10,3,8,(7,1,2)).
\end{array}$}                                      \label{eq:K-data}
\end{equation}
The numbers of stub-bearing total-degree-two vertices are two, two, and one,
respectively.
\end{lemma}

\begin{proof}
Let $m_C$ count total-degree-two vertices.  The degree identity
\begin{equation}
 m_C=(|C|-3\delta_C)+1-\pi_C                    \label{eq:s42-degree-identity}
\end{equation}
first gives $m_C=1$.  If $x$ is the unique such vertex, then
$Q_C=2I-A(C)-e_xe_x^{\trans}$.  Attach two new leaves to $x$ to obtain
$G$.  Eliminating those leaf coordinates from $2I-A(G)$ has Schur
complement $Q_C$, so $Q_C\succeq0$ exactly when $\rho(A(G))\le2$.
The marked Smith classification \cite{Smith1970} leaves a two-vertex path with weights
$(2,1)$ and a five-vertex broom with weights $(0,1,1,2,2)$; all path,
$D$, $E$, and affine alternatives fail the marked position or the order
equation.  This proves \eqref{eq:H-data}.

In the one-negative case, \eqref{eq:s42-degree-identity} gives $m_C=4$;
these four vertices are independent by reducedness.  A maximum
distance-three packing yields an all-orders bound.  If every packed target
has degree two, their closed neighbourhoods leave one vertex $w$ outside.
Let $\varepsilon\in\{0,1\}$ indicate whether $w$ has total degree three.
Suppression gives
$m_C=\delta_C+1-\varepsilon$, so $m_C=4$ implies
$\delta_C\in\{3,4\}$ and $|C|\le13$.  If one packed target has degree
three, split at it into three rooted branches.  With
$\varepsilon_i\in\{0,1\}$ denoting the corresponding root corrections, the
branch deficit identity
$m_C=\sum_i(\delta_i+1-\varepsilon_i)=4$ gives
$\sum_i\delta_i=1+\sum_i\varepsilon_i\le4$.  The central packed target
contributes one more to the cover, so
$\delta_C=1+\sum_i\delta_i\le5$ and $|C|\le16$.
Thus enumeration through order sixteen is exhaustive.  Canonically
enumerating every unlabelled subcubic tree and every independent four-set,
then applying exact domination and rational inertia tests, leaves precisely
the three representatives in \eqref{eq:K-data}.  The displayed weights also
give the deficient-stub counts directly.
\end{proof}

\begin{figure}[t]
  \centering
  \resizebox{\textwidth}{!}{\PtwoAtomGalleryGraphic}
  \caption{The new local atoms at slacks forty-one and forty-two.  Numeric
  labels are the boundary weights $b_v$.  Blue atoms are
  positive-semidefinite; red atoms have exactly one negative direction.
  The drawings use the canonical representatives in
  \eqref{eq:G-atoms}, \eqref{eq:Sstar-data}, \eqref{eq:H-data}, and
  \eqref{eq:K-data}.}
  \label{fig:phase-atoms}
\end{figure}

The remaining incidence grammars are as follows.  All supports carry one
leaf, have core degree one or two, and there are no support--support edges.
After every deep component is contracted, the component--support incidence
graph must be a tree, and every deficient stub must meet a core-degree-two
connector support.
\begin{center}
\small
\begin{tabular}{@{}lcc@{}}
\toprule
Exceptional profile & ordinary components & supports\\
\midrule
$H_0$ & $q$ & $3q+3$\\
$H_1$ & $q-1$ & $3q+3$\\
E0+E0 & $q$ & $3q+3$\\
E0+E1 & $q-1$ & $3q+3$\\
E1+E1 & $q-2$ & $3q+3$\\
$K_0$ & $q-1$ & $3q+2$\\
$K_1$ or $K_2$ & $q-2$ & $3q+2$\\
\bottomrule
\end{tabular}
\end{center}
An $H_i$ or $K_i$ \emph{incidence assembly} is a tree satisfying its row.
A \emph{two-exception assembly} satisfies the corresponding E0/E1 row.

\begin{theorem}[Slack forty-two]
\label{thm:slack-forty-two}
A subcubic tree $T$ has $\Phi(T)=42$ if and only if exactly one of the
following holds.
\begin{enumerate}[label=\textup{(\roman*)}]
 \item $T$ is a clean expansion of a slack-twenty-one quotient with exact
       unit shifts.
 \item $T$ is a forty-two-loaded equality lift, $q\ge20$.
 \item $T$ is a twenty-two-loaded slack-twenty residue tree, $q\ge10$.
 \item $T$ is a twenty-one-loaded E0 assembly with $q\ge10$, or E1 with
       $q\ge9$.
 \item $T$ has two legal loads on distinct terminal supports of a
       double-residue tree.  There is an isolated $q=0$ base and a separate
       $q=1$ base extending to every $q\ge1$.
 \item $T$ has one legal load on a pressure-two $G_0/G_1$ family.
       Loaded $G_0$ starts at $q=0$ and loaded $G_1$ at $q=1$.
 \item For $q\ge1$, $T$ has one legal load on a mixed E0/E1 family; the load
       may absorb a deficient stub.  There is also one isolated $q=0$
       loaded-E0 base, obtained by loading the unique deficient preassembly
       so that the resulting tree, though not the preassembly, is reduced.
       Loaded E1 starts at $q=1$.
 \item $T$ is an $H_0$ or $H_1$ incidence assembly, starting at $q=1$ and
       $q=2$, respectively.
 \item $T$ is a two-exception assembly of type E0+E0, E0+E1, or E1+E1;
       all start at $q=2$.
 \item $T$ is a $K_0,K_1$, or $K_2$ spectral incidence assembly, starting
       at $q=3,4,3$, respectively.
 \item $T$ has one legal load on an $S^\star$ spectral assembly, starting
       at $q=2$.
\end{enumerate}
The putative $s=0,t=2$ branch is empty.  Families \textup{(ii)},
\textup{(iii)}--\textup{(iv)}, and \textup{(v)}--\textup{(xi)} have
parameters
\[
\begin{array}{c|c|c|c}
&|T|&\gamma(T)&\mu(T)\\ \hline
\textup{(ii)}&9q+44&4q+1&3q+1\\
\textup{(iii)}\text{--}\textup{(iv)}&9q+26&4q+2&3q+2\\
\textup{(v)}\text{--}\textup{(xi)}&9q+8&4q+3&3q+3.
\end{array}
\]
\end{theorem}

\begin{proof}
The clean-expansion claim follows from the contraction identity.  For a
reduced tree the five-term identity has exactly the six rows
\[
\begin{array}{c|c|c|c}
s&t&r&a+b\\ \hline
0&0&0&42\\
1&0&0&22\\
2&0&0&2\\
0&1&0&21\\
1&1&0&1\\
0&2&0&0.
\end{array}
\]
The $s=0,t=0$ row follows from equality-lift pressure rigidity, and the
$s=1,t=0$ row from \Cref{prop:one-residue-continuation}; their arm and
loading capacities give \textup{(ii)}--\textup{(iv)}.  The $s=0,t=1$ row is empty by
\Cref{prop:uniform-one-negative-gap}, so no coordinate row is omitted.

Consider $s=2,t=0,a+b=2$.  Every residual block is positive semidefinite and
\begin{equation}
 \sum_C\pi_C=2-e_{PP}.                              \label{eq:s42-psd-pressure}
\end{equation}
If $(a,b)=(2,0)$, every component is ordinary and $e_{PP}=2$; the two leaf
units are precisely the loads in \textup{(v)}.  If $(a,b)=(1,1)$, the
unique surplus-one component is $G_0/G_1$ when $e_{PP}=0$ or E0/E1 when
$e_{PP}=1$, and the remaining unit is one legal load.  This gives
\textup{(vi)}--\textup{(vii)}.  The $q=0$ E0 output is the separately named
isolated base: its load restores reducedness, so it is not described as a
legal load preserving reducedness of an already reduced mixed assembly.
If $(a,b)=(0,2)$, either two surplus-one components occur or a single
surplus-two component occurs.  In the former case each pressure is one,
$e_{PP}=0$, and the atoms are E0/E1.  In the latter, positive
semidefiniteness gives
\[
 0\le\boldsymbol1^{\trans}Q_C\boldsymbol1
   =B_C-|C|=\pi_C-1,
\]
so $\pi_C\ge1$.  Equality would give $Q_C\boldsymbol1=0$, hence all boundary
weights one; then $R_C=\varnothing$, $\delta_C=0$, and $|C|=2$, leaving the
weighted path $(1,1)$ whose two total-degree-two vertices are adjacent.
Reducedness excludes it.  Therefore $\pi_C\ge2$, and
\eqref{eq:s42-psd-pressure} forces $\pi_C=2$ and $e_{PP}=0$.
\Cref{lem:s42-atoms} now gives $H_0/H_1$.  Stub and connector counts give
the minima in \textup{(viii)}--\textup{(ix)}.

Next consider $s=1,t=1,a+b=1$.  Its global pressure is $-2-e_{PP}$.  If
$(a,b)=(1,0)$, the unique negative component is $S^\star$, $e_{PP}=0$, and
the remaining unit is one legal load, giving \textup{(xi)}.  If
$(a,b)=(0,1)$, suppose an equality-size negative component $N$ and a
positive-semidefinite surplus-one component $E$ were distinct.  Suppression
of $N$ gives a perfect-matching tree; if $N$ has exactly one negative
direction, \Cref{lem:low-nu-perfect-matching} gives $\delta_N\le3$, hence
$\pi_N=1-\delta_N\ge-2$.  The surplus-one positive-semidefinite block has
$\pi_E\ge1$.  Thus the total pressure is at least $-1$, contradicting the
required value $-2-e_{PP}\le-2$.  Hence the unique negative component
carries the surplus.  The near-packing bound forces $e_{PP}=0$, and
\Cref{lem:s42-atoms} gives $K_0,K_1,K_2$.  Their two, two, and one
deficient stubs give the respective minima $q=3,4,3$.

Finally the $s=0,t=2,a=b=0$ row is empty by
\Cref{cor:no-two-negative-s42}.  Every listed grammar has an explicit
minimum base.  Except for the isolated $q=0$ twice-loaded double-residue
tree, terminal extension changes $(|T|,\gamma,\mu)$ by $(9,4,3)$ and proves
existence at every larger permitted $q$.  Direct reconstruction proves all
converses.  The canonical partition, five-term row, pressure, and local atom
type separate the branches.

Here are the connector calculations underlying every new minimum.  For
$H_0,H_1$, the connector counts $B-p$ are $q,q-1$; their deficient vertices
therefore force $q\ge1,2$.  For E0+E0, E0+E1, and E1+E1, the connector
counts are $q+1,q,q-1$.  After suppressing each connector, these form a tree
on the deep components.  The two E0 atoms require degree at least two at
both marked component vertices; E0+E1 requires marked degrees at least two
and one; E1+E1 requires two nonisolated marked vertices.  Each condition is
equivalent to $q\ge2$.  For $K_0,K_1,K_2$, the connector counts are
$q-1,q-2,q-2$, and the marked exceptional component needs respectively two,
two, and one incident connectors, giving $q\ge3,4,3$.

The remaining minima are direct capacities: $2q+2\ge42$ for the loaded
equality row; $2q+2\ge22$ for the loaded residue row;
\eqref{eq:E-capacities} with $j=21$ for loaded E0/E1; and nonnegative
ordinary-component counts for loaded $G_0/G_1$, loaded mixed E0/E1, and
loaded $S^\star$.  These give exactly $20$, $10$, $10/9$, $0/1$, $0/1$,
and $2$.  The minimum incidence trees are obtained by taking a tree on the
component vertices with the marked degrees just listed, subdividing its
edges by connector supports, and terminating every remaining stub at a
distinct support.  This explicitly constructs each base; terminal extension
constructs every larger $q$.

Finally, for families \textup{(v)}--\textup{(xi)}, the decompositions give
either $(p,h,t)=(3q+3,q,0)$ or $(3q+2,q+1,1)$, and in both cases
$(|T|,\gamma,\mu)=(9q+8,4q+3,3q+3)$.  Together with the inherited parameter
formulae, this verifies every converse and displayed parameter row directly.
\end{proof}

\begin{remark}[End of the explicit atlas]
The finite-kernel theorem continues for every fixed $k$, but the present
paper makes the kernel lists human-readable only through $k=42$.  No claim
about an explicit $\Phi>42$ classification is made.
\end{remark}

\section{Exact computation and its evidentiary role}
\label{sec:computation}

The unbounded global reductions are symbolic.  Computation served as a
falsification test, a bounded structural reconciliation, and---only after a
symbolic order or diameter bound---an exhaustive local certificate.  No
claim is extrapolated from an unbounded small-order window.

\begin{figure}[t]
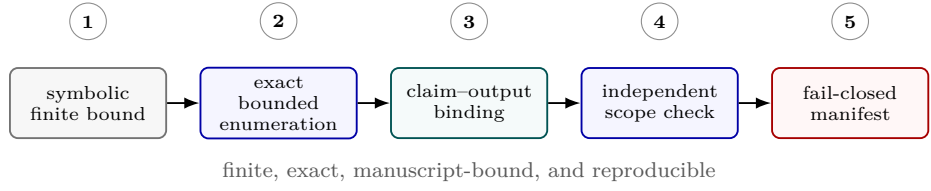

  \centering
  \resizebox{\textwidth}{!}{\PtwoCertificateWorkflowGraphic}
  \caption{Workflow used whenever computation carries evidentiary weight.
  The symbolic proof first bounds the search domain.  Every finite route binds
  recorded outputs to manuscript claims; independent checks are asserted only
  within each route's stated scope and need not reconstruct the primary
  generator.  The combined manifest freezes the candidate certificate package.  No
  all-orders claim is extrapolated from a small-order census.}
  \label{fig:certificate-workflow}
\end{figure}

\subsection{Exact parameter evaluators}

All domination numbers were computed by an exact three-state rooted-tree
dynamic program.  For a vertex $v$, the states $A_v,B_v,C_v$ respectively
mean selected, dominated by a selected child, and awaiting domination from
the parent.  For children $w$ of $v$,
\begin{align*}
 A_v&=1+\sum_w\min\{A_w,B_w,C_w\},\\
 C_v&=\sum_w B_w,\\
 B_v&=\sum_w\min\{A_w,B_w\}
      +\min_w\bigl(A_w-\min\{A_w,B_w\}\bigr),
\end{align*}
where $B_v=+\infty$ for a leaf, and the root value is
$\min\{A_r,B_r\}$.  Here $+\infty$ is an infeasibility sentinel for the
state ``dominated by a selected child''; it is never returned as a finite
parameter value.  The computation is integral.

The value of $\mu=n_-(L(T)-I)$ was computed within the standard tree
eigenvalue-location framework of Jacobs and Trevisan and its perturbed-
Laplacian extension \cite{JacobsTrevisan2011,BragaRodrigues2017}.  Our
implementation is an independent exact-rational realization of that
congruence framework.  Starting with diagonal $d_v=\deg_T(v)-1$, a nonzero leaf pivot
updates its neighbour by $d_u\leftarrow d_u-1/d_v$; a zero pivot is removed
with its neighbour as a nonsingular $2\times2$ block, contributing one
positive and one negative direction.  Sylvester's law makes the process
exact, with no floating-point threshold.

Selected boundary instances were recomputed from the full rational matrix by
dense symmetric elimination with independent pivots; order-thirteen
domination was also checked by exhaustive subset search.  These paths differ
in representation and control flow from the main evaluators.

\subsection{Isomorphism control and checked windows}

Unlabelled trees came from a complete nonisomorphic stream.  An independent
canonical code roots at the centre, sorts child codes recursively, and sorts
the two rooted halves in the bicentred case.  Stream and lift generators were
compared through these codes rather than generator labels.

\begin{table}[t]
\centering
\caption{Bounded exact controls for the main symbolic statements.  The
counts describe the checked windows only.}
\label{tab:exact-controls}
\small
\begin{tabularx}{\textwidth}{@{}>{\raggedright\arraybackslash}Xrr>{\raggedright\arraybackslash}X@{}}
\toprule
Target & All trees & Subcubic trees & Exact conclusion checked\\
\midrule
Global order law, orders $2$--$14$
 & 5,446 & -- & No violation; equality classes agree with the symbolic interface.\\
Stability, orders $5$--$16$
 & 32,503 & 4,638 & Every class with $0\le\Phi\le19$ equals an independently generated loaded class.\\
Complete $\Phi=20$ boundary, orders $5$--$16$
 & 32,503 & 4,638 & No class before order $13$; the unique order-$13$ class is the generated terminal tail.\\
Complete $\Phi=21$ layer, orders $5$--$16$
 & 32,503 & 4,638 & Nine classes: five clean expansions and four reduced residue-branch classes; no spectral-excess class.\\
\bottomrule
\end{tabularx}
\end{table}

The generated-family control checks all $4,688$ loadings from every subcubic
perfect-matching skeleton through eight vertices; path skeletons realize all
slacks through nineteen.  It also checks $66$ terminal tails, $17$ bridge
splits, $25$ killed-middle-arm controls, the first twenty-load instance at
order $103$, and defect-class counts $1,3,8,28$ for $q=1,2,3,4$.  Brute-force
domination and dense rational inertia agree at order thirteen and on the
first bridge split at order twenty-two.

An independent route exhausts $263{,}457$ labelled Pruefer sequences at
skeleton orders two, four, six, and eight and returns the same digests and
counts $1,3,8,28$ using exhaustive domination and exact characteristic-
polynomial Sturm counts.  Separate controls recompute the nine slack-
twenty-one edge lists and enumerate $42{,}509$ colourings of $7{,}741$
order-fifteen cores and $55{,}335$ of $19{,}320$ order-sixteen cores, finding
exactly E0 and E1 and no spectral-excess class.  Combination-defect controls
recover both terminal-tail locations, the bridge-split--plus--arm family, and
the $\{0,1\}$ versus $\{1,1\}$ signatures.

The bounded controls in \Cref{tab:exact-controls} were development-time
falsification and reconciliation checks.  The public archive does not contain
a producer contract for every aggregate count in that table, so those counts
are not invoked as stand-alone reproducible certificates.  For the
theorem-facing routes described below, frozen inputs give byte-identical
reruns and a fail-closed SHA-256 manifest binds the exact evaluators,
dependencies, reference outputs, negative tests, and recorded commands.

\subsection{Finite-defect and phase-atlas certificates}

All theorem-specific certificates use the same exact primitives, with
independent interfaces tailored to suppression and incidence grammars.  The local
compactness implementation checked all $49{,}516$ primary suppression states
and $793$ independently reconstructed symbolic-matrix states.  A separate
clean-contraction audit checked order, domination, inertia, slack, and
excess-degree interfaces for all $511$ contractions among $206$ nonreduced
subcubic trees in its bounded window.

The primary near-packing route exhaustively tested $3{,}600{,}047$
admissible marked states through component order sixteen; $561$ satisfy the
$m_C\ge5$ premise and none has exactly one negative direction.  A separate
Sturm route cross-checked the finite boundary table, while the rooted catalog
records the six canonical representatives; neither auxiliary route is
claimed to repeat the full $3{,}600{,}047$-state enumeration.  The
two-surplus interface in
\Cref{lem:two-surplus-gap} was tested by two implementations on $1{,}028{,}044$
unlabelled residual trees and $63{,}066$ legal marked states through residual
order twenty, with identical histograms and no counterexample.

For the phase atlas, the pressure-two $G$-atom search examined $1{,}346{,}024$
unlabelled component trees through order twenty and retained exactly $G_0$
and $G_1$.  The $K$-atom search is exhaustive because the near-packing proof
first bounds the component order by sixteen.  The primary route exhaustively
retains exactly $K_0,K_1,K_2$; the independent checker recomputes target
covers and exact inertias for those output records but does not regenerate
the marked search space.  Likewise, the low-$\nu_1$ computation
is exhaustive because the induced-$P_{10}$ argument first bounds diameter by
eight: the centre generator produces $50{,}248$ subcubic types, $279$ with
perfect matchings, and exactly $1,2,11$ types at indices $0,1,2$.

\begin{table}[t]
\centering
\caption{Selected high-risk claim--certificate routes.  The scope attached to
each ID distinguishes exhaustive finite certificates from bounded interfaces
and record-level cross-checks.}
\label{tab:certificate-map}
\small
\begin{tabularx}{\textwidth}{@{}>{\raggedright\arraybackslash}p{.29\textwidth}>{\raggedright\arraybackslash}p{.34\textwidth}X@{}}
\toprule
Claim & Symbolic finite boundary & Certificate IDs\\
\midrule
Local compactness, \Cref{thm:local-compactness}
 & Moore bound \eqref{eq:local-moore}
 & LC-P, LC-I\\
Clean contraction and the quotient interface in \Cref{thm:fixed-slack-kernels}
 & At most $\lfloor k/21\rfloor$ contractions and
   $\chi\le\lfloor k/21\rfloor$
 & CC (interface)\\
Near-packing gap, \Cref{lem:near-packing-gap}
 & Component order at most $16$
 & NP, NP-B, NP-S, NP-R\\
Two-surplus gap, \Cref{lem:two-surplus-gap}
 & Residual order at most $20$
 & TS-L, TS-P, TS-I\\
$K$ atoms in the slack-$42$ atlas
 & Component order at most $16$
 & KA-P, KA-I\\
Low-$\nu_1$ list, \Cref{lem:low-nu-perfect-matching}
 & Diameter at most $8$
 & LN-P, LN-I\\
\bottomrule
\end{tabularx}
\end{table}

The certificate IDs expand as follows (scope; result file; run record):
\begin{description}[style=nextline,leftmargin=2.7em,labelwidth=2.2em]
\footnotesize
\item[LC-P] bounded primary interface;
  \path{local-compactness-suppression-through-order-12.json};
  \path{fs_primary_suppression.json}.
\item[LC-I] independent bounded reconstruction;
  \path{local-compactness-independent-through-order-8.json};
  \path{fs_independent_compactness.json}.
\item[CC] bounded interface;
  \path{clean-contraction-interface-through-order-12.json};
  \path{fs_clean_contraction.json}.
\item[NP] exhaustive after the symbolic order-$16$ bound;
  \path{near-packing-local-through-order-16.json};
  \path{s22_near_packing_local.json}.
\item[NP-B] finite boundary-table record;
  \path{near-packing-boundary-table.json};
  \path{s22_boundary_table.json}.
\item[NP-S] independent Sturm cross-check of that table;
  \path{near-packing-boundary-sturm.json};
  \path{s22_boundary_sturm.json}.
\item[NP-R] six-representative rooted catalog;
  \path{near-packing-rooted-type-catalog.json};
  \path{s22_rooted_catalog.json}.
\item[TS-L] bounded local interface through order $17$;
  \path{two-surplus-local-through-order-17.json};
  \path{s23_two_surplus_local.json}.
\item[TS-P] exhaustive primary residual search through order $20$;\par
  \path{two-surplus-defect-core-through-residual-order-20.json};\par
  \path{s23_defect_core_primary.json}.
\item[TS-I] independent state-loop, matching-DP, and leaf-congruence
  recomputation of the finite critical $m=6$ residual table through order
  $20$;\par
  \path{two-surplus-defect-core-independent-through-residual-order-20.json};\par
  \path{s23_defect_core_independent.json}.
\item[KA-P] exhaustive after the symbolic order-$16$ bound;
  \path{new-local-atoms-through-order-17.json};
  \path{s42_local_atoms_primary.json}.
\item[KA-I] independent target-cover and exact-inertia recomputation of the
  five primary output atoms, not an independent generator;
  \path{local-atoms-independent.json};
  \path{s42_local_atoms_independent.json}.
\item[LN-P] exhaustive after the symbolic diameter-$8$ bound;
  \path{low-nu-perfect-matching-all-orders.json};
  \path{s42_low_nu_primary.json}.
\item[LN-I] independent cross-check through order $20$;
  \path{low-nu-crosscheck.json};
  \path{s42_low_nu_crosscheck.json}.
\end{description}

The machine-readable companion to this selected map is
\path{CERTIFICATE_ROUTE_MAP_v1_1.json}.  These selected routes do not exhaust
the archive's aggregate claim groups;
the corridor and slack-$40$--$42$ record, grammar, and construction controls
remain listed in the machine-readable claim ledger.  Full relative paths and
SHA-256 values of the published computational artifacts are frozen in
\path{CLAIM_CHECKS.json}; exact commands, environments, and output digests
are recorded in the corresponding files under \path{run_records/}.  The DOI
archive binds the public manuscript snapshot current at deposition; the
present review-absorbed source is separately hash-bound in the submission
packages.  The archived package is available at\linebreak
\href{https://doi.org/10.5281/zenodo.22119899}{doi:10.5281/zenodo.22119899}.

The complete-tree controls agree with the symbolic grammars: the slack-forty
window retains four double-residue records; the slack-forty-one census through
order seventeen recognizes all $26$ records; and the slack-forty-two census
contains $81{,}136$ unlabelled trees, $9{,}740$ subcubic trees, and $95$
slack-forty-two records, all $95$ recognized and independently recomputed.
Construction controls pass on all $43$ slack-forty-two generated instances.
These counts are evidence for implementation fidelity; completeness of the
local $K$ and low-$\nu_1$ lists comes from the preceding symbolic finite
reductions.

\section{Conclusion and open problems}
\label{sec:conclusion}

The relation between domination and small Laplacian eigenvalues in trees is
governed by an integer slack rather than only by an asymptotic ratio.  The
five-term decomposition of that slack proves the sharp order law and the
exact fixed-order maximum of $7\gamma-9\mu$.  In the subcubic class, the
same coordinates combine with leaf-bundle congruence, exact partial
domination, and local pressure to recover a canonical matched-skeleton
representation and all layers through slack twenty-one.

The central structural conclusion is stronger: every fixed-slack subcubic
layer has finitely many local defect types.  Suppressing an optimal packing
turns a weighted deep component into an almost-perfect-matching tree; its
residual negative index bounds induced-path length and hence order.  The
ordinary $(2,0,2)$ tile is the sole unbounded zero-cost component.  The
global pressure identity then confines all remaining deep and support defects
to a bounded ported kernel, while clean contraction contributes only a
bounded number of reverse expansions.

The explicit atlas through $\Phi=42$ shows how this general theorem becomes
concrete.  The interval $20$--$39$ is a uniform loading corridor.  Slack
forty introduces double-residue incidence trees; forty-one introduces the
first pressure-two and one-negative atoms; forty-two introduces the next
positive-semidefinite and one-negative atoms, while the apparent
two-negative equality branch is empty.  The finite computations used for
the $K$ atoms and low-$\nu_1$ trees occur only after symbolic diameter or
order bounds make the search exhaustive.

The scope remains deliberately sharp.  The explicit human-readable atlas
stops at forty-two, although the finite-kernel theorem holds for every fixed
slack.  The numerical compactness bounds are not intended as a practical
recognition algorithm.  The argument also uses subcubicity essentially;
bounded excess degree is admitted only at the finitely many quotients created
by clean contraction and does not yield an unrestricted $\Delta\ge4$
classification.  For genuinely higher maximum degree, the boundary weights
and core degrees introduce larger diagonal perturbations after suppression,
zero-cost branching is no longer controlled by the same degree-three
grammar, and the bounded-degree Moore step no longer supplies a uniform
order bound.  A higher-degree theory would therefore need a new local
pressure or a replacement compactness parameter, not merely a longer case
analysis.

Three concrete directions remain.  First, replace the coarse Moore bounds by
a canonical minimum kernel and an efficient recognition algorithm.  More
precisely, for fixed $k$, is there an algorithm of running time
$f(k)|T|^{O(1)}$ which, given a subcubic tree with $\Phi(T)\le k$, outputs a
canonical minimum ported kernel together with a verifiable decomposition
certificate?  Second,
extract explicit phase grammars beyond forty-two without losing the uniform
ported-kernel viewpoint.  Third, determine whether a different local
pressure or suppression theory can handle genuinely higher maximum degree.
Together these questions ask how much of a sharp spectral--domination
inequality can be converted into an effective structural theory.

\section*{Statements and Declarations}

\paragraph{Author contributions.}
Yufeng Wang is the sole author and was responsible for the study conception,
methodology, formal analysis, software, validation, visualization, drafting,
revision, and final approval of the manuscript.

\paragraph{Funding.}
The author received no funding for this work.

\paragraph{Competing interests.}
The author declares no competing interests.

\paragraph{Acknowledgements.}
None.

\paragraph{Data and code availability.}
This theoretical study used no external datasets.  The current manuscript
source is versioned and hash-bound.  A consolidated reproducibility
supplement contains the exact-verification code, recorded commands, and
reference outputs and is frozen under a fail-closed SHA-256 manifest.  The
privacy-sanitized public supplement is permanently archived on Zenodo at
\href{https://doi.org/10.5281/zenodo.22119899}{doi:10.5281/zenodo.22119899}.
The author selected file-level Option A licensing: MIT for Python
code and CC BY 4.0 for author-owned documentation, results, certificates,
logs, and metadata.

\paragraph{Declaration of generative AI and AI-assisted technologies in the
manuscript preparation process.}
During the research and preparation of this manuscript, the author used
OpenAI Codex and ChatGPT to assist with literature searches, exploration of
candidate constructions and proof strategies, development and checking of
exact-computation code, drafting and language editing, translation, and
refinement of TikZ figure code.  The author reviewed and verified all
mathematical claims, proofs, computations, citations, and figures,
independently reconstructed the main proof chain, and takes full responsibility
for the content.  No AI system is listed as an author.

\bibliographystyle{abbrv}
\bibliography{references}

\end{document}